\documentclass[oneside]{amsart} %

\usepackage[percent]{overpic}

\usepackage{paralist}
\usepackage{graphics} %
\usepackage{epsfig} %
\usepackage{xspace}
\usepackage{marginnote} %

\usepackage{array}

\usepackage{appendix}
\usepackage[english]{babel} 
\usepackage[T1]{fontenc} 
\usepackage[utf8]{inputenc} 
\usepackage[margin=1.1in]{geometry}
\usepackage{graphicx}
\usepackage{tabto}
\usepackage{amsmath}
\usepackage{calligra}
\usepackage{amsthm}
\usepackage{textcomp}
\usepackage{upgreek}
\usepackage{placeins}
\usepackage{pgfplotstable}
\usepackage{bbm}
\usepackage[colorinlistoftodos,prependcaption]{todonotes}
\usepackage{cancel}
\usepackage{mathrsfs}
\usepackage{mathtools}
\usepackage{version}
\usepackage{arydshln}
\usepackage[math]{cellspace}

\usepackage{multicol}

\usepackage{bbding}
\usepackage{hyperref}
\hypersetup{
    colorlinks=true,
    linkcolor=blue,
    filecolor=magenta,      
    urlcolor=cyan,
    pdftitle={Overleaf Example},
    pdfpagemode=FullScreen,
    }

\newcommand{\rvline}{\hspace*{-\arraycolsep}\vline\hspace*{-\arraycolsep}}
\newtheorem{theorem}{Theorem}
\newtheorem{remark}{Remark}
\newtheorem{problem}{Problem}
\newtheorem{lemma}{Lemma}
\newtheorem{proposition}{Proposition}
\usepackage{booktabs}
\usepackage{pgfplots}
\pgfplotsset{compat=1.18}
\usepackage{xcolor}
\usepackage{colortbl}
\usepackage{tikz}
\usepackage{amssymb}%
\usepackage{pifont}%

\usepackage{bm}
\usepackage{subcaption}

\usepackage{tikz}
\usepackage{pgfplots}
\pgfplotsset{compat=1.18}

\usepackage{amssymb}

\usepackage{hyperref}
\makeatletter
\let\amsart@setauthors\@setauthors
\def\@setauthors{%
  \amsart@setauthors
  \ifx\@empty\addresses\else
    \@setaddresses
    \global\let\addresses\@empty
  \fi
}
\makeatother

\begin{document}

\title[]{Augmented Lagrangian preconditioners for fictitious domain formulations of elliptic interface problems}

\author{Michele Benzi}
\address{Scuola Normale Superiore, Piazza dei Cavalieri, 7, 56126 Pisa, Italy}
\email{michele.benzi@sns.it}

\author{Marco Feder}
\address{Department of Mathematics, University of Pisa, Largo B. Pontecorvo, 5, Pisa, 56127, Italy}
\email{marco.feder@dm.unipi.it}

\author{Luca Heltai}
\address{Department of Mathematics, University of Pisa, Largo B. Pontecorvo, 5, Pisa, 56127, Italy}
\email{luca.heltai@unipi.it}

\author{Federica Mugnaioni}
\address{Scuola Normale Superiore, Piazza dei Cavalieri, 7, 56126 Pisa, Italy}
\email{federica.mugnaioni@sns.it}

\maketitle

\begin{abstract}
  We present a novel augmented Lagrangian (AL) preconditioner for the solution of linear systems arising from finite element discretizations of elliptic interface problems with jump coefficients. The method is based on the Fictitious Domain with Distributed Lagrange Multipliers formulation and it is designed to improve the convergence
  of the Flexible Generalized Minimal Residual (FGMRES) method in the presence of large coefficient jumps. To reduce the computational cost, we also introduce a
  cheaper block-triangular variant of the preconditioner. We prove eigenvalue clustering for the ideal AL preconditioner and study the limiting behavior of the spectrum for the modified variant in terms of parameters and the size of the jumps. Numerical experiments on different immersed geometries confirm mesh-independent iteration counts and robustness over large coefficient jumps, with substantial reductions in wall-clock time for the modified approach.

  \vspace{0.4cm}

\end{abstract}

\section{Introduction}
In this work, we consider the efficient numerical solution of elliptic interface problems with discontinuous coefficients across an internal interface.
Problems of this type arise in many scientific and engineering applications, including biosciences and fluid--structure interaction (FSI), where fluid and solid regions exhibit distinct physical properties across the interface.

A classical strategy to handle interface problems is to use \textit{fitted approaches}, which involve
generating a mesh that conforms to the interface between different materials or regions and provide
accurate solutions~\cite{barrett-interface,hou-interface,huynh-interface,mu-interface,chen-interface}. In FSI, a prominent fitted strategy is the Arbitrary Lagrangian--Eulerian (ALE) approach~\cite{donea1982arbitrary,hirt1974arbitrary,improved-ale}, where compatibility between fluid and solid kinematics can be enforced by construction. In settings with large interface motion or deformation, however, maintaining mesh quality can require mesh-motion strategies and, in some cases, remeshing, which may increase computational cost, particularly in three-dimensional simulations.
Motivated by these considerations, a range of \textit{unfitted methods} have been developed over
the past decades, where the computational mesh does not need to conform to the interface and the interface is permitted to cut through elements. Representative examples include the Immersed Boundary Method by Peskin~\cite{Peskin_2002}, where the
interface position is tracked by regularized Dirac delta distributions, and
the level set method~\cite{SUSSMAN1994146,chang1996level}, where the interface corresponds to the
zero level set of a certain function. Other techniques, such as Nitsche-XFEM~\cite{alauzet2016nitsche}, Cut-FEM~\cite{Burman_Hansbo_Larson_Zahedi_2025,CutFEM,cutStokes,HANSBO201490}, and Finite Cell methods~\cite{finitecell1,finitecell2,finitecell3} typically enforce interface/boundary conditions weakly via Nitsche-type couplings and/or stabilization terms. An alternative strategy is offered by the Fictitious Domain approach, spanning applications ranging from particulate flow simulations~\cite{glowinski2001fictitious,glowinski1997lagrange}
to more general cases considered in later works~\cite{fatboundary,yu2005dlm}.
Building on the idea of fictitiously extending one domain into the other, the method introduced in~\cite{BoffiCavalliniGastaldi2015} proposes a new Fictitious Domain formulation with Distributed Lagrange Multipliers (FD-DLM) to address elliptic interface and FSI problems. The method employs independent meshes (one for the extended domain, and one for the immersed domain) that are generated only once and remain fixed throughout the computation. To enforce that the solutions coincide in both domains, a suitable coupling term is introduced, implemented through a distributed Lagrange multiplier.

Elliptic interface and FSI problems
share the common feature of involving a domain split into two regions separated by an interface. In the former case, this separation is due to the presence of discontinuous coefficients, whereas in the latter it reflects the presence of two different equations, corresponding to constitutive models for the fluid and the solid.

Focusing on the elliptic interface case, the analysis of a fictitious domain-like approach using continuous finite elements was first introduced in~\cite{AURICCHIO201536}, and more
recently extended to discontinuous elements in~\cite{BoffiGastaldi2024}. Independently of
the chosen discretization, the method leads to a saddle point problem with a three-by-three block structure of this form:
\begin{equation}\label{eqn:intro}
  \begin{bmatrix}
    \mathsf{A} & 0             & \mathsf{C^T}    \\
    \mathsf{0} & \mathsf{A_2}  & \mathsf{-C_2^T} \\
    \mathsf{C} & \mathsf{-C_2} & 0
  \end{bmatrix}
  \begin{bmatrix}
    \mathsf{u}   \\
    \mathsf{u}_2 \\
    \mathsf{\lambda}
  \end{bmatrix}=
  \begin{bmatrix}
    \mathsf{ f} \\
    \mathsf{g}  \\
    \mathsf{0}
  \end{bmatrix},
\end{equation}
where $\mathsf{A} \in \mathbb{R}^{n \times n}$ is symmetric positive definite (SPD), $\mathsf{A_2} \in \mathbb{R}^{m \times m}$ is symmetric positive \emph{semidefinite} (SPSD), $\mathsf{C} \in \mathbb{R}^{\ell \times n}$, and $\mathsf{C_2} \in \mathbb{R}^{\ell \times m}$.

This three-by-three block structure also appears, although with different meaning in terms of operators and unknowns, in a variety of application problems. In the context of PDE-constrained optimization and optimal control problems, suitable preconditioning strategies have been proposed in~\cite{MBconstrained,schoberl,SimonciniROM,STOLL2013498,PearsonRegularization}.

The system matrix in~\eqref{eqn:intro} can be regarded as a two-by-two block matrix:
\begin{equation}\label{eqn:intro_rid}
  \mathcal{A}=\begin{bmatrix}
    \widetilde{\mathsf{A}} & \mathsf{B^T} \\
    \mathsf{B}             & 0
  \end{bmatrix}, \quad \text{where} \quad \widetilde{\mathsf{A}}= \begin{bmatrix}
    {\mathsf{A}} & 0            \\
    0            & \mathsf{A_2}
  \end{bmatrix}\quad \text{and} \quad \mathsf{B} =\begin{bmatrix}
    \mathsf{C} & -\mathsf{C_2}
  \end{bmatrix},
\end{equation}
indicating that the problem is, in principle, amenable to a \emph{standard} saddle point problem. For large discretizations, particularly in three dimensions, memory and factorization costs often make fully direct solution strategies impractical, motivating iterative Krylov methods with robust preconditioning. In addition, the resulting linear systems can become challenging for Krylov iterative methods due to coefficient contrast, the semidefinite sub-block associated with the immersed domain operator, and the multi-mesh coupling. These features can lead to severe ill-conditioning and, consequently, to a deterioration of the convergence rate. Therefore, designing a suitable and robust preconditioner is crucial to accelerate the convergence of iterative methods.

If $\widetilde{\mathsf{A}}$ were invertible, classical block preconditioners such as those surveyed in~\cite{Benzi2005} could be considered. However, these approaches rely on the Schur complement $\widetilde{\mathsf{S}}=\mathsf{B}\widetilde{\mathsf{A}}^{-1}\mathsf{B^T}$, and thus cannot be applied directly to~\eqref{eqn:intro_rid}, where $\widetilde{\mathsf{A}}$ is singular due to the semidefiniteness of $\mathsf{A_2}$. Moreover, strategies requiring an approximation of the Schur complement $\mathsf{S=CA^{-1}C^T}$ (such as the ones in~\cite{SimonciniROM}) are especially challenging for unfitted methods; difficulties arise because the
matrix $\mathsf{C}$ in~\eqref{eqn:intro} involves the product of basis functions defined on two arbitrarily overlapping grids~\cite{boffi2023comparison,BOFFI2022115650,KrauseZulian}, making it unclear to which matrix $\mathsf{S}$ should be spectrally equivalent.

Preconditioners for fictitious domain formulations of elliptic interface problems
have already been investigated, e.g. in~\cite{WangDLMFD_prec}, where an augmented Lagrangian (AL) Uzawa
iterative method is presented. Reported results indicate robustness with respect to coefficient contrast, but iteration counts and parameter sensitivity remain non-trivial in some regimes. Later, in
the context of FSI simulations, a block preconditioner was studied in~\cite{BOFFI2024406} using sparse
direct solvers for the inversion of diagonal blocks. Recently, an iterative version based on multigrid methodologies has been tested
in~\cite{alshehri2025multigridpreconditioningfddlmmethod} for the elliptic interface case.
For the fictitious domain approach with boundary-supported Lagrange multipliers applied to Poisson and Stokes problems, a novel augmented Lagrangian preconditioner was recently proposed and analyzed by the authors in~\cite{BENZI2026118522}. Originally developed for finite element discretizations of the Oseen problem arising from Picard linearizations of the
steady Navier-Stokes equations~\cite{ALprec,Farrell2019}, augmented Lagrangian preconditioners effectively bypass the need for a good approximation of the dense Schur complement matrix. In the context of elliptic interface problems, an additional important benefit arises: the augmentation removes the singularity from the $(2,2)$-block in~\eqref{eqn:intro}. Motivated by these considerations, in this work we extend the results and preconditioning techniques presented in~\cite{BENZI2026118522} to handle
linear systems of equations arising from the fictitious domain formulation with distributed Lagrange multipliers for elliptic interface problems. The resulting \emph{ideal} AL preconditioner admits a spectral analysis showing favorable eigenvalue clustering and mesh-independent bounds; numerically we observe robust iteration counts over wide coefficient contrasts.

The term \emph{ideal} refers to a theoretical formulation of the preconditioner where the action of inverses is assumed to be computed exactly. This is, however, impractical, due to its high computational cost. The main cost arises from the solution of linear systems with the augmented term, making efficient inexact solvers essential. A tailored geometric multigrid cycle used as an approximate solver for
the velocity subproblem associated with the $(1,1)$-block of the preconditioner in the AL formulation of the Oseen problem was
initially proposed in~\cite{ALprec}, and later extended to three dimensions in~\cite{Farrell2019}, inspired by the
multigrid framework presented in~\cite{schoeberl_1999}. While this strategy yields optimal results in terms of iteration counts as well as robustness with respect
to a wide range of Reynolds numbers, its implementation can be challenging for general discretizations and geometries.
For this reason, a so-called \emph{modified} AL preconditioner was introduced in~\cite{modALprec} for the Oseen
problem, whose block-triangular structure makes it easier to implement. In particular, off-the-shelf
algebraic multigrid solvers and preconditioners for scalar elliptic PDEs can be employed to solve the subsystems arising
during the application of the preconditioner, providing an efficient approach for inverting the diagonal blocks. A spectral
analysis of this modified AL-based preconditioner was presented in~\cite{spectral_analysis_modified}, giving a recipe for the choice of the augmentation parameter $\gamma$. Since the augmented operator in our formulation couples
two overlapping grids, constructing a monolithic geometric multigrid solver is challenging; we therefore adopt a modified AL strategy that admits efficient sub-solves using standard Algebraic MultiGrid (AMG) components.

Summarizing, the main contributions of our work are the following. First, we develop a novel augmented Lagrangian preconditioner for the fictitious domain (FD-DLM) formulation of elliptic interface problems, which demonstrates robustness and low iteration counts across a wide range of problem regimes. Second, we introduce a computationally efficient modified variant of this preconditioner that leverages standard algebraic multigrid solvers for the diagonal blocks, significantly reducing implementation complexity and overall time-to-solution. Third, we provide a rigorous spectral analysis, demonstrating favorable eigenvalue clustering and robustness with respect to large coefficient jumps across the interface. Finally, we present extensive numerical validation of the proposed preconditioners, including challenging three-dimensional problems in linear elasticity with heterogeneous material properties.

In this work, robustness refers to the fact that the outer iteration counts and, whenever inexact inner solvers are employed, the total number of inner iterations are essentially insensitive to mesh refinement and to large coefficient jumps across subdomains. The latter are monitored in the numerical experiments and are shown to exhibit only a very mild growth.

We point out that one important reason for studying efficient preconditioners for the elliptic interface problem
with a Lagrange multiplier formulation lies in the fact that its block structure naturally arises as a sub-block of
a four-by-four block system obtained when using a monolithic formulation for the FSI problem with Lagrange multipliers, where additional blocks of such a system are
related to the incompressibility constraint of the fluid. As proven in~\cite{CAUSIN20054506}, monolithic schemes for FSI
do not suffer from the so-called \emph{added mass effect} when the densities of fluid and solid bodies are comparable, as it is the case in the context of biological tissues.
The conditioning of the system arising from FD-DLM applied to the monolithic FSI problem was recently
analyzed in~\cite{boffi2025stability}.

The paper is organized as follows. In Section~\ref{sec:method}, we present the FD-DLM formulation for a generic elliptic interface problem, describe the finite element discretization, recall existence and uniqueness results, and derive the resulting saddle-point system. In Section~\ref{sec:AL_prec}, we derive
the ideal augmented Lagrangian preconditioner for our model problem, while Section~\ref{sec:spectral_ideal} is
devoted to its spectral analysis. In Section~\ref{sec:modified_prec}, we introduce
a modified variant of the ideal AL preconditioner, which is easier to implement and more efficient in practice. Its spectral
analysis is presented in Section~\ref{sec:spectral_modified}. Section~\ref{sec:numerical_experiments} presents
several numerical experiments to validate the proposed preconditioners, demonstrating their robustness and effectiveness
for various problem configurations, including cases with large coefficient jumps. We conclude our experiments by applying the preconditioner to a three-dimensional linear elasticity problem with heterogeneous material properties characterized by different Lamé constants in each subdomain. Finally, in Section~\ref{sec:conclusions}, we summarize our main results and
identify avenues for future research.

\section{Elliptic interface problem and fictitious domain approach}\label{sec:method}

\subsection{Notation}\label{subsec:notation}We start by fixing some notation. Given an open and bounded domain $D$, we denote with $L^2(D)$ the space of square integrable functions on $D$, endowed with the norm $\|\cdot\|_{0,D}$ induced by the inner
product $(\cdot,\cdot)_{D}$. Sobolev spaces are denoted by $W^{s,p}(D)$, where $s\in \mathbb{R}$ refers to the differentiability index and $p \in [1,\infty]$ is the integrability exponent. In the case $p=2$, we use the notation $H^s(D)=W^{s,2}(D)$, with
associated norm $\|\cdot\|_{s,D}$ and seminorm $|\cdot|_{s,D}$. Given the space $H^1(D)$, we denote by $[H^{1}(D)]^{*}$ its dual space, endowed with the
dual norm
\[
  \|\eta\|_{[H^{1}(D)]^{*}} = \sup_{v\in H^1(D)}
  \frac{\langle \eta,v\rangle}{\|v\|_{1,D}},
\]
where $\langle\cdot,\cdot\rangle$ denotes the duality pairing between $[H^{1}(D)]^{*}$ and $H^1(D)$.
We use normal font, e.g. $A$, to denote linear operators, while matrices and vectors are
denoted by the sans serif font, e.g. $\mathsf{A}$ and $\mathsf{x}$, respectively. A calligraphic font, e.g. $\mathcal{A}$, is
used to denote block matrices associated with saddle point systems. The spectrum of a matrix is denoted by $\operatorname{Spec}(\cdot)$. The symbol $\lambda$ is reserved for Lagrange multipliers and, in Section~\ref{sec:numerical_experiments}, for Lamé parameters. We will use the letter $\nu$ to denote a generic eigenvalue of a matrix.

\subsection{Elliptic interface problem}Let $\Omega$ be a domain in $\mathbb{R}^d$, where $d \in \{2,3\}$, with a bounded Lipschitz boundary $\partial \Omega$.
Let $\Omega_1$ and $\Omega_2$ be two subdomains of $\Omega$ such that $\overline{\Omega} = \overline{\Omega}_1 \cup \overline{\Omega}_2$, and let the interface $\Gamma = \overline{\Omega}_1 \cap \overline{\Omega}_2$ be Lipschitz continuous.
We call $\Omega$ the \emph{background domain} and $\Omega_2$ the \emph{immersed domain}, which we assume to be entirely contained in $\Omega$, i.e. $\overline{\Gamma} \cap \partial \Omega = \emptyset$.

We consider the following \emph{elliptic interface problem} with a jump in the coefficients.
\begin{problem}\label{prob:elliptic_interface}
For $i=1,2$, given forcing terms $f_i \colon \Omega_i \rightarrow \mathbb{R}$, and positive coefficients $\beta_i \in L^{\infty}(\Omega_i)$, find $u_1 \colon \Omega_1 \rightarrow \mathbb{R}$ and $u_2 \colon \Omega_2 \rightarrow \mathbb{R}$ such that:
\begin{equation}
  \begin{cases}\label{eqn:model_problem}
    -\operatorname{div} \bigl( \beta_i \nabla u_i \bigr) = f_i
    \>\> \>\>\>\>\>\> \>\>\>\>\>\> \>\>\>\>\>\> \>\>\>\>\>\> \mathrm{in} \> \Omega_i,                                            \\
    \>\> \>\>\>\>\>\>\>\>\>\>\>\>\>\>\>\>\>\>\>\>u_1 = u_2 \>\> \quad\>\>\>\>\>\>\>\>\>\>\>\>\>\>\>\>\>\> \mathrm{on} \> \Gamma, \\
    \>\>\>\>\>\>\beta_1 \nabla u_1 \cdot \mathbf{n}_1 = -\beta_2 \nabla u_2 \cdot \mathbf{n}_2 \>\>\>\>\mathrm{on}\> \Gamma,     \\
    \qquad \qquad  \> \>\>\>u_1 = 0  \>\>\>\>\>\>\>\>\>\>\>\>\>\>\>\>\>\>\>\>\>\> \quad \mathrm{on} \> \partial \Omega_1.
  \end{cases}
\end{equation}
\end{problem}

In the above problem, $\mathbf{n}_i$ ($i=1,2$) denotes the unit vector normal to $\Gamma$, pointing out of $\Omega_i$. The two transmission conditions on $\Gamma$ enforce the continuity of $u_1$ and $u_2$, as well as the continuity of the co-normal derivatives on the interface. In 2D, this model describes the
displacement of a membrane made of two different materials. We assume the coefficients to be bounded by positive constants $\Bar{\beta_1}$ and $\Bar{\beta_2}$, i.e., $\beta_1 > \Bar{\beta_1} > 0$ and $\beta_2 > \Bar{\beta_2} > 0$. Such coefficients $\beta_i$ may represent the stiffness of such materials, $f_i$ the loads
applied to the membrane, and $u_i$ the vertical displacement in $\Omega_i$, respectively. Notably, the condition $u_1=u_2$ on $\Gamma$ implies that the materials are perfectly bonded. For the sake of simplicity, we will assume $\beta_1$ and $\beta_2$ to be positive constants throughout the paper, although the
proposed approach can be extended to variable coefficients without significant modifications.

In this paper, we consider a fictitious domain reformulation of Problem~\eqref{prob:elliptic_interface}, following~\cite{AURICCHIO201536}. In detail, we extend $\beta_1$ and $f_1$, originally defined in $\Omega_1$, to the whole $\Omega$. We denote
the corresponding extensions with $\beta$ and $f$, respectively, so that $\beta_{|\Omega_1} = \beta_1$ and $f_{|\Omega_1} = f_1$. Similarly, $u$ will denote the extension of $u_1$ to $\Omega$, satisfying $u_{|\Omega_1} = u_1$. The extended solution $u$ is then required to match $u_2$ in the immersed domain $\Omega_2$. A graphical
representation of such procedure is shown in Figure~\ref{fig:domain_mesh}.

\begin{figure}[htbp]
  \centering
  \begin{tabular}{c}
    \includegraphics[width=0.69\textwidth]{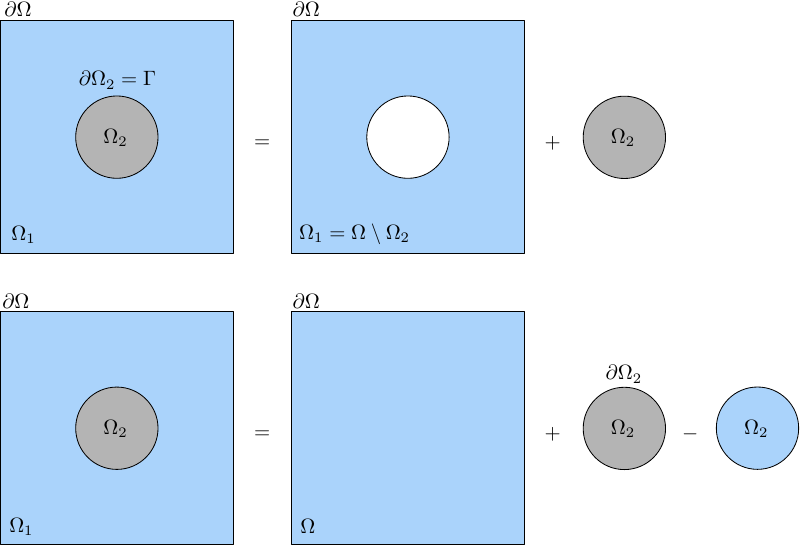}
  \end{tabular}
  \caption{Top row: the initial configuration for Problem~\eqref{prob:elliptic_interface}, involving the domains $\Omega_1$ and $\Omega_2$. Bottom row: the fictitious domain reformulation, where the immersed
    domain $\Omega_2$ is superimposed on $\Omega$. The solution $u$ is defined in the whole $\Omega$, while $u_2$ is defined only in $\Omega_2$. The fictitious contribution is then subtracted at the variational level \emph{on the domain $\Omega_2$}.}
  \label{fig:domain_mesh}
\end{figure}

The condition $u_{|\Omega_2}=u_2$ is enforced at the variational level by introducing the functional space $\Lambda$ and a bilinear form $c \colon \Lambda \times H^1(\Omega_2) \rightarrow \mathbb{R}$, satisfying
\[c(\mu,v_2) = 0 \>\>  \forall \mu \in \Lambda \Longrightarrow v_2 = 0 \>\> \text{in} \>\> \Omega_2.\]
In this work, we consider $\Lambda = [H^{1}(\Omega_2)]^{*}$, the dual space of $H^1(\Omega_2)$, and $c(\mu,v_2) = \langle\mu, v_2\rangle$, the duality pairing between $\Lambda=[H^{1}(\Omega_2)]^{*}$ and $H^1(\Omega_2)$, but other definitions for $\Lambda$ and hence for $c(\cdot,\cdot)$ are possible~\cite{Alshehri_2025}. With this choice, the norm $\|\cdot\|_\Lambda$ is the dual norm introduced in Section~\ref{subsec:notation}. Finally, the condition $u_{|\Omega_2} = u_2$ is imposed by introducing a Lagrange multiplier $\lambda \in \Lambda$ as unknown in the following weak formulation of Problem~\eqref{prob:elliptic_interface}.
\begin{problem}
\label{prob:weak_form}
Given $f \in L^2(\Omega)$ and $f_2 \in L^2(\Omega_2)$, $\beta \in L^\infty(\Omega)$ and $\beta_2 \in L^\infty(\Omega_2)$, find $(u,u_2,\lambda) \in V \times V_2 \times \Lambda$ such that
\begin{eqnarray}
  \label{eqn:LM1}
  (\beta \nabla u, \nabla v)_{\Omega} + c(\lambda, v_{|\Omega_2}) &=& (f,v)_{\Omega} \>\>\>\>\>\>\>\>\>\>\>\>\>\>\>\>\>\>\>\quad \forall v \in V \coloneqq H_0^1(\Omega), \\
  \label{eqn:LM2}
  \bigl((\beta_2 - \beta) \nabla u_2, \nabla v_2\bigr)_{\Omega_2}  - c(\lambda, v_2)&=& (f_2 -f ,v_2)_{\Omega_2} \qquad\> \forall v_2 \in V_2 \coloneqq H^1(\Omega_2), \\
  \label{eqn:LM3}
  c(\mu ,u_{|\Omega_2}-u_2) &=& 0 \qquad \>\>\>\>\>\>\>\>\>\>\>\>\>\>\>\>\qquad \forall \mu \in \Lambda \coloneqq [H^{1}(\Omega_2)]^{*}.
\end{eqnarray}
\end{problem}

Problem~\eqref{prob:weak_form} is the standard weak formulation with fictitious domain approach of the original interface Problem~\eqref{prob:elliptic_interface} (details can be found in~\cite{AURICCHIO201536}). In operator matrix form, the weak formulation can be written as follows:
\begin{equation} \label{eq:opform}
  \begin{pmatrix}
    A & 0    & \rvline & C^T    \\
    0 & A_2  & \rvline & -C_2^T \\
    C & -C_2 & \rvline & 0
  \end{pmatrix}
  \begin{pmatrix}
    u   \\
    u_2 \\
    \lambda
  \end{pmatrix}
  =
  \begin{pmatrix}
    F \\
    G \\
    0
  \end{pmatrix},
\end{equation}
where $A$ and $A_2$ are the operators associated with bilinear forms $(\beta \nabla u, \nabla v)_{\Omega}$
and $\Bigl((\beta_2 - \beta) \nabla u_2,\nabla v_2\Bigr)_{\Omega_2}$, respectively, while $(C,-C_2)$ is the operator pair
associated with $c(\mu,u_{|\Omega_2} - u_2)$, which has kernel $$\mathbb{K} = \{(v,v_2) \in V \times V_2: c(\mu,v_{|\Omega_2} - v_2) = 0 \quad \forall \mu \in \Lambda\}.$$
We now collect some basic facts about existence and uniqueness for this saddle point problem.
\begin{proposition}[\cite{AURICCHIO201536}, Prop. 1]
  Given $f \in L^2(\Omega)$ and $f_2 \in L^2(\Omega_2)$, Problem~\eqref{prob:weak_form} has a unique solution $(u,u_2,\lambda) \in V \times V_2 \times \Lambda$ such that
  the following estimate holds:
  \begin{equation}
    \label{eqn:elliptic_estimate}
    |u|_{1,\Omega} + \|u_2\|_{1,\Omega_2} + \|\lambda\|_{\Lambda} \leq C \left( \|f\|_{0,\Omega} + \|f_2\|_{0,\Omega_2} \right),
  \end{equation}
  where $C$ is a positive constant.
\end{proposition}
The well-posedness of the saddle point Problem~\eqref{prob:weak_form} follows from the fulfillment of the following sufficient conditions~\cite{boffi2013mixed}. Proofs can be found in~\cite{BoffiGastaldi2024,AURICCHIO201536}.
\begin{itemize}
  \item \textbf{Ellipticity on the kernel.} There exists a constant $\hat{\alpha} > 0$ such that
        \begin{equation}
          \label{eqn:elliptic_kernel}
          (\beta \nabla v, \nabla v)_{\Omega} + \bigl((\beta_2 - \beta) \nabla v_2, \nabla v_2\bigr)_{\Omega_2} \geq \hat{\alpha}  (|v|_{1,\Omega}^2 + \|v_2\|_{1,\Omega_2}^2) \qquad \forall (v,v_2) \in \mathbb{K}.
        \end{equation}
  \item \textbf{Inf-sup condition.} There exists a constant $\hat{\theta} > 0$ such that
        \begin{equation}
          \label{eqn:inf_sup}
          \sup_{(v,v_2) \in V \times V_2} \frac{c(\mu, v_{|\Omega_2} - v_2)}{\bigl(|v|_{1,\Omega}^2 + \|v_2\|_{1,\Omega_2}^2\bigr)^{\frac12}} \geq \hat{\theta} \|\mu\|_{\Lambda} \qquad \forall \mu \in \Lambda.
        \end{equation}
\end{itemize}

\vspace{0.1cm}
\subsection{Finite element discretization}\label{subsec:fem}

Problem~\eqref{prob:weak_form} is discretized by mixed finite elements. We employ two \emph{independent}, shape-regular mesh families for $\Omega$ and $\Omega_2$, which we denote by $\mathcal{T}_h$ and $\mathcal{T}_{2,h}$, respectively. Notably, it is customary to take the background domain as a structured quadrilateral (or hexahedral) discretization of a $d$-dimensional box. However, we require the immersed mesh to be quasi-uniform to exploit standard mass-matrix scaling and inverse inequalities during the spectral analysis. We denote by $h_{\Omega}$ the mesh size of $\mathcal{T}_h$, and by $h_{\Omega_2}$ the mesh size
of $\mathcal{T}_{2,h}$. As recently shown in~\cite{boffi2025stability} for fictitious domain discretizations of fluid-structure interaction problems with distributed Lagrange multipliers, the condition number of the algebraic system depends explicitly on both mesh sizes $h_{\Omega}$ and $h_{\Omega_2}$. The same mechanism is expected to apply in our setting, since it is independent of the specific bulk problem. In particular, the ratio $\frac{h_{\Omega}}{h_{\Omega_2}}$ should not be taken too extreme in either direction: a very fine immersed mesh inflates the condition number, while a very coarse one would under-resolve the immersed geometry and degrade the accuracy of the numerical solution. Given these considerations, in the numerical tests, we choose $h_{\Omega_2}$ and $h_{\Omega}$ to be comparable as a reasonable compromise, and will perform our numerical studies under simultaneous refinement of both grids.

Moreover, for a generic element $K \in \mathcal{T}_h$ or $\mathcal{T}_{2,h}$, we define $\mathcal{Q}^p(K)$, $p\geq1$, to be the space of polynomials defined on $K$ of degree at most $p$ in each variable.

The finite element spaces are defined as follows:  $V_h \subset H_0^1(\Omega)$, $V_{2,h}\subset H^1(\Omega_2)$, and $\Lambda_h \subset \Lambda$. At the discrete level, if $\Lambda_h \subset L^2(\Omega_2)$, the duality pairing between $\Lambda$ and its dual can be evaluated
using the usual scalar product in $L^2(\Omega_2)$. Hence,
\[c(\mu_h, v_{2,h}) = (\mu_h, v_{2,h})_{\Omega_2} \>\> \forall \mu_h \in \Lambda_h, \>\> \forall v_{2,h} \in V_{2,h}.\]

All in all, the discrete formulation of Problem~\eqref{prob:weak_form} reads as follows.
\begin{problem} \label{prob:discrete}
Given $f\in L^2(\Omega)$ and $f_2\in L^2(\Omega_2)$, find $(u_h,u_{2,h},\lambda_h)\in V_h\times V_{2,h}\times\Lambda_h$ such that
\begin{align}
  (\beta\nabla u_h,\nabla v_h)_\Omega+(\lambda_h,v_h|_{\Omega_2})_{\Omega_2}                & =(f,v_h)_\Omega              &  & \forall v_h\in V_h, \nonumber        \\
  ((\beta_2-\beta)\nabla u_{2,h},\nabla v_{2,h})_{\Omega_2}-(\lambda_h, v_{2,h})_{\Omega_2} & =(f_2-f, v_{2,h})_{\Omega_2} &  & \forall v_{2,h}\in V_{2,h},          \\
  (\mu_h,u_h|_{\Omega_2}-u_{2,h})_{\Omega_2}                                                & =0                           &  & \forall\mu_h\in\Lambda_h.  \nonumber
\end{align}
\end{problem}
In the following, we recall the discrete counterparts of conditions~\eqref{eqn:elliptic_kernel} and~\eqref{eqn:inf_sup}, which will be needed for the spectral analysis of the preconditioner. The ellipticity
on the kernel is satisfied when $\beta_2 > \beta> \bar{\beta}>0$, although the numerical findings reported in~\cite{BoffiGastaldi2024} suggest that
this assumption might be relaxed. We first need to introduce the discrete kernel of $c(\cdot, \cdot):$
$$\mathbb{K}_h = \{(v_h,v_{2,h}) \in V_h \times V_{2,h}: (\mu_h,v_{h|\Omega_2} - v_{2,h})_{\Omega_2} = 0 \quad \forall \mu_h \in \Lambda_h\}.$$ The next two propositions
ensure existence and uniqueness of the discrete solution to Problem~\eqref{prob:discrete}.  Thanks to these conditions, the theory of saddle point problems yields the usual quasi-optimal
error estimate (cfr.~\cite[Th. 5.2.2]{boffi2013mixed}). Proofs of Propositions~\ref{prop:elliptic_kernel_discrete} and~\ref{prop:inf_sup_discrete}, which depend on the choice of finite dimensional subspaces, can be found in~\cite{BoffiGastaldi2024,AURICCHIO201536}. Indeed, to
solve Problem~\eqref{prob:discrete}, a classical choice is to use continuous linear elements $\mathcal{Q}^1$ for all the variables:
\begin{align*}
  V_h       & = \{v_h \in H_0^1(\Omega): v_h|_K \in \mathcal{Q}^1(K) \text{ for all } K \in \mathcal{T}_h\} \subset H_0^1(\Omega),             \\
  V_{2,h}   & = \{v_{2,h} \in H^1(\Omega_2): v_{2,h}|_K \in \mathcal{Q}^1(K) \text{ for all } K \in \mathcal{T}_{2,h}\} \subset H^1(\Omega_2), \\
  \Lambda_h & = V_{2,h}.
\end{align*}
Other stable choices satisfying the discrete inf-sup condition are also possible, such as $\mathcal{Q}^2-\mathcal{Q}^2-\mathcal{Q}^0$ or $\mathcal{Q}^1-(\mathcal{Q}^1+\mathfrak{B})-\mathcal{Q}^0$, where
$\mathcal{Q}^1+\mathfrak{B}$ denotes the enrichment of $\mathcal{Q}^1$ elements with bubble functions defined on each element $K$ and vanishing on $\partial K$ (see Remark~\ref{rmk:spaces}).

\begin{proposition}[Discrete ellipticity on the kernel]
  \label{prop:elliptic_kernel_discrete}
  Let us consider $V_h, V_{2,h}$ and $\Lambda_h$ as defined above, and assume that $\beta_2 > \beta > \bar{\beta}>0$ in $\Omega_2$. Then, there exists a constant $\alpha > 0$, independent of the
  discretization parameters $h_{\Omega}$ and $h_{\Omega_2}$ such that for all $(v_h,v_{2,h}) \in \mathbb{K}_h$, the following inequality holds true.
  \begin{equation}
    \label{eqn:elliptic_kernel_discrete}
    (\beta \nabla v_h, \nabla v_h)_{\Omega} + \bigl((\beta_2 - \beta) \nabla v_{2,h}, \nabla v_{2,h}\bigr)_{\Omega_2} \geq \alpha (|v_h|_{1,\Omega}^2 + \|v_{2,h}\|_{1,\Omega_2}^2).
  \end{equation}
\end{proposition}

\vspace{0.05cm}
\begin{proposition}[Discrete inf-sup condition]
  \label{prop:inf_sup_discrete}
  Let $V_h, V_{2,h}$ and $\Lambda_h$ as defined above. Then, there exists a constant $\theta > 0$, independent of the discretization parameters $h_{\Omega}$ and $h_{\Omega_2}$, such that
  the following inf-sup condition holds:
  \begin{equation}
    \label{eqn:inf_sup_discrete}
    \sup_{(v_h,v_{2,h}) \in V_h \times V_{2,h}} \frac{(\mu_h, v_{h|\Omega_2} - v_{2,h})_{\Omega_2}}{\bigl(|v_h|_{1,\Omega}^2 + \|v_{2,h}\|_{1,\Omega_2}^2\bigr)^{\frac12}} \geq \theta \|\mu_h\|_{\Lambda} \qquad \forall \mu_h \in \Lambda_h.
  \end{equation}
\end{proposition}

We denote by $\{{\varphi}_i\}_{i=1}^n$, $\{\phi_j\}_{j=1}^m$, and $\{\psi_k\}_{k=1}^\ell$ the basis functions
of $V_h$, $V_{2,h}$ and $\Lambda_h$. The finite element discretization of Problem~\eqref{prob:discrete} yields the
following linear system of size $(n+m+\ell)\times (n+m+\ell)$ for $(\mathsf{u},\mathsf{u}_2,\mathsf{\lambda})$
\begin{equation}
  \begin{bmatrix}\label{eqn:matrix}
    \mathsf{A} & 0             & \mathsf{C^T}    \\
    \mathsf{0} & \mathsf{A_2}  & \mathsf{-C_2^T} \\
    \mathsf{C} & \mathsf{-C_2} & 0
  \end{bmatrix}
  \begin{bmatrix}
    \mathsf{u}   \\
    \mathsf{u}_2 \\
    \mathsf{\lambda}
  \end{bmatrix}=
  \begin{bmatrix}
    \mathsf{ f} \\
    \mathsf{g}  \\
    \mathsf{0}
  \end{bmatrix},
\end{equation}where
\begin{eqnarray}
  \begin{aligned}\label{eqn:matrices}
    (\mathsf{A})_{i,j}   & = \int_{\Omega}    \beta \nabla\varphi_i \cdot \nabla\varphi_j,     & \qquad
    (\mathsf{A_2})_{i,j} & = \int_{\Omega_2} (\beta_2 -\beta) \nabla\phi_i \cdot \nabla\phi_j,          \\
    (\mathsf{C})_{i,k}   & = (\psi_k, \varphi_i)_{\Omega_2}=\int_{\Omega_2} \psi_k \varphi_i,  & \qquad
    (\mathsf{C_2})_{k,j} & = (\phi_j,\psi_k)_{\Omega_2}=\int_{\Omega_2} \psi_k \phi_j,                  \\
    (\mathsf{f})_{i}     & = (f, \varphi_i)_{\Omega}=\int_{\Omega} f \varphi_i,                & \qquad
    (\mathsf{g})_{j}     & = (f_2 - f, \phi_j)_{\Omega_2}=\int_{\Omega_2} (f_2 - f) \phi_j.
  \end{aligned}
\end{eqnarray}
Notice that choosing $V_{2,h}\equiv\Lambda_h$ implies that the basis functions for the multiplier space $\Lambda_h$ coincide with those for $V_{2,h}$, that is, $\{\psi_k\}_{k=1}^m = \{\phi_j\}_{j=1}^m$. In this particular case,
the matrix $\mathsf{C_2}$ coincides with the mass matrix defined on the multiplier space $\Lambda_h$, whose generic entry is \begin{equation}\label{eqn:mass_matrix}
  (\mathsf{M})_{i,j} \coloneqq (\psi_i,\psi_j)_{\Omega_2}=\int_{\Omega_2} \psi_i \psi_j.
\end{equation}

\begin{remark}[Singularity of $\mathsf{A_2}$]
  \label{rmk:singularity}
  Due to the lack of boundary conditions on $\Omega_2$, the discretization with
  finite elements of the bilinear form $A_2$ will result in a singular
  matrix. More precisely, it coincides with the matrix of the pure Neumann problem defined
  on $\Omega_2$, for which the solution is unique up to an additive constant and hence $\ker(\mathsf{A_2}) = \operatorname{span}\{\mathsf{1}\}$. This does not imply
  that the resulting linear system of equations is singular, but poses significant challenges
  in the construction of a preconditioner.
\end{remark}
\begin{remark}[Other mixed discretizations]
  \label{rmk:spaces}
  Recent inf-sup stable discretizations based on discontinuous multipliers
  such as $\mathcal{Q}^2-\mathcal{Q}^2-\mathcal{Q}^0$ or $\mathcal{Q}^1-(\mathcal{Q}^1+\mathfrak{B})-\mathcal{Q}^0$
  lead to $\dim(V_{2,h})\ne \dim(\Lambda_h)$, and hence to a rectangular matrix $\mathsf{C_2}$. We stress, however, that the forthcoming spectral analysis will hold regardless of
  the specific choice of finite element spaces, as long as stability is guaranteed.

  For ease of implementation, we employ a square mass matrix in most of the numerical experiments. However, we also present results obtained using a discontinuous space for the multiplier. The influence
  of this choice on the performance of a multigrid preconditioner for this problem has been
  recently investigated in~\cite{alshehri2025multigridpreconditioningfddlmmethod}.
\end{remark}

We point out that the flexibility of employing independent meshes for $\Omega$ and $\Omega_2$ comes
at the price of a certain implementation effort. The assembly procedure for the interface matrix
$\mathsf{C}$ is far from trivial. Indeed, the
evaluation of $c(\cdot,\cdot)$ as a standard $L^2(\Omega_2)$ scalar product requires the integration of the product of two basis functions, $\psi_k$ and
$\varphi_i$, over the immersed mesh $\mathcal{T}_{2,h}$. However, $\psi_k$ is defined on the mesh $\mathcal{T}_{2,h}$, while $\varphi_i$ is defined
on the background mesh $\mathcal{T}_h$. The efficient numerical integration of such terms requires
the usage of suitable geometric search procedures and efficient collision detection algorithms, to identify pairs of
background and immersed elements that overlap. The interested reader can find detailed discussions
about this topic in~\cite{BoffiGastaldi2024,boffi2023comparison,BOFFI2022115650,boffi2024quadratureerrorestimatesnonmatching}. Finally, by virtue
of the inf-sup condition in~\eqref{eqn:inf_sup_discrete}, matrix $\mathsf{C}$ must have full row rank.

\section{Augmented Lagrangian-based preconditioning}\label{sec:AL_prec}
In this section, we derive an augmented Lagrangian-based preconditioner
for the linear system~\eqref{eqn:matrix}, extending the strategy investigated in~\cite{BENZI2026118522} for preconditioning other fictitious domain formulations. The derivation follows the AL approach introduced
in~\cite{ALprec} for the Oseen problem. The fundamental idea behind the AL approach is to
replace the original linear system
\begin{equation}\label{eqn:original}
  \begin{bmatrix}
    \mathsf{A} & 0             & \mathsf{C^T}    \\
    \mathsf{0} & \mathsf{A_2}  & \mathsf{-C_2^T} \\
    \mathsf{C} & \mathsf{-C_2} & 0
  \end{bmatrix}
  \begin{bmatrix}
    \mathsf{u}   \\
    \mathsf{u}_2 \\
    \mathsf{\lambda}
  \end{bmatrix}=
  \begin{bmatrix}
    \mathsf{ f} \\
    \mathsf{g}  \\
    \mathsf{0}
  \end{bmatrix} \qquad \text{ or } \qquad \mathcal{A}\mathsf{x} = \mathsf{b},
\end{equation}with an equivalent formulation. To this end, let us introduce a strictly positive real number $\gamma$, and a symmetric and positive definite matrix $\mathsf{W}$ that will need
to be chosen properly. Our goal is to exploit the constraint imposed by the last row
of the system. To do so, we start by augmenting the first equation with the term $\gamma\mathsf{C^T W^{-1}C u}$, obtaining
\begin{equation}\label{eqn:first_row_augmented}
  \mathsf{A u + C^T \lambda + \gamma C^T W^{-1}C u = f + \gamma C^T W^{-1}C u}. \\
\end{equation}
The last row of the system gives $\mathsf{Cu = C_2 u_2}$, which plugged into~\eqref{eqn:first_row_augmented} gives
\begin{equation}
  \mathsf{\Bigl(A +  \gamma C^T W^{-1}C \Bigr)u -\gamma C^T W^{-1}C_2 u_2 +C^T \lambda  = f }. \\
\end{equation}Proceeding in the same vein, we augment the second row with the term $\gamma \mathsf{C_2^T W^{-1}C_2 u_2}$ , which yields
\begin{equation}\label{eqn:second_row_augmented}
  \mathsf{-\gamma C_2 W^{-1}C u + \Bigl(A_2 +  \gamma C_2^T W^{-1}C_2 \Bigr)u_2 -C_2 \lambda  = g }. \\
\end{equation}With Equations~\eqref{eqn:first_row_augmented} and \eqref{eqn:second_row_augmented}, the whole augmented
system reads:

\begin{equation}
  \begin{bmatrix}\label{eqn:ALmatrix}
    \mathsf{A +  \gamma C^T W^{-1}C} & \mathsf{-\gamma C^T W^{-1}C_2}         & \mathsf{C^T}    \\
    \mathsf{-\gamma C_2^T W^{-1}C}   & \mathsf{A_2 +  \gamma C_2^T W^{-1}C_2} & \mathsf{-C_2^T} \\
    \mathsf{C}                       & \mathsf{-C_2}                          & 0
  \end{bmatrix}
  \begin{bmatrix}
    \mathsf{u}   \\
    \mathsf{u_2} \\
    \mathsf{\lambda}
  \end{bmatrix}=
  \begin{bmatrix}
    \mathsf{f} \\
    \mathsf{g} \\
    \mathsf{0}
  \end{bmatrix}\qquad \text{ or } \qquad \mathcal{A}_{\gamma}\mathsf{x} = \mathsf{b}.
\end{equation}
Having defined the augmented two-by-two block as:
\begin{equation}
  \mathsf{A_{\gamma}}\coloneqq
  \begin{bmatrix}\label{eqn:ALmatrix2by2}
    \mathsf{A +  \gamma C^T W^{-1}C} & \mathsf{-\gamma C^T W^{-1}C_2}         \\
    \mathsf{-\gamma C_2^T W^{-1}C}   & \mathsf{A_2 +  \gamma C_2^T W^{-1}C_2}
  \end{bmatrix},
\end{equation}
and the operator pair $\mathsf{B}$ as

\begin{equation}\mathsf{B}\coloneqq
  \begin{bmatrix}\label{eqn:B}
    \mathsf{C} & \mathsf{-C_2}
  \end{bmatrix},
\end{equation} we have that the system matrix in~\eqref{eqn:ALmatrix} can be rewritten (in compact form) as
\begin{equation}\label{eqn:ALmatrix_compact}
  \mathcal{A_{\gamma}} \coloneqq
  \begin{bmatrix}
    \mathsf{A_{\gamma}} & \mathsf{B^T} \\
    \mathsf{B}          & 0
  \end{bmatrix}.
\end{equation}
Notably, the augmentation
also removes the singularity in the original $(2,2)$-block of the system (cf. Remark~\ref{rmk:singularity}). Moreover,  $\mathsf{A_\gamma}$ can be expressed as:
\begin{equation}\label{eqn:A_comp}\mathsf{A_\gamma}=\widetilde{\mathsf{A}} + \gamma \mathsf{B^T W^{-1} B},
\end{equation}
where
\begin{equation}\label{eqn:Atilde}
  \widetilde{\mathsf{A}} \coloneqq \begin{bmatrix}
    \mathsf{A} & \mathsf{0}   \\
    \mathsf{0} & \mathsf{A_2}
  \end{bmatrix}.
\end{equation}

Therefore, since the linear system in~\eqref{eqn:original} has a unique solution and $\widetilde{\mathsf{A}}$ is SPSD, $\ker(\widetilde{\mathsf{A}})\cap \ker(\mathsf{B}) =\{0\}\footnote{This algebraic condition can be seen as an implicit consequence of the \emph{discrete ellipticity on the kernel} condition stated in Proposition~\ref{prop:elliptic_kernel_discrete}.} $ and $\mathsf{A_\gamma}$ is SPD~\cite{Benzi2005}. An \emph{ideal}~\cite{ALprec} preconditioner for the augmented linear system is given by the block triangular matrix
\begin{equation}\label{eqn:precAL}
  \mathcal{P_{\gamma}} \coloneqq
  \begin{bmatrix}
    \mathsf{A_{\gamma}} & \mathsf{B^T}                \\
    0                   & \mathsf{-\frac{1}{\gamma}W}
  \end{bmatrix}.
\end{equation}
In practice, the action of $\mathcal{P}_{\gamma}^{-1}$ is given by
$$\mathcal{P}_{\gamma}^{-1}
  =
  \begin{bmatrix}
    \mathsf{{{A}_{\gamma}}^{-1}} & 0               \\
    0                            & \mathsf{I}_\ell
  \end{bmatrix}
  \begin{bmatrix}
    \mathsf{I}_{n+m} & \mathsf{B^T}     \\
    0                & \mathsf{-I}_\ell
  \end{bmatrix}
  \begin{bmatrix}
    \mathsf{I}_{n+m} & 0                      \\
    0                & \gamma \mathsf{W^{-1}}
  \end{bmatrix},
$$
where $\mathsf{I}_{n+m}$ and $\mathsf{I}_\ell$ are identity matrices of size $n+m$ and $l$, respectively. The last identity implies that the application of the preconditioner to a vector requires one solve with $\mathsf{W}$, and one solve with the augmented term $\mathsf{{A_\gamma}}$. The solve with $\mathsf{A_\gamma}$ does not need to be exact and can be approximated by an inner iteration with a loose tolerance (e.g., $10^{-2}$), which is usually sufficient to keep the number of outer iterations low when combined with (flexible) GMRES.
Nevertheless, this step can still be challenging due to the large kernel of $\mathsf{B^T W^{-1} B}$. Increasing $\gamma$ leads to a larger condition number of $\mathsf{A_\gamma}$, owing to the increasing weight of the positive semidefinite augmentation term $\mathsf{B^T W^{-1}B}$, which may adversely affect the convergence of iterative solvers. In particular, the asymptotic deterioration for large $\gamma$ is consistent with the result of Fortin and Glowinski~\cite{FortinAL}, who prove that the condition number of the augmented block grows asymptotically \emph{linearly} with the augmentation parameter. Conversely, for sufficiently small positive values of $\gamma$, the augmentation is too weak to effectively regularize the singular (2,2)-block, so the condition number remains large because of the singularity of $\widetilde{\mathsf{A}}$. Hence, the condition number of $\mathsf{A_\gamma}$ is expected to attain a minimum for a certain value of $\gamma$. This value, however, does not necessarily yield the best overall performance of the preconditioner if the spectrum of the preconditioned system is not sufficiently clustered. Consequently, the choice of $\gamma$ requires balancing the cost of the inner solves against the number of outer iterations. As we will see in Section~\ref{sec:spectral_ideal}, $\gamma$ appears to be essentially independent of the mesh size. Therefore, a suitable value can be determined inexpensively on small-scale problems and subsequently reused for large-scale computations without any additional tuning effort. Indeed, the numerical experiments reported in Section~\ref{subsec:ideal_validation} show that a single choice of $\gamma$ is consistently effective across all refinement levels.

Given the
promising results obtained in~\cite{BENZI2026118522} in the fictitious domain-type context, $\mathsf{W}$ is chosen as
\begin{equation}
  \mathsf{W} \coloneqq  \mathsf{M^2},
\end{equation}where $\mathsf{M}$ is the mass matrix on the multiplier space $\Lambda_h$ defined in Equation~\eqref{eqn:mass_matrix}. The reason for this choice will become clear in Section~\ref{sec:spectral_ideal}, where
we perform a spectral analysis of the ideal preconditioner~\eqref{eqn:precAL}. Specifically, setting $\mathsf{W} =  \mathsf{M^2}$ ensures that the eigenvalues of the preconditioned matrix remain bounded away from zero uniformly in $h_{\Omega}$, $h_{\Omega_2}$ and, in practice, also fairly insensitive with respect to the jump in the coefficients $\beta_2-\beta$ (see Theorem~\ref{thm:spectral_AL_independence_of_h}).

In the particular case when $V_{2,h} \equiv \Lambda_h$ (see Remark~\ref{rmk:spaces}), one has $\mathsf{C_2} = \mathsf{M}$. With this specific choice for $\mathsf{W}$, the augmented $(2,2)$-block reduces to the SPD matrix $\mathsf{A_2 + \gamma I}_m$. This avoids potential loss of sparsity and allows the direct assembly of this augmented term, without the need to perform additional sparse matrix-matrix products.

\begin{remark}[Case $V_{2,h} \not\equiv \Lambda_h$]\label{rmk:choice_W}
  In the general case $V_{2,h} \not\equiv \Lambda_h$, the augmented diagonal $(2,2)$-block $\mathsf{A_2+\gamma C_2^T W^{-1} C_2}$ is still positive
  definite even though it is a sum of two positive semi-definite matrices. This follows upon noticing that $\ker(\mathsf{A_2}) = \operatorname{span}\{\mathsf{1}\}$, but
  $\mathsf{1} \not \in \ker(\mathsf{C_2})$, as can be verified by explicit computation.\footnote{This follows from the definition of $\mathsf{C_2}$ in~\eqref{eqn:matrices}, together with the fact that $\{\phi_j\}_{j=1}^m$ and $\{\psi_j\}_{j=1}^\ell$ are basis sets.} Hence, $\ker(\mathsf{A_2}) \cap \ker(\mathsf{C_2}) = \{0\}$, which implies the positive definiteness of the whole sum. This allows the spectral analysis performed in Section~\ref{sec:spectral_modified} to be general, and not tailored to a specific choice of the finite element discretization.
\end{remark}

\section{Spectral analysis of ideal preconditioner}\label{sec:spectral_ideal}
In this Section, we derive lower and upper bounds for the eigenvalues of the preconditioned matrix $\mathcal{P}_{\gamma}^{-1} \mathcal{A}_{\gamma}$.  However, in general eigenvalues alone may not fully characterize the convergence of nonsymmetric matrix iterations like
  GMRES, especially when the problem is far from normal~\cite{GMRES_eigs}. In such cases, the field-of-values provides a more appropriate theoretical framework for deriving convergence estimates~\cite{FoVAL}. Nevertheless, practical experience suggests that convergence is often fast when the spectrum is real, positive and confined within a moderately
  sized interval bounded away from $0$. Since, in our setting, the spectrum is shown to possess these properties, we restrict our theoretical analysis to the eigenvalue distribution, leaving a field-of-values analysis as an interesting direction for future work.

Specifically, we will show that the eigenvalues of $\mathcal{P}_{\gamma}^{-1} \mathcal{A}_{\gamma}$ cluster towards $1$ as the augmentation parameter $\gamma$ increases, confirming the usual behavior of AL-based preconditioners. Then, we will address
mesh-independence of the smallest eigenvalue with respect to discretization parameters. We first recall the following result about generalized Rayleigh quotients.
\begin{lemma}\label{lemma:gen_theory}
  Let $\mathsf{Q}$ and $\mathsf{N}$ be symmetric and symmetric positive definite matrices, respectively, with generalized eigenvalues $\nu_1 \le \cdots \le \nu_n$ and eigenvectors $\mathsf{v}_1, \ldots, \mathsf{v}_n \in \mathbb{R}^n$, such that
  $\mathsf{Q} \mathsf{v}_i = \nu_i \mathsf{N} \mathsf{v}_i.$
  Then:
  \begin{itemize}
    \item The smallest eigenvalue $\nu_1$ can be characterized as
          \[
            \nu_1 = \min_{\mathsf{x} \neq 0} \frac{\mathsf{x^T} \mathsf{Q} \mathsf{x}}{\mathsf{x^T} \mathsf{N} \mathsf{x}},
            \quad \text{achieved when } \mathsf{x} = \pm \mathsf{v}_1.
          \]
    \item The second smallest eigenvalue $\nu_{2}$ satisfies
          \[
            \nu_{2} = \min_{\substack{\mathsf{x} \neq 0 \\ \mathsf{x^T} \mathsf{N} \mathsf{v}_1 = 0}}
            \frac{\mathsf{x^T} \mathsf{Q} \mathsf{x}}{\mathsf{x^T} \mathsf{N} \mathsf{x}},
            \quad \text{achieved when } \mathsf{x} = \pm \mathsf{v}_{2},
          \]
  \end{itemize}
  and so on.
\end{lemma}

\begin{theorem}[Spectrum of preconditioned matrix]\label{thm:lambda_bound}
  Assume that $\mathcal{A}_\gamma$ and $\mathcal{P}_{\gamma}$ are defined by the matrices in Equations~\eqref{eqn:ALmatrix_compact} and~\eqref{eqn:precAL}, respectively. The eigenvalues of
  the preconditioned matrix $\mathcal{P}_{\gamma}^{-1} \mathcal{A}_\gamma$ are all real and positive. More precisely, let $\widetilde{\mathsf{A}}$ be defined as in~\eqref{eqn:Atilde}
  and $(\mathsf{x; y})$ be an eigenvector of $\mathcal{P}_{\gamma}^{-1} \mathcal{A}_{\gamma}$. It holds
  $$\operatorname{Spec}(\mathcal{P}_{\gamma }^{-1} \mathcal{A}_{\gamma}) \subseteq \left [ \eta, 1 \right ], $$
  where
  $$ \eta \coloneqq \min_{\mathsf{x} \in \mathcal{S}}  \frac{\gamma \mathsf{x^T B^T W^{-1} B x}}{\mathsf{x^T} \widetilde{\mathsf{A}}\mathsf{ x} + \gamma \mathsf{x^T B^T W^{-1} B x} }  ,
  $$

  \vspace{0.2cm}
  \noindent
  and with $\nu = 1$ being an eigenvalue of algebraic multiplicity at least $n+m$. Here $\mathcal{S}$ denotes the set $\mathcal{S}\coloneqq{\{\mathsf{v} \in \mathbb{R}^{n+m} | \mathsf{u^T} \mathsf{\widetilde{A}}\mathsf{v} =0\>\> \text{for} \> \mathsf{u} \in \ker(\mathsf{B})\}}$.
\end{theorem}

\vspace{0.1cm}
\begin{proof}

  Let $\nu \ne 0$ be an arbitrary eigenvalue of the preconditioned matrix, with a corresponding
  eigenvector $(\mathsf{x;y})$. Since both the original system matrix~\eqref{eqn:ALmatrix_compact} and the preconditioner~\eqref{eqn:precAL} are nonsingular, the preconditioned matrix is also nonsingular, and therefore $\nu=0$ cannot be an eigenvalue.

  Notice that $\mathsf{x}$ is
  actually a \emph{block} vector, which means that $\mathsf{x}=(\mathsf{x_1;x_2})$, with $\mathsf{x_1} \in \mathbb{R}^n$ and
  $\mathsf{x_2} \in \mathbb{R}^m$.
  The generalized eigenvalue problem can be stated as:
  \begin{equation}
    \begin{bmatrix}\label{eqn:eigenproblem}
      \mathsf{A_{\gamma}} & \mathsf{B^T} \\
      \mathsf{B}          & 0
    \end{bmatrix}
    \begin{bmatrix}
      {\mathsf x} \\
      {\mathsf y}
    \end{bmatrix}= \nu
    \begin{bmatrix}
      \mathsf{A_{\gamma}} & \mathsf{B^T}                 \\
      0                   & -\frac{1}{\gamma} \mathsf{W}
    \end{bmatrix}
    \begin{bmatrix}
      {\mathsf x} \\
      {\mathsf y}
    \end{bmatrix},
  \end{equation}
  which can be written explicitly as the following system of equations:
  \begin{align}
    \mathsf{A_{\gamma}x + B^T y } & = \nu \mathsf{(A_{\gamma}x + B^T y)}, \label{eqn:first_row_eigenproblem} \\
    \mathsf{Bx}                   & = \mathsf{-\frac{\nu}{\gamma} W y}.\label{eqn:second_row_eigenproblem}
  \end{align}
  Notice that $\mathsf{x} \ne 0$; otherwise, the positive definiteness of $\mathsf{W}$ implies $\mathsf{y=0}$, in
  contradiction with the fact that $(\mathsf{x, y})$ is an eigenvector. Therefore, we can assume $\mathsf{x} \ne 0$.

  It is evident from Equation~\eqref{eqn:first_row_eigenproblem} that $\nu=1$ is an eigenvalue of $\mathcal{P}_{\gamma}^{-1} \mathcal{A}_{\gamma}$, with
  associated eigenvector $(\mathsf{x;-\gamma W^{-1}Bx})$ when $\mathsf{x} \not \in \ker(\mathsf{B})$. Moreover, when $\mathsf{x} \in \ker(\mathsf{B})$, we
  have $(\mathsf{x;0})$ as associated eigenvector.

  We now assume $\nu \ne 1$. From~\eqref{eqn:first_row_eigenproblem} we derive
  \begin{equation} \label{eqn:only_x}
    \mathsf{A_{\gamma}} \mathsf x + \mathsf{B^T} \mathsf y  = 0,
  \end{equation}
  whereas from~\eqref{eqn:second_row_eigenproblem} we obtain $\mathsf{y=-\frac{\gamma}{\nu}W^{-1} Bx}$, which plugged into the previous equation
  gives
  \begin{equation}\label{eqn:first_row_mod}
    \mathsf{A_{\gamma}x - \frac{\gamma}{\nu} B^T W^{-1} B x = 0}.
  \end{equation}
  Multiplying both sides of~\eqref{eqn:first_row_mod} by $\nu \mathsf x^\ast$, we get
  \begin{equation}\label{ref:eq_for_lambda}
    \nu \mathsf{x^*}\mathsf{A_{\gamma}x - \gamma \mathsf{x^*} B^T W^{-1} B x = 0},
  \end{equation}
  and using~\eqref{eqn:A_comp}, we obtain the following explicit equation for $\nu$
  \begin{equation}\label{eqn:lambda}
    \nu  =  \frac{\gamma \mathsf{x^\ast B^T W^{-1} B x}}{\mathsf{x^\ast} \widetilde{\mathsf{A}} \mathsf{ x} + \gamma {\mathsf{x^\ast B^T W^{-1} B x}}}.
  \end{equation}The positive definiteness of $\mathsf{A_\gamma}$ and the positive semidefiniteness of $\mathsf{B^TW^{-1}B}$ imply that all the eigenvalues are positive and real\footnote{Hence, the corresponding eigenvector can also be chosen to be real: for this reason $\mathsf{x}^\ast$ will be replaced by $\mathsf{x^T}$.}. Moreover, note that in this case $\mathsf{x}$ cannot belong to $\ker(\widetilde{\mathsf{A}})$, since we are assuming $\nu \neq 1$, nor to $\ker(\mathsf{B})$ since $\nu \ne 0$.

  From~\eqref{eqn:lambda} we also deduce that $\nu < 1$, and that all eigenvalues satisfying~\eqref{eqn:lambda} cluster towards $1$ as $\gamma \rightarrow +\infty$, consistently with the theory of
  classical AL-based preconditioners~\cite{ALprec}. We know that all nonunit eigenvalues of $\mathcal{P}^{-1}_{ \gamma} \mathcal{A}_{\gamma}$ satisfy~\eqref{ref:eq_for_lambda}, hence the smallest eigenvalue can be characterized using Lemma~\ref{lemma:gen_theory}:
  $$ \eta \coloneqq \min_{\mathsf{x} \in \mathcal{S}}  \frac{\gamma \mathsf{x^T B^T W^{-1} B x}}{\mathsf{x^T} \widetilde{\mathsf{A}}\mathsf{ x} + \gamma \mathsf{x^T B^T W^{-1} B x} },
  $$which concludes the proof.
\end{proof}

\subsection{Numerical test: spectrum of preconditioned system}\label{sec:numerical_validation}

We perform some preliminary numerical experiments to illustrate the impact of the preconditioner $\mathcal{P}_{\gamma}$
on the spectrum of the original augmented system in~\eqref{eqn:ALmatrix}. In particular, we examine how the distribution of the eigenvalues of the preconditioned matrix changes when varying $\gamma$ and the magnitude of the coefficient jump $\beta_2 - \beta$. We will use $\mathcal{Q}^1$ Lagrangian elements for all the finite element spaces involved in the formulation. We consider the following
geometric configuration:
\begin{itemize}
  \item $\Omega \! = \! [0,1]^2$,

        \vspace{0.1cm}
  \item $\Omega_2 \!=\! [0.2,0.5]^2$,

        \vspace{0.1cm}
  \item $\beta = 1$,
\end{itemize}
and a discretization for which $\mathsf{A} \in \mathbb{R}^{1089 \times 1089}$, $\mathsf{A_2} \in \mathbb{R}^{81 \times 81}$, $\mathsf{C_2} \in \mathbb{R}^{81 \times 81}$, and $\mathsf{C} \in \mathbb{R}^{81 \times 1089}$. The global
system has then size $1251$. Our choice for $\mathsf{W}$ is $\mathsf{M}^{2}$. Eigenvalues have been computed using the \texttt{eig} function from \textsc{Matlab}.

We report in Figure~\ref{fig:ideal_spectrum_beta_100} the spectrum of the unpreconditioned matrix (top row, in magenta) and preconditioned matrix (bottom row, in green)  for increasing
values of the augmentation parameter $\gamma$ and a fixed value $\beta_2=100$. We note that the unpreconditioned system is not singular, but shows some
small negative eigenvalues clustered near the origin, which arise from the indefinite nature of the saddle point system and do not contradict the clustering result in Theorem~\ref{thm:lambda_bound}, which applies to the preconditioned matrix $\mathcal{P}_{\gamma}^{-1}\mathcal{A}_{\gamma}$. Several observations
  can already be drawn from this figure. First, the computed eigenvalues of $\mathcal{P}_{\gamma}^{-1}\mathcal{A}_{\gamma}$ are all positive and real, thus confirming the first part of Theorem~\ref{thm:lambda_bound}. Moreover, as expected, higher values of the
augmentation parameter $\gamma$ result in a wider spectrum for the original system $\mathcal{A}_\gamma$. The bottom row shows the nice effect of the preconditioner $\mathcal{P}_{\gamma}$, which clusters the spectrum near 1.

We perform the same test, setting this time $\beta_2 =10^6$. The numerical findings are shown in Figure~\ref{fig:ideal_spectrum_beta_1e6}. We observe
the spectrum of the original matrix to scale with the size of the jump, as can be seen from the different x-axis scales in the top rows of Figures~\ref{fig:ideal_spectrum_beta_100} and~\ref{fig:ideal_spectrum_beta_1e6}. In contrast, the spectrum of the preconditioned matrix
is essentially identical to the one in the bottom row of the previous figure, indicating that the preconditioner
$\mathcal{P}_{\gamma}$ remains quite effective in clustering the eigenvalues also for higher contrast in the coefficients.

\begin{figure}[h]
  \centering
  \begin{subfigure}{0.3\textwidth} %
    \centering
    \includegraphics{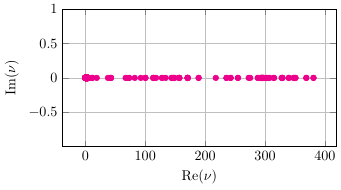}
  \end{subfigure}\hspace{1.1cm} %
  \begin{subfigure}{0.3\textwidth}
    \centering
    \includegraphics{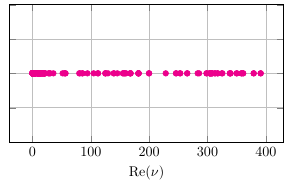}
  \end{subfigure}\hspace{0.05cm} %
  \begin{subfigure}{0.3\textwidth}
    \centering
    \includegraphics{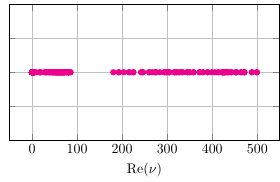}
  \end{subfigure}

  \begin{subfigure}{0.3\textwidth} %
    \centering
    \includegraphics{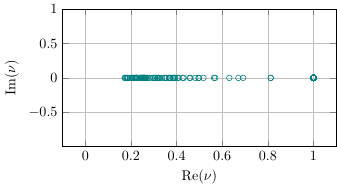}
    \caption*{\hspace{2.0cm} $\gamma\! =\! 1$}
  \end{subfigure}\hspace{1.1cm} %
  \begin{subfigure}{0.3\textwidth}
    \centering
    \includegraphics{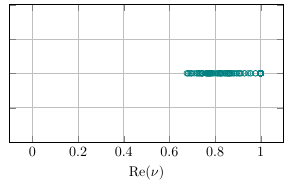}
    \caption*{$\gamma\! =\! 10$}
  \end{subfigure}\hspace{0.05cm} %
  \begin{subfigure}{0.3\textwidth}
    \centering
    \includegraphics{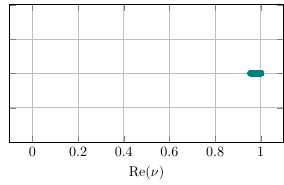}
    \caption*{$\gamma\! =\! 100$}
  \end{subfigure}
  \vspace{0.1cm}
  \caption{$\beta=1$, $\beta_2 = 100$. Spectrum of the original system matrix $\mathcal{A}_{\gamma}$ (top row, in magenta) and $\mathcal{P}_{\gamma}^{-1}\mathcal{A}_{\gamma}$ (bottom row, in green) for
    increasing values of the augmentation parameter $\gamma$.}
  \label{fig:ideal_spectrum_beta_100}
\end{figure}

\begin{figure}[h]
  \centering
  \begin{subfigure}{0.3\textwidth} %
    \centering
    \includegraphics{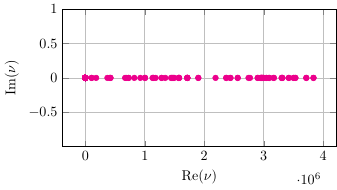}
  \end{subfigure}\hspace{1.1cm} %
  \begin{subfigure}{0.3\textwidth}
    \centering
    \includegraphics{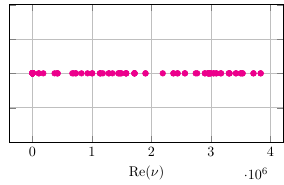}
  \end{subfigure}\hspace{0.05cm} %
  \begin{subfigure}{0.3\textwidth}
    \centering
    \includegraphics{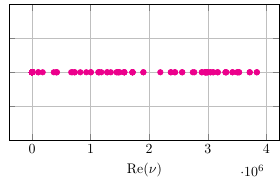}
  \end{subfigure}

  \vspace{0.2cm} %

  \begin{subfigure}{0.3\textwidth} %
    \centering
    \includegraphics{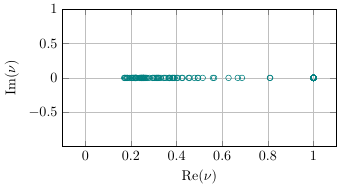}
    \caption*{\hspace{2.0cm} $\gamma\! =\! 1$}
  \end{subfigure}\hspace{1.1cm} %
  \begin{subfigure}{0.3\textwidth}
    \centering
    \includegraphics{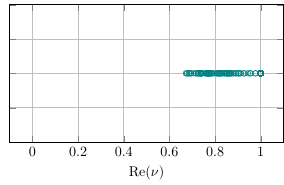}
    \caption*{$\gamma\! =\! 10$}
  \end{subfigure}\hspace{0.05cm} %
  \begin{subfigure}{0.3\textwidth}
    \centering
    \includegraphics{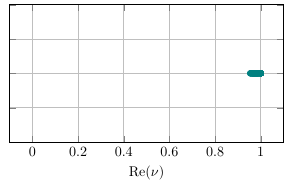}
    \caption*{$\gamma\! =\! 100$}
  \end{subfigure}
  \vspace{0.1cm}
  \caption{$\beta=1$, $\beta_2 = 10^6$. Spectrum of the original system matrix $\mathcal{A}_{\gamma}$ (top row, in magenta) and $\mathcal{P}_{\gamma}^{-1}\mathcal{A}_{\gamma}$ (bottom row, in green) for
    increasing values of the augmentation parameter $\gamma$. Notice the different scale in the x-axis of the first row.}
  \label{fig:ideal_spectrum_beta_1e6}
\end{figure}

\FloatBarrier
\subsection{Mesh-independence of lower bound $\eta$}
Concerning the choice $\mathsf{W}\!=\!\mathsf{M^2}$, we will show that this selection is motivated by its property of ensuring that the lower bound $\eta$ in Theorem~\ref{thm:lambda_bound} remains
uniformly bounded away from zero with respect to mesh sizes $h_\Omega$, $h_{\Omega_2}$, with dependence on the problem coefficients that will be made explicit. We stress that this only requires inf-sup stable spaces.

Before addressing the proof, we collect some useful technical results that will be used in the sequel.
\begin{lemma}[Spectral equivalence with $h$-scaled mass matrix]
  \label{lemma:spectral_equivalence_scaled_mass}
  Let $d=2$, $\mathsf{M}$ be the mass matrix defined on the multiplier space introduced in Section~\ref{sec:method}, and let $h_{\Omega_2}$ denote the mesh size of $\mathcal{T}_{2,h}$. Then, the
  matrices $h_{\Omega_2}^{-2}\mathsf{M^{-1}}$ and $\mathsf{M^{-2}}$ are
  spectrally equivalent, i.e. there exist constants $c_1, C_2$, independent of the discretization parameters, such that
  \begin{equation}\label{eqn:spectral_equivalence_scaled_mass}
    0 < c_1 \leq \frac{\mathsf{w^T M^{-2} w}}{\mathsf{w^T} h_{\Omega_2}^{-2} \mathsf{M^{-1} w}} \leq C_2, \quad \forall \mathsf{w} \in \mathbb{R}^\ell , \mathsf{w} \ne 0.
  \end{equation}
  Considerations about the case $d=3$ will be discussed later.

\end{lemma}

\begin{proof}
  Let $0 \ne \mathsf{w} \in \mathbb{R}^\ell$ be an arbitrary vector. Recall that two families of SPD matrices $\{\mathsf{Q}_\ell\}$ and $\{\mathsf{N}_\ell\}$ (parametrized by their dimension $\ell$) are said to be \textit{spectrally equivalent} if there exist $\ell$-independent constants $c_1$ and $C_2$ with
  $$0<c_1 \le \frac{ \mathsf{w^T} \mathsf{Q}_\ell \> \mathsf w}{\mathsf{w^T} \mathsf{N}_\ell \> \mathsf w} \le C_2.$$
  We drop the subscript $\ell$. We first multiply and then divide the generalized Rayleigh quotient in~\eqref{eqn:spectral_equivalence_scaled_mass} by $\mathsf{w^T w}$, so that it can be rewritten in terms of the two Rayleigh quotients associated with the matrices $(h_{\Omega_2}^2 \mathsf{M})^{-1}$ and $\mathsf{M^{-2}} $.

  For a general Hermitian matrix $\mathsf{H}$, its Rayleigh quotient lies in the interval $[\nu_{\text{min}}(\mathsf{H}), \nu_{\text{max}}(\mathsf{H})]$. Hence, we get
  \begin{equation}\label{eqn:intermediate_step}
    \frac{\nu_{\text{min}}(h_{\Omega_2}^2 \mathsf{M})}{\nu_{\text{max}}(\mathsf{M^2})} \leq \frac{\mathsf{w^T M^{-2} w}}{\mathsf{w^T w}} \frac{\mathsf{w^T w}}{\mathsf{w^T} h_{\Omega_2}^{-2} \mathsf{M^{-1} w}} \leq \frac{\nu_{\text{max}}(h_{\Omega_2}^2 \mathsf{M})}{\nu_{\text{min}}(\mathsf{M^2})}.
  \end{equation}

  \vspace{0.1cm}
  Since the immersed mesh $\mathcal{T}_{2,h}$ is assumed to be discretized in a quasi-uniform fashion, the following bound (see e.g.~\cite{Ern_Guermond-FEM-2004}) on the eigenvalues of the mass matrix holds
  \[
    c h_{\Omega_2}^2 \leq \frac{\mathsf{w^T M w}}{\mathsf{w^T w}} \leq C h_{\Omega_2}^2 \quad \forall \mathsf{w} \in \mathbb{R}^\ell ,
  \]for some positive constants $c$ and $C$ independent of the discretization parameters. Using this in~\eqref{eqn:intermediate_step} and noting that $\nu (h_{\Omega_2}^2 \mathsf{M}) = h_{\Omega_2}^2 \nu ( \mathsf{M})  $ and $\nu (\mathsf{M}^2) =(\nu ( \mathsf{M}))^2$ yields
  \[
    0 < \frac{c}{C^2} \leq \frac{\mathsf{w^T M^{-2} w}}{\mathsf{w^T} h_{\Omega_2}^{-2} \mathsf{M^{-1} w}} \leq \frac{C}{c^2}, \quad \forall \mathsf{w} \in \mathbb{R}^\ell ,
  \]which concludes the proof upon setting $c_1 = \frac{c}{C^2}$ and $C_2 = \frac{C}{c^2}$.
\end{proof}

The following result is based on the more general framework presented in~\cite{LiRenCang} for computing the eigenvalues and eigenvectors of Hermitian matrix pencils $\mathsf{Q} - \nu \mathsf{N}$, where $\mathsf{N}$ may be indefinite and even singular. For clarity of presentation, we state it in Theorem~\ref{thm:Courant-Fisher} in a simplified form, since in our problem we work exclusively with symmetric positive definite and symmetric positive semidefinite matrices.

We refer to $\mathsf{Q} - \nu \mathsf{N}$ as a \textit{positive semidefinite symmetric} pencil if both $\mathsf{Q}$ and $\mathsf{N}$ are symmetric and there exists a real scalar $\nu_0$ such that $\mathsf{Q} - \nu_0 \mathsf{N}$ is positive semidefinite. Moreover, $\mathsf{Q} - \nu \mathsf{N}$ has $r=\text{rank}(\mathsf{N})$ finite eigenvalues, all of which are real. In the algebraic theory it is customary to say that the  remaining eigenvalues are \textit{infinite} and their corresponding eigenvectors belong to the kernel of $\mathsf{N}$.
\begin{theorem}\label{thm:Courant-Fisher}
  Let $\mathsf{Q} - \nu \mathsf{N}$ be a positive semidefinite symmetric pencil, with $\mathsf{Q}$, $\mathsf{N}$ $\in \mathbb{R}^{n \times n}$. Let $\nu_i$ be the finite eigenvalues of $\mathsf{Q} - \nu \mathsf{N}$ arranged in the order $\nu_1 \leq \nu_2 \leq \cdots \leq \nu_r $. Then
  \[
    \nu_i =
    \inf_{\substack{\mathcal{X} \\ \dim \mathcal{X} = i}}
    \ \sup_{\substack{\mathsf{x} \in \mathcal{X} \\ \mathsf{x^T N x} \neq 0}}
    \frac{\mathsf{x^T Q x}}{\mathsf{x^T N x}},
  \]
  where $\mathcal{X}$ is a subspace of $\mathbb{R}^n$.
\end{theorem}

\vspace{0.3cm}
This theorem demonstrates that the presence of infinite eigenvalues does not affect the minimax formulas (see also~\cite{NAJMAN1993183}). Note that, when $\mathsf{N}$ is also positive definite, the associated eigenvalue problem is equivalent to a standard symmetric eigenvalue problem, and therefore the Courant-Fischer minimax principles continue to hold.

Finally, we define the following norms:
\begin{equation}\label{eq:norm1}
  |v_h|_{\mathsf{A},\Omega} \coloneqq  (\mathsf{v_1^T} \mathsf{A}  \mathsf{v_1})  ^{1/2},
\end{equation}
\begin{equation}\label{eq:norm2}
  |v_{2,h}|_{\mathsf{A_2},\Omega_2} \coloneqq  (\mathsf{v_2^T} \mathsf{A_2}  \mathsf{v_2})  ^{1/2},
\end{equation}
\begin{equation}\label{eq:norm3}
  \|\mu_h\|_{0,\Omega_2} \coloneqq ( \mathsf{\upmu^T} \mathsf{M}  \mathsf{\upmu} ) ^{1/2},
\end{equation}
where $\mathsf v_1$, $\mathsf v_2$, and $\mathsf{\upmu}$ are vectors of the coefficients associated with the basis sets $\{\varphi_i\}_{i=1}^n$, $\{\phi_j\}_{j=1}^m$, and $\{\psi_k\}_{k=1}^\ell $. The matrices $\mathsf{A}$, $\mathsf{A_2}$, and $\mathsf{M}$ are
defined in Section~\ref{subsec:fem}.

The following Lemma provides an algebraic interpretation of the discrete inf-sup stability conditions presented
in Section~\ref{sec:method}, and is based on the seminal work by Malkus~\cite{malkus}, which gives a specific link between
the discrete inf-sup constant and the eigenvalues of a generalized eigenvalue problem. The next proof will follow similar arguments to the ones in~\cite[Lemma 3]{BENZI2026118522}.
\begin{lemma}[Algebraic inf-sup condition]
  \label{lemma:algebraic_infsup}
  Let $\mathsf{A}, \mathsf{A_2}$, $\mathsf{B}$ and $\mathsf{M}$ be the matrices defined in previous sections, and assume that the discrete inf-sup condition
  in Proposition~\ref{prop:inf_sup_discrete} is satisfied. Moreover, let $\tilde{\mathsf{{A}}}$ be defined as in~\eqref{eqn:Atilde}, and $\widetilde{\mathsf{M}}_2= \begin{bmatrix}
      0 & 0            \\
      0 & \mathsf{M_2}
    \end{bmatrix}$, being $\mathsf{M_2}$ the mass matrix defined on $V_{2,h}$\footnote{In principle, $\mathsf{M_2}$ might not coincide with the mass matrix $\mathsf{M}$ defined on the multiplier space $\Lambda_h$, unless $V_{2,h}\equiv\Lambda_h$.}. Then, there exists a positive constant $\bar \theta $ independent of the mesh sizes $h_{\Omega}$ and $h_{\Omega_2}$ (but depending on coefficients $\beta$ and $\beta_2$) such that
  \begin{equation}
    \bar \theta^2 \le \min_{\mathcal{S}\setminus \ker(\widetilde{\mathsf{{A}}})} \frac{\mathsf{v^T} \mathsf{B^T} \mathsf{M^{-2}}\mathsf{B}  \mathsf v}{\mathsf{v^T} \widetilde{\mathsf{{A}}} \mathsf v},
  \end{equation}
  where $\mathcal{S}\coloneqq{\{\mathsf{v} \in \mathbb{R}^{n+m} | \mathsf{u^T} \mathsf{\widetilde{A}}\mathsf{v} =0\>\> \text{for} \> \mathsf{u} \in \ker(\mathsf{B})\}}$.
\end{lemma}

\begin{proof}
  We start by recalling the following inverse inequality, valid for all $\mu_h \in \Lambda_h$:
  \begin{equation}\label{eqn:inverse_inequality}
    \|\mu_h\|_{-1,\Omega_2} \geq C h_{\Omega_2} \|\mu_h\|_{0,\Omega_2},
  \end{equation}with $C$ a positive constant independent of $h_{\Omega_2}$ \cite{BoffiGastaldi2024}. By definition of the norms in~\eqref{eq:norm1} and~\eqref{eq:norm2}, and assuming constant coefficients $\beta_i$ in each subdomain,
  we have
  $$|v_h|_{\mathsf{A},\Omega}= \beta^{\frac{1}{2}}\,|v_h|_{1,\Omega},$$
  $$|v_{2,h}|_{\mathsf{A_2},\Omega_2}= (\beta_2 -\beta)^{\frac{1}{2}}\,|v_{2,h}|_{1,\Omega_2}.$$
  It is straightforward to verify that the sum of norms appearing in the denominator of the inf-sup condition~\eqref{eqn:inf_sup_discrete} satisfies the following inequality
  $$\max\{\beta, \beta_2 - \beta, 1\} \bigl(|v_h|_{1,\Omega}^2 + \|v_{2,h}\|_{1,\Omega_2}^2\bigr) \ge \bigl(|v_h|_{\mathsf{A},\Omega}^2 + |v_{2,h}|_{\mathsf{A_2},\Omega_2}^2 + \|v_{2,h}\|_{0,\Omega_2}^2\bigr).$$
  Combining the above inequality with~\eqref{eqn:inf_sup_discrete} and~\eqref{eqn:inverse_inequality}, we obtain the following \emph{rescaled} inf-sup condition:
  \begin{equation}
    \label{eqn:inf_sup_discrete_scaled}
    \inf_{\mu_h \in \Lambda_h}\sup_{(v_h,v_{2,h}) \in V_h \times V_{2,h}} \frac{(\mu_h, v_{h|\Omega_2} - v_{2,h})_{\Omega_2}}{\bigl(|v_h|_{\mathsf{A},\Omega}^2 + |v_{2,h}|_{\mathsf{A_2},\Omega_2}^2 + \|v_{2,h}\|_{0,\Omega_2}^2\bigr)^{\frac12}\,\, h_{\Omega_2}\|\mu_h\|_{0,\Omega_2}}  \geq \tilde\theta  .
  \end{equation}
  where $\tilde\theta \coloneqq \frac{C \theta}{\sqrt{\max\{\beta, \beta_2-\beta,1\}}}$, independently of the mesh sizes.

  It is easy to see that $\mathsf{v^T} \bigl(\widetilde{\mathsf{{A}}}+\widetilde{\mathsf{{M}}}_2\bigr) \mathsf v$ realizes the sum of norms appearing in the denominator of~\eqref{eqn:inf_sup_discrete_scaled}. Using the fact that $(\mu_h,v_{h|\Omega_2} - v_{2,h})_{\Omega_2}= \mathsf{v^T B^T \upmu}$, the discrete inf-sup condition in~\eqref{eqn:inf_sup_discrete_scaled} can be
  rewritten in the following algebraic form:
  \begin{equation}\label{eqn:discrete_infsup_algebraic}
    \inf_{\upmu \ne 0} \sup_{\mathsf{v} \ne 0} \frac{\mathsf{v^T B^T \upmu}}{(\mathsf{v^T} \bigl(\widetilde{\mathsf{A}}+\widetilde{\mathsf{M}}_2 \bigr)\mathsf{v})^{\frac{1}{2}} (\mathsf{\upmu^T} h_{\Omega_2}^2\mathsf{M \mathsf{\upmu}})^{\frac{1}{2}}}  \geq \tilde\theta.
  \end{equation}Then, we argue as in~\cite[Th. 2]{malkus}, and let $0 < \sigma_1 \leq \sigma_2 \leq \cdots \leq \sigma_{\ell}$ be the $\ell$ largest eigenvalues of the eigenproblem
  \begin{equation}\label{eqn:gen_eig_prob}
    (\mathsf{B^T} h_{\Omega_2}^{-2}\mathsf{M^{-1}} \mathsf{B})\mathsf{v} = \sigma \bigl(\widetilde{\mathsf{A}}+\widetilde{\mathsf{M}}_2 \bigr)\mathsf{ v}.
  \end{equation}
  The constant appearing in Equation~\eqref{eqn:discrete_infsup_algebraic} is given by the square root of $\sigma_1$, which is thus independent
  of the mesh sizes (the proof follows by using the same approach of~\cite{malkus}, where the classical Stokes problem was considered). By considering the properties of the generalized Rayleigh quotient, the following characterization holds:
  \begin{equation}\label{eqn:beta_min}
    \tilde\theta^2 = \sigma_1 = \min_{\mathcal{R}} \frac{\mathsf{v^TB^T} h_{\Omega_2}^{-2}\mathsf{M^{-1}} \mathsf{B} \mathsf{v}}{\mathsf{v^T} \bigl(\widetilde{\mathsf{A}}+\widetilde{\mathsf{M}}_2 \bigr) \mathsf{v}},
  \end{equation}
  where $\mathcal{R}\coloneqq\{\mathsf{v} \in \mathbb{R}^{n+m} ~|~   \mathsf{u^T}\mathsf{(\widetilde{A}+\widetilde{M}_2)}\mathsf{v}=0 \>\> \text{for}\> \mathsf{u} \in \ker(\mathsf{B})\}$.

  Using now the spectral equivalence established in Lemma~\ref{lemma:spectral_equivalence_scaled_mass} between $h_{\Omega_2}^{-2} \mathsf{M}^{-1}$ and $\mathsf{M}^{-2}$, we have for a $\mathsf{w} \in \mathbb{R}^\ell $:
  \begin{equation}\label{eqn:spetr_eq}\frac{c}{C^2} \leq \frac{\mathsf{w^T M^{-2} w}}{\mathsf{w^T} h_{\Omega_2}^{-2} \mathsf{M^{-1} w}} \leq \frac{C}{c^2}, %
  \end{equation}
  where the positive constants $c, C$ are independent of the discretization parameters. Setting $\mathsf{w} \!=\! \mathsf{B} \mathsf v$, and multiplying each term of (\ref{eqn:spetr_eq}) by $\frac{\mathsf{w^T} (h_{\Omega_2}^2 \mathsf{M})^{-1} \mathsf{w}}{\mathsf{v^T} \bigl(\widetilde{\mathsf{A}}+\widetilde{\mathsf{M}}_2 \bigr) \mathsf v}$, we get
  $$0 < \frac{c}{C^2} \> \frac{\mathsf{v^T} \mathsf{B^T}(h_{\Omega_2}^2 \mathsf{M})^{-1} \mathsf{B}\mathsf v}{\mathsf{v^T} \bigl(\widetilde{\mathsf{A}}+\widetilde{\mathsf{M}}_2 \bigr)\mathsf v}\le \frac{\mathsf{v^T} \mathsf{B^T} \mathsf{M^{-2}} \mathsf{B}\mathsf v,}{\mathsf{v^T} \bigl(\widetilde{\mathsf{A}}+\widetilde{\mathsf{M}}_2 \bigr) \mathsf v} \le \frac{C}{c^2}\> \frac{\mathsf{v^T}\mathsf{B^T}(h_{\Omega_2}^2 \mathsf{M})^{-1} \mathsf{B} \mathsf v}{\mathsf{v^T} \bigl(\widetilde{\mathsf{A}}+\widetilde{\mathsf{M}}_2 \bigr) \mathsf v}.$$

  Combining this result with~\eqref{eqn:beta_min}, it follows that

  \begin{equation}\label{eqn:first_bound}
    \bar \theta^2 \leq \min_{\mathcal{R}} \frac{\mathsf{v^TB^T} \mathsf{M^{-2}} \mathsf{B} \mathsf{v}}{\mathsf{v^T} \bigl(\widetilde{\mathsf{A}}+\widetilde{\mathsf{M}}_2 \bigr)\mathsf{ v}},
  \end{equation}
  (where $\bar\theta^2 \coloneqq \frac{{c}}{C^2}\tilde\theta^2$), uniformly in $h_\Omega$ and $h_{\Omega_2}$.

  We now conclude the proof by comparing the generalized Rayleigh quotient in~\eqref{eqn:first_bound} with the one whose denominator is $\mathsf{v^T}\widetilde{\mathsf{A}}\mathsf{v}$. According to the theory of generalized eigenvalue problems~\cite{LiRenCang}, since $\dim \bigl(\ker(\widetilde{\mathsf{{A}}})\bigr)=1$, the matrix pencil $\mathsf{B^T} \mathsf{M}^{-2}\mathsf{B}- \nu \widetilde{\mathsf{A}}$ has exactly one infinite eigenvalue with eigenvector belonging to $\ker(\widetilde{\mathsf{{A}}})$.

  We have
  \begin{equation}\label{eq:pointwise_compare_spd_final}
    \frac{\mathsf{v^T}\mathsf{B^T M^{-2} B}\,\mathsf v}{\mathsf{v^T}(\widetilde{\mathsf{A}}+\widetilde{\mathsf{M}}_2)\mathsf v}
    \le
    \frac{\mathsf{v^T}\mathsf{B^T M^{-2} B}\,\mathsf v}{\mathsf{v^T}\widetilde{\mathsf{A}}\mathsf v} \qquad \forall \mathsf{v} \in\mathbb{R}^{n+m}, \text{ with } \mathsf{v} \notin \ker(\widetilde{\mathsf{A}}).
  \end{equation}
  Using this inequality together with Theorem~\ref{thm:Courant-Fisher}, we can compare the eigenvalues of the two associated generalized eigenvalue problems. We note that the minimax characterization in Theorem~\ref{thm:Courant-Fisher} can naturally accommodate infinite eigenvalues. In particular, vectors in $\ker(\widetilde{\mathsf A})$ formally correspond to an infinite generalized eigenvalue in the second quotient, leaving the inequality trivially satisfied. In other words, in the right-hand side of the inequality above, including or excluding $\ker(\widetilde{\mathsf A})$ from the subspaces over
  which the minimax principle is taken, does not affect the smallest finite eigenvalue.
  Consequently, the restriction $\mathsf{v^T} \widetilde{\mathsf A}\, \mathsf v \neq 0$ in the supremum can be removed, and we get:

  \[
    \sigma_i \le
    \inf_{\substack{\mathcal{X} \\ \dim \mathcal{X} = i}}
    \ \sup_{\substack{\mathsf{v} \in \mathcal{X}}}
    \frac{\mathsf{v^T B^T M^{-2} B v}}{\mathsf{v^T \bigl(\widetilde{\mathsf{A}}+\widetilde{\mathsf{M}}_2 \bigr) v}} \le
    \inf_{\substack{\mathcal{X} \\ \dim \mathcal{X} = i}}
    \ \sup_{\substack{\mathsf{v} \in \mathcal{X} }}
    \frac{\mathsf{v^T B^T M^{-2} B v}}{\mathsf{v^T \widetilde{\mathsf{A}} v}}=\nu_i,
  \]
  where $\sigma_i$ are the generalized eigenvalues of~\eqref{eqn:gen_eig_prob}, while $\nu_i$ are those of $\mathsf{B^T M^{-2} B v}=\nu  \widetilde{\mathsf{A}}\mathsf{v}$.
  Therefore, each $\nu_i$ is greater than or equal to the corresponding $\sigma_i$. This also applies to the smallest positive eigenvalue, and we obtain
  $$\bar{\theta}^2 \le \min_{\mathcal{R}} \frac{\mathsf{v^TB^T} \mathsf{M^{-2}} \mathsf{B} \mathsf{v}}{\mathsf{v^T} \bigl(\widetilde{\mathsf{A}}+\widetilde{\mathsf{M}}_2 \bigr)\mathsf{ v}} \le \min_{\mathcal{S}} \frac{\mathsf{v^TB^T} \mathsf{M^{-2}} \mathsf{B} \mathsf{v}}{\mathsf{v^T} \widetilde{\mathsf{A}}\mathsf{ v}}, $$
  where $\mathcal{S}\coloneqq{\{\mathsf{v} \in \mathbb{R}^{n+m} | \mathsf{u^T} \mathsf{\widetilde{A}}\mathsf{v} =0\>\> \text{for} \> \mathsf{u} \in \ker(\mathsf{B})\}}$.
  The thesis follows from the observation that the minimization over $\mathcal{S}$ is unaffected by the presence of $\ker(\widetilde{\mathsf A})$. Therefore, minimizing over $\mathcal{S}$ or over $\mathcal{S}\setminus\ker(\widetilde{\mathsf A})$ yields the same value.
\end{proof}

\begin{theorem}[Mesh-independence for lower bound]\label{thm:spectral_AL_independence_of_h}
  Let $V_h$, $V_{2,h}$ and $\Lambda_h$ be defined as in Section~\ref{sec:method}. If $\mathsf{W} \coloneqq \mathsf{M^2}$, then the lower bound
  in $$\operatorname{Spec}(\mathcal{P}_{\gamma}^{-1} \mathcal{A}_{\gamma}) \subseteq \left [ \eta, 1 \right ]$$
  from Theorem~\ref{thm:lambda_bound} is bounded away from zero independently of the discretization parameters, $h_\Omega$ and $h_{\Omega_2}$ (but depending on coefficients $\beta$ and $\beta_2$).
\end{theorem}
\begin{proof}
  From the proof of Theorem~\ref{thm:lambda_bound} (more precisely from Equation~\eqref{eqn:first_row_mod}), we know that all nonunit eigenvalues of $\mathcal{P}^{-1}_{ \gamma} \mathcal{A}_{\gamma}$ coincide with the eigenvalues of the following generalized eigenvalue problem:
  \begin{equation}\label{eq:gen_EP}
    \gamma  \mathsf{B^T} \mathsf{W}^{-1}\mathsf{B}  \mathsf x = \nu (\widetilde{\mathsf{A}} + \gamma \mathsf{B^T W^{-1} B})\mathsf x,
  \end{equation}
  excluding the cases $\nu = 0$ and $\nu = 1$, which correspond respectively to eigenvectors $\mathsf{x} \in \ker(\mathsf{B})$ and $\mathsf{x} \in \ker(\widetilde{\mathsf{A}})$.
  Our goal is to show that $\eta=\nu_{\min}^+$, the smallest positive eigenvalue of~\eqref{eq:gen_EP}, is bounded away from zero. By Remark~\ref{lemma:gen_theory}, this eigenvalue admits the Rayleigh quotient characterization
  \begin{equation}\label{eqn:eta_min}\eta = \min_{\mathsf{v} \in \mathcal{S}} \frac{\gamma \mathsf{v^T B^T W^{-1} B v}}{\mathsf{v^T}(\widetilde{\mathsf{A}} + \gamma \mathsf{B^T W^{-1} B })\mathsf{v}},\end{equation}where $\mathcal{S}$ is defined as

  \[\mathcal{S} \coloneqq \{\mathsf{v} \in \mathbb{R}^{n+m}: \mathsf{u^T}(\widetilde{\mathsf{A}} + \gamma \mathsf{B^T W^{-1} B })\mathsf{v} =0 \quad \forall \mathsf{u} \in \ker(\mathsf{B}) \}.  \]
  The condition $\mathsf{u} \in \ker(\mathsf{B})$ gives the following alternative characterization of $\mathcal{S}$:
  \[\mathcal{S} = \{\mathsf{v} \in \mathbb{R}^{n+m}: \mathsf{u^T}\mathsf{\widetilde{\mathsf{A}}}\mathsf{v} =0 \quad \forall \mathsf{u} \in \ker(\mathsf{B}) \}.\]
  Note that $\mathcal{S}$ contains any nonzero vector in $\ker(\mathsf{\widetilde{A}})$. However, such vectors correspond to the eigenvalue $\nu=1$ in~\eqref{eq:gen_EP}, that is the largest eigenvalue rather than the smallest positive one. Nevertheless, it is convenient to exclude $\ker(\mathsf{\widetilde{A}})$ from $\mathcal{S}$, because for every $\mathsf{v}\in\mathcal{S}\setminus\ker(\widetilde{\mathsf{A}})$ we have $\mathsf{v^{T}\widetilde{\mathsf{A}}v>0}$, thus preventing division by zero in the Rayleigh quotient below. This exclusion does not alter the value of the smallest \emph{finite} eigenvalue, since
  minimizing over $\mathcal{S}$ or over
  $\mathcal{S}\setminus \ker(\widetilde{\mathsf{A}})$ yields the same result for $\eta$.

  Consequently, we may safely multiply and divide the Rayleigh quotient in~\eqref{eqn:eta_min} by $\mathsf{v^{T}\widetilde{\mathsf{A}}v}$, and obtain
  \[
    \eta = \min_{\mathsf{v} \in \mathcal{S}\setminus\ker(\widetilde{\mathsf{A}})} \frac{\gamma \mathsf{v^T B^T W^{-1} B v}}{\mathsf{v^T \widetilde{\mathsf{A}} v}} \frac{\mathsf{v^T \widetilde{\mathsf{A}} v}}{\mathsf{v^T}(\widetilde{\mathsf{A}} + \gamma \mathsf{B^T W^{-1} B })\mathsf{v}}.
  \]

  As long as the denominator is nonzero, we can define $r(\mathsf{v}) \coloneqq \frac{\mathsf{v^T B^T W^{-1} B v}}{\mathsf{v^T} \widetilde{\mathsf{A}} \mathsf{v}}$ as the generalized Rayleigh quotient associated with $\mathsf{B^T W^{-1} B}$ and $\widetilde{\mathsf{{A}}}$.  We observe that $\eta$ can
  be written as
  \[\eta = \min_{\mathsf{v} \in \mathcal{S}\setminus\ker(\widetilde{\mathsf{A}})} \frac{\gamma r(\mathsf{v})}{1 + \gamma r(\mathsf{v})} = \min_{\mathsf{v} \in \mathcal{S}\setminus\ker(\widetilde{\mathsf{A}})} f(r(\mathsf{v})).\] Since $f_{\gamma}: x \mapsto \frac{\gamma x}{1 + \gamma x}$ is a monotone increasing
  function for $x>0$,
  $$\min_{\mathsf{v} \in \mathcal{S}\setminus\ker(\widetilde{\mathsf{A}})} f(r(\mathsf{v})) =  f\biggl(\min_{\mathsf{v} \in \mathcal{S}\setminus\ker(\widetilde{\mathsf{A}})}r(\mathsf{v})\biggr).$$

  Choosing $\mathsf{W=M^2}$, from Lemma~\ref{lemma:algebraic_infsup} we conclude that
  \begin{equation}\label{eqn:lower_bound}
    \eta \geq  f(\bar \theta^2)  = \frac{\gamma \bar \theta^2}{1 + \gamma \bar \theta^2}>0.
  \end{equation}
\end{proof}

\begin{remark}[Three-dimensional case]\label{remark:3d_lower_bound}
  When $d=3$, using the same arguments as above, we get
  \[\eta \geq  f(\bar \theta^2)  = \frac{\gamma \bar \theta^2}{h_{\Omega_2} + \gamma \bar \theta^2}>0.\]
  To show this, assuming a quasi-uniform discretization for $\mathcal{T}_{2,h}$, it holds:
  \begin{equation}
    c h_{\Omega_2}^3 \le \frac{ \mathsf{w^T}\mathsf{M} \mathsf{w}  }{\mathsf{w^T} \mathsf{w}} \le C h_{\Omega_2}^3 \quad \quad \forall \mathsf{w} \in \mathbb{R}^\ell ,
  \end{equation}
  yielding
  \begin{equation}\label{eq:bounds1}
    0 < \frac{c}{C^2 h_{\Omega_2}} \le \frac{\mathsf{w^T} \mathsf{M^{-2}} \mathsf{w}}{\mathsf{w^T} (h_{\Omega_2}^2 \mathsf{M})^{-1} \mathsf{w}} \le \frac{C}{c^2 h_{\Omega_2}}.
  \end{equation}
  Setting again $\mathsf{w} = \mathsf{B v}$ in the proof of Lemma~\ref{lemma:algebraic_infsup}, it holds that
  \begin{equation}\label{eq:2d}
    \min_{\{\mathsf{v} \in \mathbb{R}^{n+m} | \mathsf{u^T}\widetilde{\mathsf{A}} \mathsf{v}  =0 \text{ for }  \mathsf{u} \in \ker(\mathsf{B})\}}\frac{\mathsf{v^T} \mathsf{B^T} \mathsf{M^{-2}}\mathsf{B}\mathsf{v} }{ \mathsf{v^T}\widetilde{\mathsf{A}} \mathsf v } \ge \frac{c}{C^2}\frac{\tilde{\theta}^2}{h_{\Omega_2}}.
  \end{equation}
  Following the same steps used to derive Equation~\eqref{eqn:lower_bound}, we employ the inequality in~\eqref{eq:2d} above to obtain (denoting $\bar \theta^2=\frac{c}{C^2} \tilde \theta^2$)
  \begin{equation}\label{eqn:bound_h_beta}
    0<\frac{\gamma \bar{\theta}^2}{h_{\Omega_2}+\gamma \bar{\theta}^2} \leq  \eta.
  \end{equation}
  Notice how in this case the lower bound for the eigenvalues of the preconditioned matrix depends explicitly on $h_{\Omega_2}$ but in a favorable way. Indeed, it is immediate to observe that as $h_{\Omega_2}\rightarrow 0$, the lower bound in~\eqref{eqn:bound_h_beta} tends to $1$. Therefore, we conclude that also in this scenario the bound is robust with respect to the discretization parameters. In particular, the theoretical analysis suggests that, for the ideal AL preconditioner, the spectral distribution does not deteriorate under mesh refinement and may even improve.

\end{remark}
\begin{remark}[Parameter robustness]\label{rmk:parameter_robustness}
  From the previous proof and Lemma~\ref{lemma:algebraic_infsup}, it follows that the lower bound obtained in Theorem~\ref{thm:lambda_bound} will in general depend on~$\beta$ and on the jump $\beta_2-\beta$ through the quantity $\max\{\beta,\beta_2-\beta,1\}$. Indeed, the lower bound in~\eqref{eqn:lower_bound} involves the constant $\bar{\theta}^2$, which is defined in terms of $\tilde{\theta}^2$ in~\eqref{eqn:first_bound}; in turn, $\tilde{\theta}^2$ depends on $\max\{\beta,\beta_2-\beta,1\}$ as introduced in the discrete inf-sup estimate~\eqref{eqn:inf_sup_discrete_scaled}:
    \begin{equation*}
      \eta \geq f(\bar{\theta}^2)
      = \frac{\gamma\bar{\theta}^2}{1+\gamma\bar{\theta}^2}
      = \frac{\gamma\frac{c}{C^2}\tilde{\theta}^2}
      {1+\gamma\frac{c}{C^2}\tilde{\theta}^2}
      = \frac{\gamma\frac{\tilde c}{\max\{\beta,\beta_2-\beta,1\}}\theta^2}
      {1+\gamma\frac{\tilde c}{\max\{\beta,\beta_2-\beta,1\}}\theta^2}
      >0,
    \end{equation*}
    where $\tilde c$ contains the involved constants.
  These results show that choosing $\gamma = \mathcal{O}\!\left(\max\{\beta,\beta_2-\beta,1\}\right)$
  in~\eqref{eqn:lower_bound} or~\eqref{eqn:bound_h_beta} is sufficient to guarantee that all nonunit eigenvalues of $\mathcal{P}_{\gamma}^{-1}\mathcal{A}_{\gamma}$ remain uniformly bounded away from $0$. 
This behaviour is consistent with the observations reported in~\cite{ALprec} for the Oseen problem, where the authors show that the spectral bounds of the preconditioned system depend on the kinematic viscosity $\upnu$, and therefore a parameter choice of the form $\gamma = \mathcal{O}(\upnu^{-1})$ is theoretically required. Consequently, $\gamma$ becomes
large in the high Reynolds number regime (for small values of $\upnu$).
In our setting, an analogous behavior holds: as the coefficient jump
increases, $\gamma$ must be taken correspondingly larger. As pointed out in~\cite{ALprec}, this theoretical guideline is of limited practical relevance.  Indeed, in the numerical results in Section~\ref{sec:numerical_experiments} we will observe that, for the ideal AL preconditioner $\mathcal{P}_\gamma$, using a fixed value of $\gamma$ already yields convergence of the Krylov subspace method that is independent of $h_\Omega$, $h_{\Omega_2}$, the coefficients $\beta$, $\beta_2$, and their jump $\beta_2-\beta$, without the need to scale $\gamma$ according to the problem parameters.  Therefore, this theoretical condition on $\gamma$ is sufficient but conservative.

\end{remark}

\section{Modified AL preconditioner}\label{sec:modified_prec}

As shown in prior studies, e.g.~\cite{ALprec,modALprec}, the performance of the AL approach crucially depends on the availability of an effective approximate solver for the augmented block.
However, as previously mentioned, developing such a solver is a nontrivial task.
The main difficulty comes from the fact that when \(\gamma\) is large the augmented block becomes increasingly ill-conditioned. For this reason, in~\cite{ALprec}, a robust multigrid solver for it was proposed for the Oseen problem and later extended in~\cite{Farrell2019} to the three-dimensional case.
We note that developing such a tailored solver in our setting is more cumbersome, due to the independence of the two grids $\mathcal{T}_h$ and $\mathcal{T}_{2,h}$, which further complicates the development of possible multilevel approaches.
Therefore, we adopt the modified AL approach proposed in~\cite{modALprec}, adjusted to take advantage of the specific properties of the elliptic interface problem.

In what follows, we recall the idea behind the modified AL approach, briefly referring to the context in which it was first introduced in~\cite{modALprec}, namely the Oseen problem. In this section, \( \mathsf{u} \) and \( \mathsf{p} \) denote nodal values for velocity and pressure, respectively, while \( \mathsf{N} \) and \( \mathsf{B}  \) represent the discretizations of the diffusion--convection terms and (negative) divergence operator, respectively.

\subsection{%
  Modified AL preconditioner for the Oseen problem}
Discretization of the Oseen equations using stable finite element pairs yields large, sparse linear systems of the form
\begin{equation}\label{eqn:oseen}
  \begin{bmatrix}
    \mathsf{N} & \mathsf{B^T} \\[0.3em]
    \mathsf{B} & \mathsf{0}
  \end{bmatrix}
  \begin{bmatrix}
    \mathsf{u} \\[0.3em] \mathsf{p}
  \end{bmatrix}
  =
  \begin{bmatrix}
    \mathsf{f} \\[0.3em] \mathsf{g}
  \end{bmatrix}.
\end{equation}

The key observation is that, in two dimensions, the matrix \( \mathsf{N} \) associated with the convection-diffusion operator is block diagonal: \( \mathsf{N} = \mathrm{diag}(\mathsf{N_1}, \mathsf{N_2}) \), and writing  $\mathsf{B} = \begin{bmatrix}
    \mathsf{B_1} & \mathsf{B_2}
  \end{bmatrix}$ , the saddle-point matrix in~\eqref{eqn:oseen} can be rewritten as
\begin{equation}\label{eqn:mat_oseen}
  \mathcal{A} =
  \begin{bmatrix}
    \mathsf{N_1} & 0            & \mathsf{B_1^T} \\[0.3em]
    0            & \mathsf{N_2} & \mathsf{B_2^T} \\[0.3em]
    \mathsf{B_1} & \mathsf{B_2} & \mathsf{0}
  \end{bmatrix}.
\end{equation}

Therefore, the augmented \((1,1)\)-block can be written as
\[
  \mathsf{N_\gamma} =
  \mathsf{N} + \gamma \mathsf{B^T W^{-1} B}
  =
  \begin{bmatrix}
    \mathsf{N_1} + \gamma \mathsf{B_1^T W^{-1} B_1} & \gamma \mathsf{B_1^T W^{-1} B_2}                \\[0.3em]
    \gamma \mathsf{B_2^T W^{-1} B_1}                & \mathsf{N_2} + \gamma \mathsf{B_2^T W^{-1} B_2}
  \end{bmatrix}
  =
  \begin{bmatrix}
    \mathsf{N_{11}} & \mathsf{N_{12}} \\[0.3em] \mathsf{N_{21}} & \mathsf{N_{22}}
  \end{bmatrix},
\]where $\gamma$ is a positive scalar and $\mathsf{W}$ is again a suitable SPD matrix. The \emph{modified} AL preconditioner is then defined as the following block-triangular matrix:
\[
  \widetilde{\mathcal{P}}_\gamma =
  \begin{bmatrix}
    \mathsf{N_{11}} & \mathsf{N_{12}} & \mathsf{B_1^T}       \\[0.3em]
    0               & \mathsf{N_{22}} & \mathsf{B_2^T}       \\[0.3em]
    0               & 0               & \mathsf{\widehat{S}}
  \end{bmatrix},
\]
where \( \mathsf{\widehat{S}} \) is taken as $-\frac 1\gamma \mathsf{M_p}$, being $\mathsf{M_p}$ the classical pressure mass matrix~\cite{spectral_analysis_modified}. Due to the new block-triangular structure of \( \widetilde{\mathcal{P}}_\gamma \), most of the computational effort involved in applying \( \widetilde{\mathcal{P}}_\gamma^{-1} \) to a vector reduces to solving two linear systems with matrices \( \mathsf{N_{11}} \) and \( \mathsf{N_{22}} \).
These two subsystems can be solved either exactly or approximately by means of iterative methods. This approach also extends naturally to three-dimensional problems.

As pointed out in~\cite{modALprec}, the fact that the modified AL preconditioner is derived from the ideal one by neglecting a block of the form \( \gamma \mathsf{B_2^T W^{-1} B_1} \) suggests that the parameter \( \gamma \) should not be chosen excessively large. Although the preconditioner
is robust with respect to mesh refinement (see, e.g., the analysis in~\cite{FoVAL}), it still exhibits some dependence on the viscosity parameter.

\subsection{Modified AL preconditioner for elliptic interface problem}

Having introduced the idea behind the modified AL approach, we observe that the matrix $\mathcal{A}$ in~\eqref{eqn:mat_oseen} shares the same block structure as our system matrix in~\eqref{eqn:matrix}. Hence, the natural structure of the preconditioner looks as follows:
\begin{equation}\label{eqn:mod_gammaid}
  \mathcal{\widetilde P}_{\gamma} =
  \begin{bmatrix}
    \mathsf{A +  \gamma C^T W^{-1}C} & \mathsf{-\gamma C^T W^{-1}C_2}         & \mathsf{C^T}     \\
    0                                & \mathsf{A_2 +  \gamma C_2^T W^{-1}C_2} & \mathsf{-C_2^T}  \\
    0                                & 0                                      & \mathsf{\hat{S}}
  \end{bmatrix}.
\end{equation}
We note that, in our problem, the matrices \( \mathsf{A} \) and \( \mathsf{A_2} \) correspond to two \emph{different} equations and are not associated with a block partitioning of unknowns for a single discrete operator, as was the case in the Oseen problem. The same applies to \( \mathsf{C} \) and \( \mathsf{C_2} \), which in our problem denote two distinct operators, rather than two blocks of the same discrete operator in each coordinate direction, as $\mathsf{B_1}$ and $\mathsf{B_2}$.
This allows us to consider a different way of augmenting them. That is, we augment the first and second rows of the system with positive but different parameters $\gamma_1$ and $\gamma_2$, respectively. This yields the following augmented system:

\begin{equation}
  \begin{bmatrix}\label{eqn:ALmatrix_mod}
    \mathsf{A +  \gamma_1 C^T W^{-1}C} & \mathsf{-\gamma_1 C^T W^{-1}C_2}         & \mathsf{C^T}    \\
    \mathsf{-\gamma_2 C_2^T W^{-1}C}   & \mathsf{A_2 +  \gamma_2 C_2^T W^{-1}C_2} & \mathsf{-C_2^T} \\
    \mathsf{C}                         & \mathsf{-C_2}                            & 0
  \end{bmatrix}
  \begin{bmatrix}
    \mathsf{u}   \\
    \mathsf{u_2} \\
    \mathsf{\lambda}
  \end{bmatrix}=
  \begin{bmatrix}
    \mathsf{f} \\
    \mathsf{g} \\
    \mathsf{0}
  \end{bmatrix},
\end{equation} for which we propose the following \emph{modified} AL preconditioner:
\begin{equation}\label{eqn:P_AL_mod_gamma_12}
  \mathcal{\widetilde P}_{\gamma_1,\gamma_2}\coloneqq
  \begin{bmatrix}
    \mathsf{A +  \gamma_1 C^T W^{-1}C} & \mathsf{-\gamma_1 C^T W^{-1}C_2}         & \mathsf{C^T}     \\
    0                                  & \mathsf{A_2 +  \gamma_2 C_2^T W^{-1}C_2} & \mathsf{-C_2^T}  \\
    0                                  & 0                                        & \mathsf{\hat{S}}
  \end{bmatrix}.
\end{equation}
The selection of $\mathsf{\hat{S}}$ will be detailed in the spectral analysis presented in Section~\ref{sec:spectral_modified}, where we demonstrate that an appropriate choice is $\mathsf{\hat{S}} = -\frac{1}{\gamma_1}\mathsf{W}$.
Note that the $(2,2)$-block in~\eqref{eqn:mod_gammaid} and~\eqref{eqn:P_AL_mod_gamma_12} is nonsingular, owing to the fact that $\mathsf{A_2}$ is positive definite on $\ker(\mathsf{C_2})$ (cf. Remark~\ref{rmk:singularity}).

The main benefit of this new variant is that we can now select \emph{different} values for the augmentation parameters $\gamma_1$ and $\gamma_2$. In particular, $\gamma_1$ can be taken sufficiently large to ensure the strong eigenvalue clustering inherited from the ideal AL preconditioner, while $\gamma_2$ can remain small, thereby mitigating the effect of dropping the $(2,1)$-block while stabilizing the semidefinite block associated with $\mathsf{A_2}$.

However, note that, in practice, the choice of the parameters $\gamma_1$ and $\gamma_2$ is guided by a tradeoff between the outer convergence and the cost of approximately inverting the augmented diagonal blocks: just as $\gamma_1$ cannot be made too large, since this would render the augmented $(1,1)$-block in~\eqref{eqn:P_AL_mod_gamma_12} increasingly ill-conditioned, $\gamma_2$ likewise cannot be taken too small, as this would make the augmented $(2,2)$-block nearly singular, due to the semidefiniteness of $\mathsf{A_2}$. We will see in the numerical experiments that, in contrast with classical modified approaches (such as in~\cite{modALprec,spectral_analysis_modified}), we do not need to tune these parameters upon mesh refinement. The values adopted in Section~\ref{sec:numerical_experiments} should be regarded as robust empirical choices for the class of problems considered, rather than as universally optimal tuning parameters.

\section{Spectral analysis for the modified preconditioner}\label{sec:spectral_modified}
In this section, we perform a spectral analysis for the modified AL preconditioner introduced in~\eqref{eqn:P_AL_mod_gamma_12}. Our derivation closely follows the spectral analysis developed in~\cite{spectral_analysis_modified} for the modified AL preconditioner applied to the steady-state Navier-Stokes equations. We start by recalling the AL formulation~\eqref{eqn:ALmatrix_mod}:
\begin{equation}\label{eqn:aug_gammadiv}
  \mathcal{A}_{\gamma_1,\gamma_2}\coloneqq
  \begin{bmatrix}
    \mathsf{A +  \gamma_1 C^T W^{-1}C} & \mathsf{-\gamma_1 C^T W^{-1}C_2}         & \mathsf{C^T}    \\
    \mathsf{-\gamma_2 C_2^T W^{-1}C}   & \mathsf{A_2 +  \gamma_2 C_2^T W^{-1}C_2} & \mathsf{-C_2^T} \\
    \mathsf{C}                         & \mathsf{-C_2}                            & 0
  \end{bmatrix}
  =
  \begin{bmatrix}
    \mathsf{A_{11}} & \mathsf{A_{12}} & \mathsf{C^T}    \\
    \mathsf{A_{21}} & \mathsf{A_{22}} & \mathsf{-C_2^T} \\
    \mathsf{C}      & \mathsf{-C_2}   & 0
  \end{bmatrix},
\end{equation}
and the associated modified AL preconditioner~\eqref{eqn:P_AL_mod_gamma_12} when different parameters $\gamma_1$ and $\gamma_2$ are used:
\begin{equation}
  \mathcal{\widetilde P}_{\gamma_1,\gamma_2} =
  \begin{bmatrix}
    \mathsf{A_{11}} & \mathsf{A_{12}} & \mathsf{C^T}     \\
    0               & \mathsf{A_{22}} & \mathsf{-C_2^T}  \\
    0               & 0               & \mathsf{\hat{S}}
  \end{bmatrix}=
  \begin{bmatrix}
    \widetilde{\mathsf{A}}_{\gamma_1,\gamma_2} & \mathsf{B^T}     \\
    0                                          & \mathsf{\hat{S}}
  \end{bmatrix}.
\end{equation}
The block upper-triangular structure of $\mathcal{\widetilde P}_{\gamma_1,\gamma_2}$ yields the following factorization of its inverse:
\begin{equation}\label{eqn:fact}
  \mathcal{\widetilde P}_{\gamma_1,\gamma_2}^{-1} =
  \begin{bmatrix}
    \widetilde{\mathsf{A}}_{\gamma_1,\gamma_2}^{-1} & 0               \\
    0                                               & \mathsf{I}_\ell
  \end{bmatrix}
  \begin{bmatrix}
    \mathsf{I}_{n+m} & \mathsf{B^T}     \\
    0                & \mathsf{-I}_\ell
  \end{bmatrix}
  \begin{bmatrix}
    \mathsf{I}_{n+m} & 0                      \\
    0                & -\hat{\mathsf{S}}^{-1}
  \end{bmatrix},
\end{equation}
where
\[
  \widetilde{\mathsf{A}}_{\gamma_1,\gamma_2} \coloneqq
  \begin{bmatrix}
    \mathsf{A} + \gamma_1 \mathsf{C^T W^{-1} C} & -\gamma_1 \mathsf{C^T W^{-1} C_2}                 \\[0.3em]
    0                                           & \mathsf{A_2} + \gamma_2 \mathsf{C_2^T W^{-1} C_2}
  \end{bmatrix}.
\]

\vspace{0.1cm}
The result stated in the lemma below is proved in Appendix~\ref{sec:appendix_proofs}. Notice that all the necessary matrices are invertible (cf. Remark~\ref{rmk:choice_W}).
\begin{lemma}\label{lemma:prec_mat}
  The right-preconditioned matrix has the following block structure:
  \begin{equation}\label{eqn:apinv}
    \mathcal{A}_{\gamma_1,\gamma_2} \mathcal{\widetilde P}_{\gamma_1,\gamma_2}^{-1} = \begin{bmatrix}
      \mathsf{I}_n                & 0                                                          & 0               \\
      \mathsf{A_{21} A_{11}^{-1}} & \mathsf{I}_m - \mathsf{A_{21}A_{11}^{-1}A_{12}A_{22}^{-1}} & \mathsf{T_{23}} \\
      \mathsf{C A_{11}^{-1}}      & -\mathsf{C A_{11}^{-1}A_{12}A_{22}^{-1}-C_2 A_{22}^{-1}}   & \mathsf{T_{33}}
    \end{bmatrix}
  \end{equation}
  where
  \par\vspace{0.2cm}
  \noindent
  $\mathsf{T_{23}} = (\mathsf{-A_{21}A_{11}^{-1}C^T-A_{21}A_{11}^{-1}A_{12}A_{22}^{-1}C_2^T})\hat{\mathsf{S}}^{-1}$,
  $\mathsf{T_{33}} = -(\mathsf{C A_{11}^{-1}C^T+C A_{11}^{-1}A_{12}A_{22}^{-1}C_2^T + C_2 A_{22}^{-1}C_2^T})\hat{\mathsf{S}}^{-1}$.

\end{lemma}
\par\vspace{0.3cm}

In order to proceed further, we need the following lemma, whose proof can also be found in Appendix~\ref{sec:appendix_proofs}.
\begin{lemma}\label{lemma:SM}
  The following identity holds:
  \begin{equation}\label{eqn:identity}
    \gamma_1 \mathsf{CA_{11}^{-1}C^TW^{-1}}=\mathsf{I}_\ell-\left( \mathsf{I}_\ell+\gamma_1 \mathsf{C A^{-1}C^TW^{-1}}\right)^{-1}.
  \end{equation}
\end{lemma}
Applying~\eqref{eqn:identity} to the preconditioned matrix in~\eqref{eqn:apinv}, we obtain the following result, whose proof is given in Appendix~\ref{sec:appendix_proofs}.
\begin{lemma}\label{lemma:dfge}
  Letting
  \begin{align*}
    \mathsf{D} & = \mathsf{\gamma_2 C_2^T W^{-1}}\bigl(\mathsf{I}_\ell-\bigl(\mathsf{I}_\ell+\gamma_1 \mathsf{C A^{-1} C^T W^{-1}}\bigr)^{-1}\bigr), \\[4pt]
    \mathsf{E} & = \mathsf{C_2 A_{22}^{-1}},                                                                                                         \\[4pt]
    \mathsf{G} & = \mathsf{I}_\ell - \gamma_1 \mathsf{C_2 A_{22}^{-1} C_2^T W^{-1}},                                                                 \\[4pt]
    \mathsf{F} & = \bigl(\mathsf{I}_\ell + \gamma_1 \mathsf{C A^{-1} C^T W^{-1}}\bigr)^{-1},
  \end{align*}
  the right-preconditioned matrix can be written as:
  \begin{equation}\label{eqn_apinvdefg}
    \mathcal{A}_{\gamma_1,\gamma_2} \mathcal{\widetilde P}_{\gamma_1,\gamma_2}^{-1} = \begin{bmatrix}
      \mathsf{I}_n                & 0                        & 0                                                                                                                 \\
      \mathsf{A_{21} A_{11}^{-1}} & \mathsf{I}_m-\mathsf{DE} & \mathsf{DG\left(\frac{1}{\gamma_1}W\right)}\hat{\mathsf{S}}^{\mathsf{-1}}                                         \\
      \mathsf{C A_{11}^{-1}}      & -\mathsf{FE}             & \left(\mathsf{I}_\ell-\mathsf{FG}\right) \mathsf{\left(-\frac{1}{\gamma_1}W\right)}\hat{\mathsf{S}}^{\mathsf{-1}}
    \end{bmatrix}.
  \end{equation}
\end{lemma}
The foregoing lemma suggests the choice $\hat{\mathsf{S}}^{\mathsf{-1}}=-\mathsf{\gamma_1 W^{-1}}$. This leads to the following theorem, which provides
a characterization of the spectrum of the preconditioned system with the modified variant, and resembles the findings obtained in~\cite{spectral_analysis_modified}.

\begin{theorem}[Spectrum of preconditioned system with modified variant]\label{thm:nonunit_modified}
  Let $\hat{\mathsf{S}}^{\mathsf{-1}}=-\mathsf{\gamma_1 W^{-1}}$. We obtain
  \begin{equation}\label{eqn:apinvdefg}
    \mathcal{A}_{\gamma_1,\gamma_2} \mathcal{\widetilde P}_{\gamma_1,\gamma_2}^{-1} = \begin{bmatrix}
      \mathsf{I}_n                & 0                        & 0                           \\
      \mathsf{A_{21} A_{11}^{-1}} & \mathsf{I}_m-\mathsf{DE} & \mathsf{-DG}                \\
      \mathsf{C A_{11}^{-1}}      & -\mathsf{FE}             & \mathsf{I}_\ell-\mathsf{FG}
    \end{bmatrix}.
  \end{equation}
  The matrix $\mathcal{A}_{\gamma_1,\gamma_2} \mathcal{\widetilde P}_{\gamma_1,\gamma_2}^{-1}$ has $1$ as an eigenvalue of algebraic multiplicity at least $n+m$. The remaining eigenvalues are the non-unit eigenvalues of the matrix
  \begin{equation}\label{eqn:lasteig}
    \begin{bmatrix}
      \mathsf{I}_m - \mathsf{DE} & \mathsf{-DG}                \\
      \mathsf{-FE}               & \mathsf{I}_\ell-\mathsf{FG}
    \end{bmatrix}.
  \end{equation}
\end{theorem}

\vspace{0.2cm}
\begin{proof}
  Equality~\eqref{eqn:apinvdefg} immediately follows from~\eqref{eqn_apinvdefg} by substitution. Furthermore, the spectrum of $\mathcal{A}_{\gamma_1,\gamma_2} \widetilde{\mathcal{P}}_{\mathsf{\gamma_1,\gamma_2}}^{-1}$ consists of the eigenvalue $1$ of multiplicity $n$, plus the spectrum of the matrix
  $$\begin{bmatrix}
      \mathsf{I}_m - \mathsf{DE} & \mathsf{-DG}                \\
      \mathsf{-FE}               & \mathsf{I}_\ell-\mathsf{FG}
    \end{bmatrix}=
    \begin{bmatrix}
      \mathsf{I}_m & 0               \\
      0            & \mathsf{I}_\ell
    \end{bmatrix}-
    \begin{bmatrix}
      \mathsf{DE} & \mathsf{DG} \\
      \mathsf{FE} & \mathsf{FG}
    \end{bmatrix}.$$
  Observing that
  \begin{equation}\label{eqn:smallblock_DEDG}
    \begin{bmatrix}
      \mathsf{DE} & \mathsf{DG} \\
      \mathsf{FE} & \mathsf{FG}
    \end{bmatrix}=
    \begin{bmatrix}
      \mathsf{D} \\
      \mathsf{F}
    \end{bmatrix} \begin{bmatrix}
      \mathsf{E} &
      \mathsf{G}
    \end{bmatrix},
  \end{equation}
  we have
  $$\text{rank}\begin{bmatrix}
      \mathsf{DE} & \mathsf{DG} \\
      \mathsf{FE} & \mathsf{FG}
    \end{bmatrix}\le \min \left( \text{rank}
    \begin{bmatrix}
      \mathsf{D} \\
      \mathsf{F}
    \end{bmatrix}, \text{rank} \begin{bmatrix}
      \mathsf{E} &
      \mathsf{G}
    \end{bmatrix}\right) \le \ell,$$
  because $\begin{bmatrix}
      \mathsf{D} \\
      \mathsf{F}
    \end{bmatrix} \in \mathbb{R}^{(m+\ell)\times \ell}$ and $\begin{bmatrix}
      \mathsf{E} &
      \mathsf{G}
    \end{bmatrix} \in \mathbb{R}^{\ell \times (m+\ell)}$. This implies that the matrix~\eqref{eqn:lasteig} has the eigenvalue $1$ of multiplicity at least $m$. Therefore, $\mathcal{A}_{\gamma_1,\gamma_2} \widetilde{\mathcal{P}}_{\mathsf{\gamma_1,\gamma_2}}^{-1}$ has $1$ as an eigenvalue of algebraic multiplicity at least $n+m$.
\end{proof}
\vspace{0.2cm}
 Unlike the ideal AL preconditioner supported by rigorous mesh-independent spectral bounds in Section~\ref{sec:spectral_ideal}, the analysis of the modified variant focuses on asymptotic spectral behavior with respect to the augmentation parameters. The mesh-independent behavior of the practical modified variant is supported by the numerical experiments in Section~\ref{sec:numerical_experiments}.

\subsection{Asymptotic behavior of remaining eigenvalues}
Similar to the case of the ideal preconditioner developed in Section~\ref{sec:AL_prec}, we set $\mathsf{W} = \mathsf{M}^2$. In this case, some observations can be made regarding the asymptotic behavior of the remaining $\ell$ eigenvalues of matrix~\eqref{eqn:lasteig} that differ from $1$. So far, our analysis has considered the general case $V_{2,h} \not\equiv \Lambda_h$. We now specialize to the case $V_{2,h} \equiv \Lambda_h$ to gain further insight into the eigenvalue clustering of the modified approach in a simplified setting of practical interest that allows for a more explicit analysis. Hence, we have $\mathsf{C_2} = \mathsf{M}$, and consequently $\mathsf{C_2^T W^{-1} C_2} = \mathsf{I}_m.$ The spectral analysis below is therefore stated for this algebraic setting. We omit the subscript in the identity matrix, since $m=\ell$. Note that the numerical experiments reported in Section~\ref{subsec:diff_fem} indicate that the qualitative behavior predicted by the simplified analysis carries over in practice also to the more general case $V_{2,h}\not\equiv\Lambda_h$.

As shown in Section~\ref{sec:spectral_ideal}, for the ideal AL preconditioner all the eigenvalues of the preconditioned matrix cluster towards $1$ as $\gamma \!\rightarrow\! \infty$. This is not generally the case for the modified variants~\eqref{eqn:mod_gammaid} and~\eqref{eqn:P_AL_mod_gamma_12}. However, the following result holds for the latter.

\begin{theorem}[Spectral convergence of the non-zero eigenvalues of the matrix~\eqref{eqn:smallblock_DEDG}]\label{thm:spectral_conv}
  Let $\mathsf{W} = \mathsf{M}^2$, and let $V_h$, $V_{2,h}$, and $\Lambda_h$ be continuous Lagrangian finite element spaces. Then, the reciprocals of the non-zero eigenvalues of the matrix~\eqref{eqn:smallblock_DEDG} converge to the eigenvalues of $\mathsf{-LA_2}$ as $\gamma_1 \rightarrow \infty$ and $\gamma_2 \rightarrow 0$, where $\mathsf{L\coloneqq M^{-1} C A^{-1} C^T M^{-1}}$.
\end{theorem}
\begin{proof}
  By choosing $\mathsf{W} = \mathsf{M}^2$ and selecting the finite element spaces so that $\mathsf{C}_2 = \mathsf{M}$, the matrices $\mathsf{D}$, $\mathsf{E}$, $\mathsf{G}$, and $\mathsf{F}$ introduced in Lemma~\ref{lemma:dfge} can be rewritten as follows:
  \begin{align*}
    \mathsf{D} & = \gamma_2 \mathsf{M^{-1}}\bigl(\mathsf{I}-\bigl(\mathsf{I}+\gamma_1 \mathsf{C A^{-1} C^T M^{-2}}\bigr)^{-1}\bigr), \\[4pt]
    \mathsf{E} & = \mathsf{M A_{22}^{-1}},                                                                                           \\[4pt]
    \mathsf{G} & = \mathsf{I} - \gamma_1 \mathsf{M A_{22}^{-1} M^{-1}},                                                              \\[4pt]
    \mathsf{F} & = \bigl(\mathsf{I} + \gamma_1 \mathsf{C A^{-1} C^T M^{-2}}\bigr)^{-1}.
  \end{align*}
  Note that all the necessary inverses exist. As observed in Theorem~\ref{thm:nonunit_modified}, the matrix~\eqref{eqn:smallblock_DEDG} has rank at most $\ell$, and its non-zero eigenvalues coincide with those of the smaller matrix
  $$\begin{bmatrix}
      \mathsf{E} &
      \mathsf{G}
    \end{bmatrix}\begin{bmatrix}
      \mathsf{D} \\
      \mathsf{F}
    \end{bmatrix} = \mathsf{ED + GF}.$$
  Substituting the expressions for $\mathsf{D}$, $\mathsf{E}$, $\mathsf{G}$, and $\mathsf{F}$ yields:
  {\small$$\mathsf{ED + GF} = \gamma_2 \mathsf{M A_{22}^{-1}} \mathsf{M^{-1}} \bigl(\mathsf{I}-\bigl(\mathsf{I}+\gamma_1 \mathsf{C A^{-1} C^T M^{-2}}\bigr)^{-1}\bigr) + \bigl(\mathsf{I} - \gamma_1 \mathsf{M A_{22}^{-1} M^{-1}}\bigr) \bigl(\mathsf{I} + \gamma_1 \mathsf{C A^{-1} C^T M^{-2}}\bigr)^{-1}.$$}
  Using the fact that $\mathsf{A_{22}=A_2 + \gamma_2 I}$ and performing some simple algebra, we obtain:
  {\small$$\mathsf{ED + GF} = \gamma_2 \mathsf{M \bigl(A_2 + \gamma_2 I \bigr)^{-1}} \mathsf{ M^{-1}} -  \bigl(\gamma_1+\gamma_2 \bigr) \mathsf{M \bigl(A_2 + \gamma_2 I \bigr)^{-1}} \mathsf{ M^{-1}} \bigl(\mathsf{I}+\gamma_1 \mathsf{C A^{-1} C^T M^{-2}}\bigr)^{-1} + \bigl(\mathsf{I}+\gamma_1 \mathsf{C A^{-1} C^T M^{-2}}\bigr)^{-1}.$$}
  Note that
    {\small$$\bigl(\mathsf{I}+\gamma_1 \mathsf{C A^{-1} C^T M^{-2}}\bigr)^{-1} = \bigl( \mathsf{M}\bigl(\mathsf{I}+\gamma_1 \mathsf{M^{-1}C A^{-1} C^T M^{-1}}\bigr)\mathsf{M^{-1}}\bigr)^{-1}= \mathsf{M}\bigl(\mathsf{I}+\gamma_1 \mathsf{M^{-1}C A^{-1} C^T M^{-1}}\bigr)^{-1}\mathsf{M^{-1}},$$}
  and substituting into the expression for $\mathsf{ED + GF}$ gives:
  {\small$$\mathsf{ED + GF} = \mathsf{M}\Bigl(\gamma_2 \bigl(\mathsf{A_2 + \gamma_2 I} \bigr)^{-1} -  \bigl(\gamma_1+\gamma_2 \bigr) \mathsf{\bigl(A_2 + \gamma_2 I \bigr)^{-1}}  \bigl(\mathsf{I}+\gamma_1 \mathsf{M^{-1}} \mathsf{C A^{-1} C^T M^{-1}}\bigr)^{-1} + \bigl(\mathsf{I}+\gamma_1 \mathsf{M^{-1}} \mathsf{C A^{-1} C^T M^{-1}}\bigr)^{-1} \Bigr) \mathsf{M}^{-1}.$$}
  Therefore, we can conclude that $\mathsf{ED + GF}$ is similar to the matrix
    {\small$$ \gamma_2 \bigl(\mathsf{A_2 + \gamma_2 I} \bigr)^{-1} -  \bigl(\gamma_1+\gamma_2 \bigr) \mathsf{\bigl(A_2 + \gamma_2 I \bigr)^{-1}}  \bigl(\mathsf{I}+\gamma_1 \mathsf{M^{-1}C A^{-1} C^T M^{-1}}\bigr)^{-1} + \bigl(\mathsf{I}+\gamma_1 \mathsf{M^{-1}C A^{-1} C^T M^{-1}}\bigr)^{-1}.$$}
  For the arguments that follow, it is convenient to consider the inverse of this matrix. By performing straightforward algebraic manipulations, we obtain that this inverse can be expressed as:
  $$\mathsf{N}\coloneqq\left(\mathsf{L + \frac{1}{\gamma_1} I} \right) \left(\gamma_2 \left(\mathsf{L + \frac{1}{\gamma_1} I} \right) + \frac{1}{\gamma_1} \bigl(\mathsf{A_2+\gamma_2 I}\bigr) - \frac{\gamma_1 + \gamma_2}{\gamma_1} \>\mathsf{I} \right)^{-1} \bigl(\mathsf{A_2+\gamma_2 I}\bigr),$$
  where $\mathsf{L\coloneqq M^{-1} C A^{-1} C^T M^{-1}}$. It is evident that this matrix converges entrywise to $-\mathsf{L A_2}$ as
  $\gamma_1 \rightarrow \infty$ and $\gamma_2 \rightarrow 0$.

  To conclude the proof, we use the well-known fact that the eigenvalues of a square matrix depend continuously on its entries~\cite{horn}. In other words, if we consider an infinite sequence of matrices $\mathsf{N_k}$ of dimensions $m\times m$, which converges to $\mathsf{-LA_2}$, we can conclude that, among the up to $m!$ possible ways of ordering the eigenvalues of the matrices $\mathsf{N_{k_1}, N_{k_2}}, \dots$ as a $m$-dimensional vector, there exists at least one ordering for each matrix such that the corresponding eigenvalue vectors converge to a vector whose components consist precisely of all the eigenvalues of $\mathsf{-LA_2}$. Since $\mathsf{ED + GF}$ is similar to the inverse of $\mathsf{N}$, this concludes the proof.

\end{proof}

\begin{remark}[Scaling considerations on $\mathsf{LA_2}$]\label{rmk:scalingLA2}
  Notice that $\mathsf{L A_2}$ is the product of the symmetric and positive definite matrix $\mathsf{M^{-1} C A^{-1} C^T M^{-1}}$
  and the symmetric positive semidefinite matrix $\mathsf{A_2}$. Hence, $\mathsf{LA_2}$ is similar to a symmetric positive semidefinite matrix and therefore has real and positive eigenvalues, except for a single zero
  eigenvalue arising from the rank deficiency of $\mathsf{A_2}$.
  Moreover, using well-known scaling properties of FEM matrices, we conclude that the entries of $\mathsf{L}\mathsf{A}_2$ scale as
  $\frac{\beta_2 - \beta}{\beta}\,\bigl(\frac{h_{\Omega_2}}{h_{\Omega}}\bigr)^{d-2}.$
  Notice that the two meshes are in practice chosen so that the ratio $\frac{h_{\Omega_2}}{h_{\Omega}} \in \mathcal{O}(1)$, while the ratio $\frac{\beta_2 - \beta}{\beta}$ is larger than one thanks to our assumptions. Consequently, the spectrum of $\mathsf{LA_2}$ moves away from $1$ as the jump increases.
\end{remark}

Theorem~\ref{thm:spectral_conv} provides insight into the behavior of the non-unit eigenvalues of the preconditioned matrix $\mathcal{A}_{\gamma_1,\gamma_2} \mathcal{\widetilde P}_{\gamma_1,\gamma_2}^{-1}$, which coincide with the non-unit eigenvalues of the matrix~\eqref{eqn:lasteig}.
As observed in Theorem~\ref{thm:nonunit_modified}, the matrix~\eqref{eqn:lasteig} can be expressed as:
$$\begin{bmatrix}
    \mathsf{I} - \mathsf{DE} & \mathsf{-DG}           \\
    \mathsf{-FE}             & \mathsf{I}-\mathsf{FG}
  \end{bmatrix}=
  \begin{bmatrix}
    \mathsf{I} & 0          \\
    0          & \mathsf{I}
  \end{bmatrix}-
  \begin{bmatrix}
    \mathsf{DE} & \mathsf{DG} \\
    \mathsf{FE} & \mathsf{FG}
  \end{bmatrix},$$
implying that its eigenvalues are of the form $1-\nu$, where $\nu$ are the eigenvalues of~\eqref{eqn:smallblock_DEDG}.
From Theorem~\ref{thm:spectral_conv}, it follows that the non-zero eigenvalues of~\eqref{eqn:smallblock_DEDG} tend to become real and negative, and to converge to zero as the jump increases, except for a single outlier that diverges to infinity. This behavior is advantageous, since the resulting eigenvalues $1-\nu$ are tightly clustered around $1$.
This observation also explains the improved iteration counts observed for larger jumps in Section~\ref{sec:numerical_experiments}. Moreover, note that the presence of one outlier does not deteriorate the convergence behavior of GMRES. Indeed, an isolated eigenvalue lying outside the main spectral cluster and having a larger magnitude than the remaining eigenvalues, as in our case, is expected to incur only a modest penalty. By contrast, outliers that lie much closer to the origin may significantly slow down convergence~\cite{trefethen}.

It is also worth noting that a closer examination of the matrix $\mathsf{N}$ sheds light on why it is preferable to choose $\gamma_1 \neq \gamma_2$ and in particular why it is convenient to select $\gamma_1$ large and $\gamma_2$ small, rather than using $\gamma_1 = \gamma_2$, which yields the same configuration analyzed in~\cite{spectral_analysis_modified}. Indeed, if one considers the case $\gamma_1 = \gamma_2$ and lets them tend either to $0$ or to $+\infty$, the structure of the matrix $\mathsf{N}$ shows that certain terms become unbounded.  Hence, the previous conclusions do not hold in this regime.

The limits $\gamma_1\to\infty$ and $\gamma_2\to0$ should be interpreted as an explanation of the observed spectral mechanism, rather than as a practical prescription. In computations, $\gamma_1$ must remain moderate enough to keep the $(1,1)$-block solvable by the chosen inner method, while $\gamma_2$ should remain small but sufficiently far from zero to effectively regularize the singular immersed-domain block. Similarly to the ideal case, fixed finite values of these parameters are in practice sufficient.

\subsection{Numerical test on remaining eigenvalues}
We validate the theoretical results by computing the eigenvalues of the lower two-by-two block \[\begin{bmatrix}\label{eqn:lower_block}
    \mathsf{I}_m - \mathsf{DE} & \mathsf{-DG}                \\
    \mathsf{-FE}               & \mathsf{I}_\ell-\mathsf{FG}
  \end{bmatrix}\]
when $\gamma_1 \!\rightarrow \! \infty$ and $\gamma_2 \! \rightarrow \! 0$. To this aim, we consider the same configuration used in the experiments in Section~\ref{sec:numerical_validation}. The size of the whole lower block associated with such a discretization is $m+\ell = 81+81=162$. As shown in the previous analysis, $\nu=1$ is an eigenvalue of~\eqref{eqn:lasteig} with algebraic multiplicity $m=81$. The spectrum is displayed in Figure~\ref{fig:mod_spectrum_beta_100}.
In addition to the expected $81$ eigenvalues at $1$, the rest of them progressively become real and move away from $0$ as $\gamma_1$ increases and $\gamma_2$ decreases, in agreement with the fact that the spectrum of $\mathsf{-LA_2}$ is real and negative. In the second and third plots the outlier is omitted for visualization reasons; when $\gamma_1 = 10$ and $\gamma_2 = 10^{-2}$, the outlier is real and located at $Re(\nu) \approx 6.7$, whereas for $\gamma_1 = 10^2$ and $\gamma_2 = 10^{-3}$ it is
$Re(\nu)\approx68.6$.

Although we do not report the corresponding plots for larger jumps, the observed trend is that increasing the jump yields a tighter clustering around $1$.
\begin{figure}[h]
  \begin{subfigure}{0.3\textwidth} %
    \centering
    \includegraphics{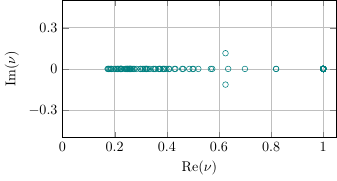}
    \caption*{\hspace{2.0cm} $\gamma_1\! =\! 1$, $\gamma_2\! =\! 10^{-1}$ }
  \end{subfigure}\hspace{1.1cm} %
  \begin{subfigure}{0.3\textwidth}
    \centering
    \includegraphics{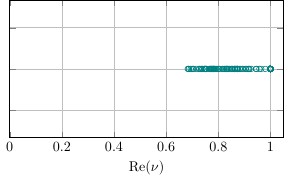}
    \caption*{$\gamma_1\! =\! 10$, $\gamma_2\! =\! 10^{-2}$}
  \end{subfigure}\hspace{0.05cm} %
  \begin{subfigure}{0.3\textwidth}
    \centering
    \includegraphics{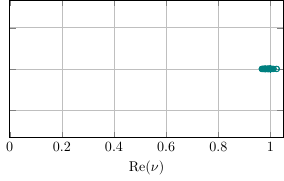}
    \caption*{$\gamma_1\! =\! 10^2$, $\gamma_2\! =\! 10^{-3}$}
  \end{subfigure}
  \vspace{0.1cm}
  \caption{$\beta=1$, $\beta_2 = 100$. Spectrum of the lower block~\eqref{eqn:lower_block} for increasing $\gamma_1$ and decreasing $\gamma_2$.}
  \label{fig:mod_spectrum_beta_100}
\end{figure}

\FloatBarrier
\section{Numerical experiments}
\label{sec:numerical_experiments}

In this section, we present numerical experiments aimed at evaluating the performance of the
proposed augmented Lagrangian preconditioners for the augmented linear system of equations~\eqref{eqn:ALmatrix_mod}, stemming from the finite element
discretization of Problem~\eqref{prob:elliptic_interface}.
All the experiments have been performed using the \textsc{C++} finite element library \textsc{deal.II}~\cite{dealII95,dealIIdesign} and
are available at the dedicated GitHub repository provided by the authors\footnote{Available at \url{https://github.com/fdrmrc/fictitious_domain_AL_preconditioners}.}. All computations have been carried out on a laptop equipped with an AMD Ryzen 9 PRO 7940HS processor and 64 GB of RAM, running Ubuntu 22.04 LTS.

Throughout the following examples, unless otherwise stated, we employ $\mathcal{Q}^1$ Lagrangian finite elements for all finite element spaces $V_h, V_{2,h},$ and $\Lambda_h$, on
background and immersed meshes made of quadrilaterals or hexahedra. In particular, the background domain $\Omega$ is always
discretized using a structured mesh $\mathcal{T}_h$.

The outer Krylov subspace solver is set to FGMRES(30) with a restart every 30 iterations. All iterations start from a zero initial guess and are stopped when the relative residual norm is reduced below $10^{-10}$, or when its absolute value is below $10^{-10}$. We notice that we have chosen a rather strict tolerance compared to other papers analyzing preconditioners, which
leads to higher iteration numbers. We have chosen this convergence criterion because it makes seeing trends easier due to the higher number of outer iterations. Right preconditioning is used in all cases. We will investigate AL and modified AL (MAL) preconditioners by
varying \emph{both} the immersed geometries $\Omega_2$ and the size of the jump in the coefficients ($\beta_2 - \beta$), while monitoring the iteration counts across simultaneous mesh refinement of both grids.

We will start by confirming the theoretical findings for the AL preconditioner $\mathcal{P}_{\gamma}$ presented in Section~\ref{sec:AL_prec}. Then, we will analyze the
performance of the modified variant, showing its practical advantages in terms of reduced computational cost. Following Section~\ref{sec:AL_prec}, the SPD
matrix $\mathsf{W}$ is set to $\mathsf{M^2}$. When considering the modified AL preconditioner, sparse direct solvers could be used to solve the augmented subsystems in~\eqref{eqn:P_AL_mod_gamma_12} efficiently in 2D, while in the 3D case they quickly become prohibitively expensive, and inner iterative methods should be preferred in practice. For this reason, we will approximately solve the augmented subsystems using the conjugate-gradient (CG) method preconditioned by one V-cycle of AMG with a loose inner tolerance set to $10^{-2}$. We adopt the \textsc{TrilinosML}~\cite{Trilinos} implementation with
parameters reported in Table~\ref{tab:amg_params}. Moreover, to facilitate the construction of the AMG preconditioner, we replace $\mathsf{W}$ with $\operatorname{diag}(\mathsf{W})$, i.e., the matrix containing only the main diagonal of $\mathsf{W}$, which avoids the computationally expensive triple matrix product involving the dense matrix $\mathsf{W^{-1}}$. This modification is applied to both the system matrix and the preconditioner. When $\mathsf{W}$ is not replaced by its diagonal, its inverse is applied using a direct solver.

When the augmented diagonal blocks are approximately inverted and $\mathsf{W}$ is replaced by $\operatorname{diag}(\mathsf{W})$, the theory presented in the previous section no longer applies; this variant should be viewed as a practical approximation. We do not claim a separate complete theory for such approximation, although eigenvalue bounds can be derived following the analysis presented in~\cite{bakrani2023preconditioningtechniquesclassdouble}.
We denote such inexact variant of the ideal modified AL preconditioner by $\mathcal{\widetilde P}_{\gamma_1,\gamma_2}^{\text{diag}}$. Except where otherwise specified, the source terms are set to $f=1$ and $f_2=2$, while $\beta=1$, and the value of the extended $u$ on $\partial \Omega$ is set to zero. In the forthcoming tables, the number of degrees of freedom for a finite dimensional space will be denoted by $|\cdot|$, e.g. $|V_h|$. Throughout the numerical experiments, we monitor only the inner iteration counts of the $(1,1)$-block, which is particularly challenging due to the issues arising from the augmentation with the coupling matrix $\mathsf{C}$.

Table~\ref{tab:preconditioner_variants} summarizes the preconditioner variants considered in this section. The numerical experiments are performed with $\mathcal{P}_\gamma$, $\mathcal{P}_\gamma^{\mathrm{inexact}}$, and $\mathcal{\widetilde P}_{\gamma_1,\gamma_2}^{\mathrm{diag}}$. The first preconditioner is used to validate the theoretical results, whereas the latter two are intended for practical computations. Since the ideal MAL preconditioner is itself an adaptation of the ideal AL preconditioner, we do not report results for $\mathcal{\widetilde{P}}_{\gamma_1,\gamma_2}$, focusing instead on the practical variant $\mathcal{\widetilde P}_{\gamma_1,\gamma_2}^{\mathrm{diag}}$. 

\begin{table}[h!]
  \centering
  \small
  \begin{tabular}{l l c c}
    \toprule
    \textbf{Variant} & \textbf{Notation}                                            & \textbf{Augmented block(s) solve} & $\mathsf{W}$                                  \\
    \midrule
    Ideal AL         & $\mathcal{P}_\gamma$                                         & Exact                             & $\mathsf{M}^2$                                \\
    Inexact AL       & $\mathcal{P}_\gamma^{\mathrm{inexact}}$                      & CG + AMG                          & $\operatorname{diag}\bigl(\mathsf{M}^2\bigr)$ \\
    Ideal MAL        & $\mathcal{\widetilde{P}}_{\gamma_1,\gamma_2}$                & Exact                             & $\mathsf{M}^2$                                \\
    Inexact MAL      & $\mathcal{\widetilde P}_{\gamma_1,\gamma_2}^{\mathrm{diag}}$ & CG + AMG                          & $\operatorname{diag}\bigl(\mathsf{M}^2\bigr)$ \\
    \bottomrule
  \end{tabular}
  \caption{Summary of the preconditioner variants. The modified variants ($\mathcal{\widetilde{P}}_{\gamma_1,\gamma_2}$ and $\mathcal{\widetilde P}_{\gamma_1,\gamma_2}^{\mathrm{diag}}$) employ two AL parameters $\gamma_1, \gamma_2$ and a block-triangular structure. We will not test exact solves for the MAL case ($\mathcal{\widetilde{P}}_{\gamma_1,\gamma_2}$).}
  \label{tab:preconditioner_variants}
\end{table}

\subsection{Validation of the AL preconditioner}\label{subsec:ideal_validation}
We first consider a uniform discretization of the domain $\Omega_2 \coloneqq [-0.14,0.47]^2$, immersed in the background
domain $\Omega \coloneqq [-1,1]^2$, fixing $\beta=1$ and varying $\beta_2$ from $10$ to $10^7$. For each
value of $\beta_2$, we report in Table~\ref{tab:ideal_square} the number of outer iterations with the ideal AL preconditioner $\mathcal{P}_{\gamma}$ upon mesh refinement. We invert all
the diagonal blocks $\mathsf{A_\gamma}$ and $\mathsf{W}$ exactly. The augmentation parameter is set to $\gamma=10$.
We observe very low iteration counts, ranging from $4$ to $8$. These counts are independent \emph{both} of the refinement level, as well as of the size of the
jump in the coefficients, confirming the parameter robustness, and matching the spectral behaviour observed in Figures~\ref{fig:ideal_spectrum_beta_100} and~\ref{fig:ideal_spectrum_beta_1e6}.

We repeat the same experiment, this time considering a ball $\Omega_2 \coloneqq \mathcal{B}_{0.3}(0,0)$ as immersed domain. The results with this geometry are shown in Table~\ref{tab:ideal_ball}, leading to identical conclusions in terms of iteration counts and
robustness. Numerical solutions when $\beta_2=10^7$ for both geometries are shown in Figure~\ref{fig:ideal_solutions}, showing in a wireframe representation the immersed domains.

A classical and appealing feature of AL-based preconditioners is that they do not require accurate solves for the augmented block for effectiveness~\cite{ALprec,FoVAL}. We have hence verified that identical FGMRES iterations are obtained when applying an inexact inversion of $\mathsf{A}_\gamma$ via a conjugate gradient method, preconditioned by a block-diagonal preconditioner with one AMG V-cycle on each diagonal block and a loose inner tolerance of $10^{-2}$. This variant correspond to the preconditioner $\mathcal{P}_{\gamma}^{\mathrm{inexact}}$ listed in Table~\ref{tab:preconditioner_variants}, where $\mathsf{W}$ is replaced by $\operatorname{diag}(\mathsf{W})$. However, this inversion remains computationally expensive since the inner iterations grow unboundedly upon mesh refinement (exceeding an average of $70$ PCG iterations for the finest discretization). The current lack of a $\gamma$-robust and mesh-independent geometric multigrid solver for $\mathsf{A}_\gamma$ in the fictitious domain setting represents the main computational bottleneck for this (unmodified) version of the AL preconditioner.
\begin{table}[ht]
  \centering
  \begin{subtable}{0.48\textwidth}
    \centering
    \resizebox{\textwidth}{!}{%
      \begin{filecontents*}{data_ideal_iterations_square.csv}
#DoF                     #10 #1e3  #1e7 #N       #time
289 + 25+25,8,8,8,339,0.1073
1089 + 81+81,7,7,7,1251,0.09673
{4,225 +289 +289},6,7,7,4803,0.171
{16,641 + 1,089 + 1,089},6,6,6,18819,0.435
{66,049 + 4,225 + 4,225},5,5,5,74499,1.68
{263,169 + 16,641 + 16,641},4,5,5,296451,11.07
{1,050,625 + 66,049 + 66,049},4,4,4,1182723,57.47
{4,198,401 + 263,169 + 263,169},4,4,4,4720642,304.31

\end{filecontents*}
\pgfplotstabletypeset[
        col sep=comma,
        string type,
        header=false,
        columns={0,1,2,3},
        columns/0/.style={column name={DoF $\bigl(|V_h|+|V_{2,h}|+|\Lambda_h|\bigr)$}, column type={w{l}{4.6cm}||}},
        columns/1/.style={column name={$\beta_2=10$}, column type={w{c}{1cm}}},
        columns/2/.style={column name={$\beta_2=10^3$}, column type={w{c}{1cm}}},
        columns/3/.style={column name={$\beta_2=10^7$}, column type={w{c}{1cm}}},
        every head row/.style={
            before row=\toprule
            \multicolumn{4}{c}{\textbf{Iteration counts with $\mathcal{P}_{\gamma}$}} \\ \midrule,
            after row=\midrule
          },
        every last row/.style={after row=\bottomrule}
      ]{data_ideal_iterations_square.csv}
    }
    \caption{Immersed square.}
    \label{tab:ideal_square}
  \end{subtable}\hfill
  \begin{subtable}{0.48\textwidth}
    \centering
    \resizebox{\textwidth}{!}{%
      \begin{filecontents*}{data_ideal_iterations_ball.csv}
#DoF                     #10 #1e3  #1e7 #N      #time
289 + 8+8                , 4,    3 ,  3, 339  ,  0.060
1089 + 25+25              , 7,    7 ,  7, 1251 ,  0.072     
{4,225 + 89+89}           , 7,    7 ,  7, 4803 ,  0.146     
{16,641 + 337+337}         , 7,    7 ,  7, 18819,  0.526   
{66,049 + 1,313+1,313}       , 6,    6 ,  6, 74499,  2.528   
{263,169 + 5,185+5,185}      , 6,    6 ,  6, 296451, 11.723   
{1,050,625 + 20,609+20,609}   , 5,    5 ,  5, 1182723,59.70  
{4,198,401 + 82,177+82,177}   , 5,    5 ,  5, 4720642,304.98  

\end{filecontents*}
\pgfplotstabletypeset[
        col sep=comma,
        string type,
        header=false,
        columns={0,1,2,3},
        columns/0/.style={column name={DoF $\bigl(|V_h|+|V_{2,h}|+|\Lambda_h|\bigr)$}, column type={w{l}{4.5cm}||}},
        columns/1/.style={column name={$\beta_2=10$}, column type={w{c}{1cm}}},
        columns/2/.style={column name={$\beta_2=10^3$}, column type={w{c}{1cm}}},
        columns/3/.style={column name={$\beta_2=10^7$}, column type={w{c}{1cm}}},
        every head row/.style={
            before row=\toprule
            \multicolumn{4}{c}{\textbf{Iteration counts with $\mathcal{P}_{\gamma}$}} \\ \midrule,
            after row=\midrule
          },
        every last row/.style={after row=\bottomrule}
      ]{data_ideal_iterations_ball.csv}
    }
    \caption{Immersed ball.}
    \label{tab:ideal_ball}
  \end{subtable}
  \caption{GMRES iteration counts with the ideal AL preconditioner $\mathcal{P_\gamma}$, varying the jump coefficient $\beta_2$ across several refinement levels for the square and ball geometries.}
  \label{tab:ideal_preconditioner}
\end{table}

\begin{figure}[ht]
  \centering
  \begin{minipage}{0.35\textwidth}
    \centering
    \includegraphics[width=\linewidth]{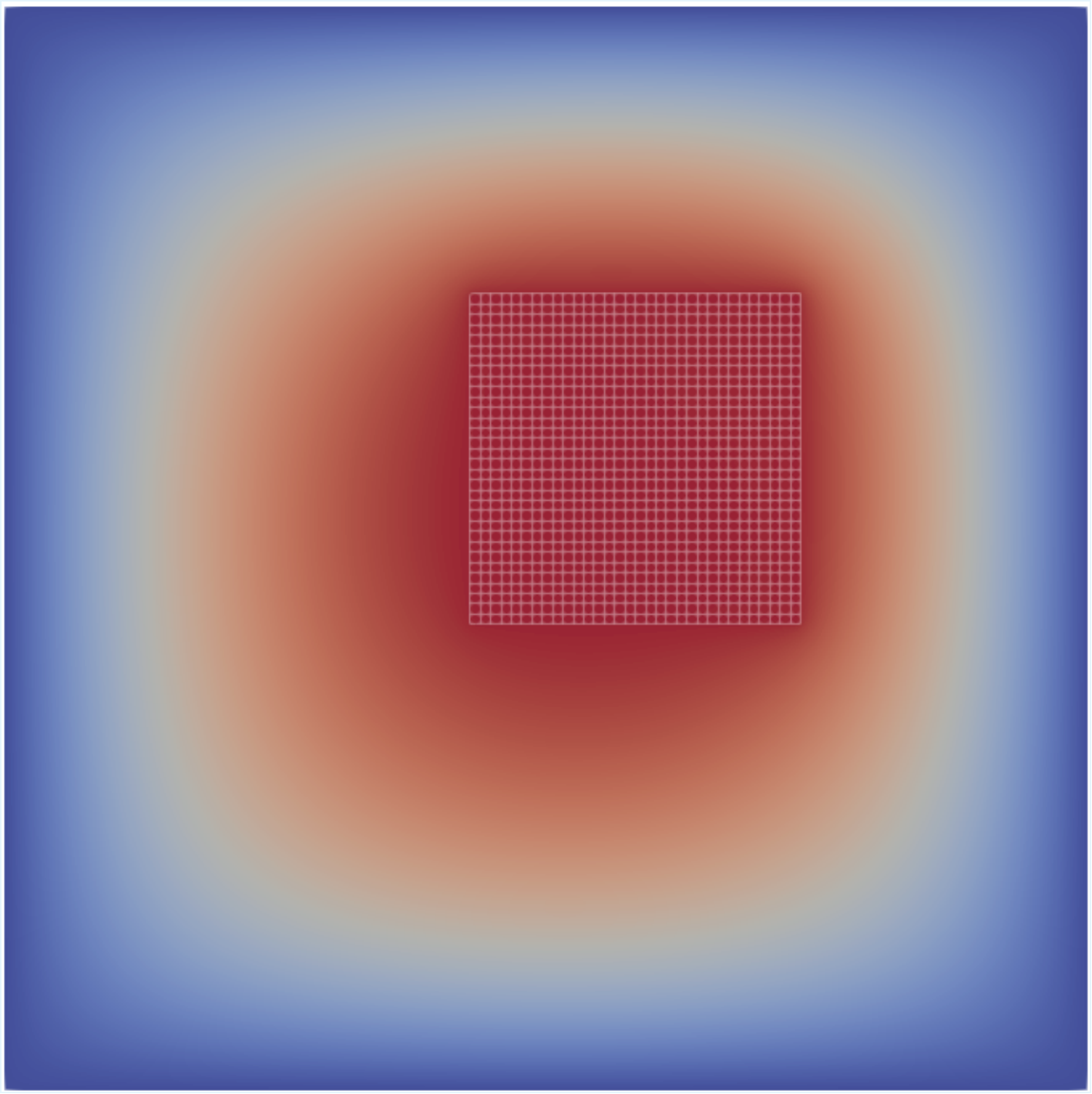}
  \end{minipage}\hfill
  \begin{minipage}{0.35\textwidth}
    \centering
    \includegraphics[width=\linewidth]{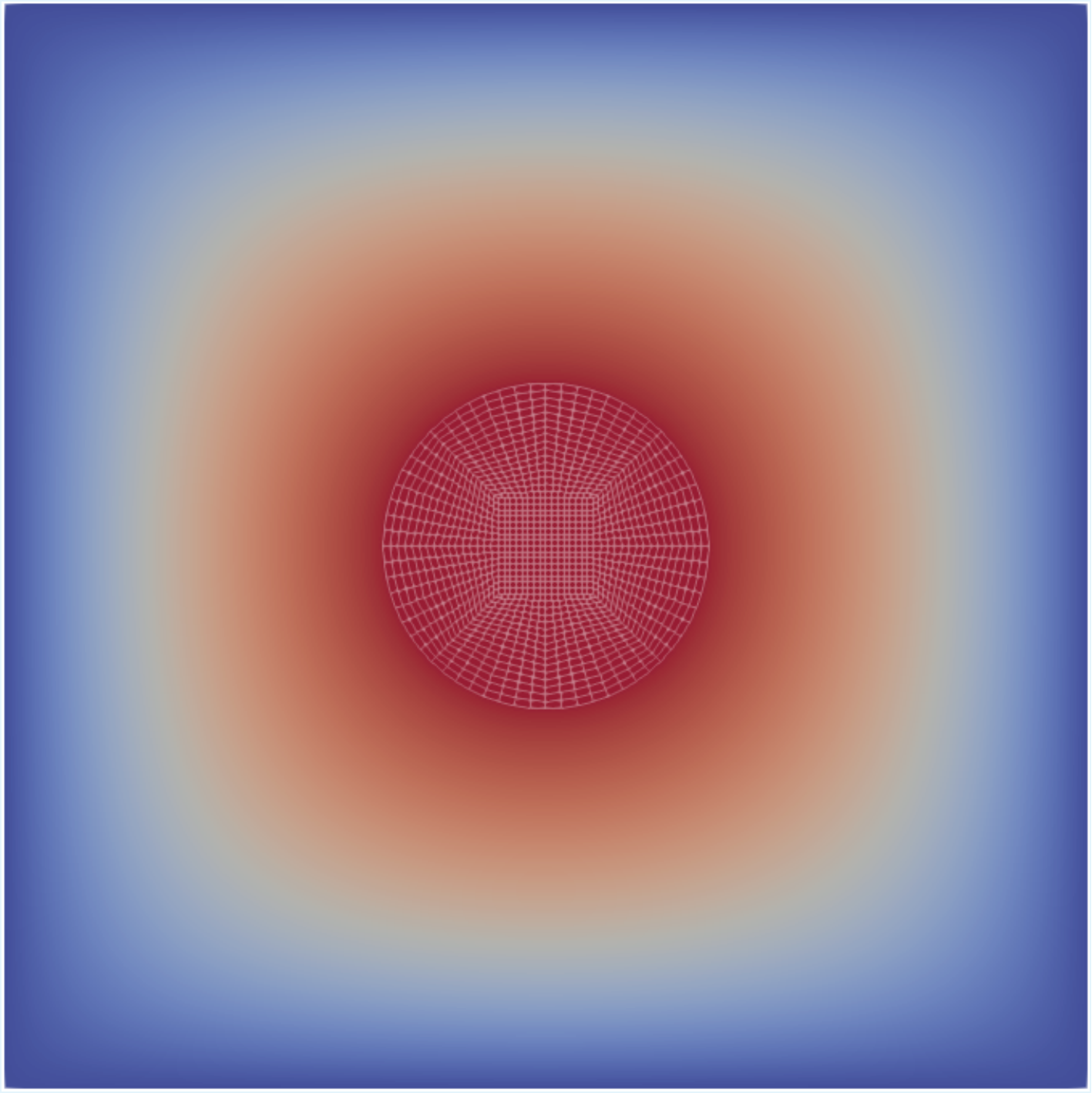}
  \end{minipage}

  \vspace{0.3cm}

  \includegraphics{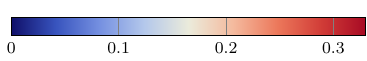}
  \caption{Numerical solutions $u_h$ when $\beta_2=10^7$ for the square and ball geometries.}
  \label{fig:ideal_solutions}
\end{figure}

\subsection{Sanity check on manufactured test case}
We validate our implementation through a convergence test using a manufactured solution. We consider $\Omega =[-1.4,1.4]^2$, $\Omega_2 = \mathcal{B}_1(0,0)$, $f=f_2=1$, and $u_1(x,y)=\frac{4-x^2 - y^2}{4}$ in $\Omega_1$, and $u_2(x,y)= \frac{31-x^2 -y^2}{40}$ in $\Omega_2$. The coefficients are set to $\beta = 1$, $\beta_2=10$.
We report in Figure~\ref{fig:convergence_manufactured} the convergence curves. The observed orders of convergence are almost $1$ in $L^2$
and $\frac{1}{2}$ in $H^1$, consistently with the regularity of the solution $u$ in $H^r(\Omega)$, $r \in (1,\frac{3}{2})$~\cite[Sec. 4.2]{BoffiGastaldi2024}.

\begin{figure}[ht]
  \centering
  \includegraphics{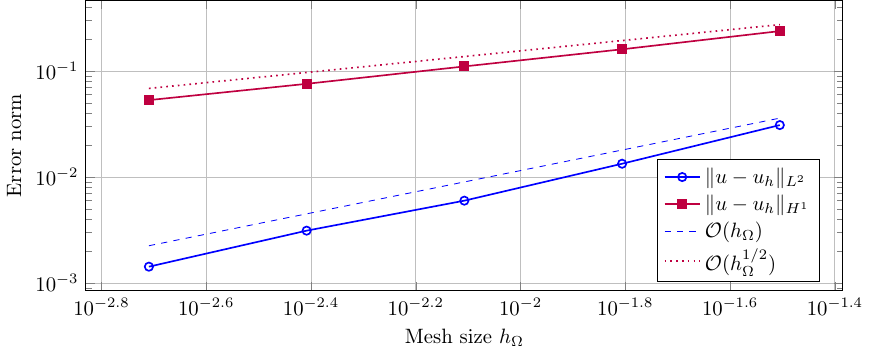}
  \caption{Convergence of the $L^2$ and $H^1$ norms of the error for the manufactured solution
      under simultaneous refinement of both grids. The dashed and dotted lines indicate
      the expected orders $\mathcal{O}(h_\Omega)$ and $\mathcal{O}(h_{\Omega}^{1/2})$, respectively.}
  \label{fig:convergence_manufactured}
\end{figure}

Finally, we investigate in Table~\ref{tab:ratio} the effect of the ratio $\frac{h_{\Omega}}{h_{\Omega_2}}$ on the behavior of the ideal AL preconditioner. We display the outer iteration counts for $\mathcal{P}{\gamma}$ (with $\gamma=10$), for the manufactured test case introduced above, under simultaneous refinement of both grids. For each fixed value of the ratio, the background and immersed meshes are refined simultaneously so that $\frac{h_{\Omega}}{h_{\Omega_2}}$ is kept approximately constant across refinement levels. Alongside the number of background degrees of freedom $|V_h|$, we also report the number of immersed degrees of freedom $|V_{2,h}|=|\Lambda_h|$, which coincide in this test case.
As the immersed mesh becomes finer than the background mesh, the iteration count increases. However, this growth remains moderate for the range of comparable mesh sizes considered, in agreement with the assumptions adopted throughout this work. Moreover, for each fixed mesh-size ratio, the iteration count remains stable across successive refinement levels.

\begin{table}[htbp]
  \centering
  \begin{tabular}{c !{\vrule} c c !{\vrule} c c !{\vrule} c c}
    \toprule
    \multicolumn{7}{c}{\textbf{Outer iteration counts with $\mathcal{P}_{\gamma}$}}                                                                                                                                                                               \\ \midrule
               & \multicolumn{2}{c !{\vrule}}{$\frac{h_{\Omega}}{h_{\Omega_2}}=0.46$} & \multicolumn{2}{c !{\vrule}}{$\frac{h_{\Omega}}{h_{\Omega_2}}=0.70$} & \multicolumn{2}{c}{$\frac{h_{\Omega}}{h_{\Omega_2}}=1.25$}                                         \\
    $|V_h|$    & $|V_{2,h}|=|\Lambda_h|$                                              & it.                                                                  & $|V_{2,h}|=|\Lambda_h|$                                    & it.  & $|V_{2,h}|=|\Lambda_h|$ & it.  \\
    \midrule
    $289$      & $25$                                                                 & $6$                                                                  & $89$                                                       & $11$ & $337$                   & $18$ \\
    $1{,}089$  & $89$                                                                 & $8$                                                                  & $337$                                                      & $13$ & $1{,}313$               & $19$ \\
    $4{,}225$  & $337$                                                                & $8$                                                                  & $1{,}313$                                                  & $13$ & $5{,}185$               & $17$ \\
    $16{,}641$ & $1{,}313$                                                            & $8$                                                                  & $5{,}185$                                                  & $13$ & $20{,}609$              & $16$ \\
    \bottomrule
  \end{tabular}
  \caption{Outer iteration counts for ideal AL preconditioner $\mathcal{P}_{\gamma}$ (with $\gamma=10$), applied to the manufactured test case, for different values of the ratio $\frac{h_{\Omega}}{h_{\Omega_2}}$ between the background and immersed mesh sizes. For each ratio, we report the number of background degrees of freedom $|V_h|$ together with the number of immersed degrees of freedom $|V_{2,h}|=|\Lambda_h|$ (which coincide in this test case) and the outer iteration count (it.); grids are refined simultaneously so that the ratio is kept approximately constant across refinement levels.}
  \label{tab:ratio}
\end{table}

\subsection{Validation of the inexact modified AL preconditioner}\label{subsec:modified_numerical_experiments}
Having validated the ideal AL preconditioner and its inexact variant, we consider the inexact MAL preconditioner $\mathcal{\widetilde P}_{\gamma_1,\gamma_2}^{\text{diag}}$ , in which inner iterative methods and a diagonal approximation of $\mathsf{W}$ are employed. We repeat the same experiments for both geometries in Section~\ref{subsec:ideal_validation}, fixing $\gamma_1=10$ and $\gamma_2=10^{-2}$.
The results are shown in Table~\ref{tab:modified_preconditioner}.
Slightly higher iteration counts than in the unmodified case are observed, which are nevertheless independent of the refinement level and reasonably robust with respect to the jump in the coefficients. However, when $\beta_2=10$, we notice an increase in iteration counts. This is consistent with the observations made in Remark~\ref{rmk:scalingLA2} about the better properties of the modified preconditioner for larger jumps. Between
parentheses, we report the iteration counts when the parameter $\gamma_2$ is taken to the lower value of $10^{-3}$, which provides
consistent results with those for higher values of $\beta_2$. We have checked that using $\gamma_2 = 10^{-3}$ for the other values of $\beta_2$ gives lower outer iterations, but the gain is not as large as with $\beta_2=10$, since the spectrum is already more clustered when $\beta_2 - \beta$ is large. For this reason, we will keep $\gamma_2=10^{-2}$ in these regimes.

The last column of Table~\ref{tab:modified_preconditioner} shows the (average) number of inner conjugate-gradient iterations required for the approximate inversion of the $(1,1)$-block
of the preconditioner when $\beta_2=10^7$. We observe very low inner iteration counts upon mesh refinement, indicating that this modified variant might be a viable option to further
reduce the computational cost of the AL preconditioner.
Indeed, the gain in computational cost of the inexact MAL preconditioner with respect to the inexact AL preconditioner is quite significant, as indicated by the wall-clock time reported in Figure~\ref{fig:wallclock_comparison}, where we show, in a log-log scale, the wall-clock time (in seconds) to solve
the linear system for a sequence of refinement cycles when $\beta_2=10^7$ for the square and ball test cases. The curve for the modified approach lies well below that of the inexact AL preconditioner. Although it requires more outer iterations, the former yields the lowest times, thanks to the lower number of inner iterations. More precisely, for the finest refinement level of the immersed square test case, $\mathcal{P}_{\gamma}^{\text{inexact}}$ takes about $300$ seconds, while $\mathcal{\widetilde P}_{\gamma_1,\gamma_2}^{\text{diag}}$  needs $67$ seconds. Compared to the unmodified variant, $\mathcal{\widetilde P}_{\gamma_1,\gamma_2}^{\text{diag}}$ gives a speed-up of a factor greater than four. A possible alternative would be to use $\operatorname{diag}(\mathsf{W})$  only in the preconditioner. This choice reduces the number of outer iterations but substantially increases the number of inner iterations, limiting its applicability due to the observed higher computational cost. Therefore, we omit it from the discussion. The timing data in Figure~\ref{fig:wallclock_comparison} should be read as representative of the implementation and hardware described above, rather than as a complete performance study.

\begin{table}[ht]
  \centering

  \begin{subtable}{0.48\textwidth}
    \centering
    \resizebox{\textwidth}{!}{%
    \begin{filecontents*}{data_modified_iterations_square_diagonal.csv}
#DoF                     #10        #1e3  #1e7 #N      #time   #avg inner
289 + 25 + 25                ,      16 , 17 , 17, 339     ,  0.01630 , 2
1089 + 81 + 81               ,      18 (17) , 18 , 18, 1251    ,  0.01880 , 2
{4,225 + 289 + 289}           ,     19 (17) , 19 , 19, 4803    ,  0.040135 , 8
{16,641 + 1,089 + 1,089}        ,   20 (16) , 20 , 20, 18819   ,  0.1137 , 8
{66,049 + 4,225 + 4,225}        ,   22 (16) , 22 , 22, 74499   ,  0.568 ,  9
{263,169 + 16,641 + 16,641}      ,  25 (18) , 24 , 24, 296451  ,  3.221 ,  9
{1,050,625 + 66,049 + 66,049}    ,  39 (19) , 25 , 25, 1182723 ,  14.54 , 10
{4,198,401 + 263,169 + 263,169}   , 100 (28) , 22 , 21, 4720642 ,  65.32,  11

\end{filecontents*}
\pgfplotstabletypeset[
    col sep=comma,
    string type,
    header=false,
    columns={0,1,2,3,6},
    columns/0/.style={column name={DoF $\bigl(|V_h|+|V_{2,h}|+|\Lambda_h|\bigr)$}, column type={w{l}{4.6 cm}||}},
    columns/1/.style={column name={$\beta_2=10$}, column type={w{c}{1cm}}},
    columns/2/.style={column name={$\beta_2=10^3$}, column type={w{c}{1cm}}},
    columns/3/.style={column name={$\beta_2=10^7$}, column type={w{c}{1cm}}},
    columns/6/.style={column name={Inner}, column type={w{c}{1cm}}},
    every head row/.style={
    before row=\toprule
    \multicolumn{5}{c}{\textbf{FGMRES iteration counts with $\mathcal{\widetilde P}_{\gamma_1,\gamma_2}^{\text{diag}}$}} \\ \midrule,
    after row=\midrule
    },
    every last row/.style={after row=\bottomrule}
    ]{data_modified_iterations_square_diagonal.csv}
    }
    \caption{Immersed square.}
    \label{tab:modified_square_diag}
  \end{subtable}\hfill
  \begin{subtable}{0.48\textwidth}
    \centering
    \resizebox{\textwidth}{!}{%
    \begin{filecontents*}{data_modified_iterations_ball_diagonal.csv}
#DoF                     #10        #1e3  #1e7 #N      #time   #avg inner
289 + 8 + 8                ,      5 ,  5 ,  6,   339     ,  0.00659 ,  2
1089 + 25 + 25              ,     10 (11) , 10 , 10,  1251    ,  0.01059 ,  2
{4,225 + 89 + 89}           ,     16 (14) , 16 , 16,  4803    ,  0.03439 ,  8
{16,641 + 337 + 337}         ,    16 (13) , 16 , 16,  18819   ,  0.10339 ,  8
{66,049 + 1,313 + 1,313}       ,  19 (14) , 18 , 18,  74499   ,  0.5320 ,  9
{263,169 + 5,185 + 5,185}      ,  20 (14) , 19 , 18, 296451  ,  2.6984 ,  9
{1,050,625 + 20,609 + 20,609}   , 23 (16) , 21 , 21, 1182723 ,  14.6557 , 10
{4,198,401 + 82,177 + 82,177}   , 36 (17) , 21 , 21, 4720642 ,  66.0122,  11

\end{filecontents*}
\pgfplotstabletypeset[
    col sep=comma,
    string type,
    header=false,
    columns={0,1,2,3,6},
    columns/0/.style={column name={DoF $\bigl(|V_h|+|V_{2,h}|+|\Lambda_h|\bigr)$}, column type={w{l}{4.5cm}||}},
    columns/1/.style={column name={$\beta_2=10$}, column type={w{c}{1cm}}},
    columns/2/.style={column name={$\beta_2=10^3$}, column type={w{c}{1cm}}},
    columns/3/.style={column name={$\beta_2=10^7$}, column type={w{c}{1cm}}},
    columns/6/.style={column name={Inner}, column type={w{c}{1cm}}},
    every head row/.style={
    before row=\toprule
    \multicolumn{5}{c}{\textbf{FGMRES iteration counts with $\mathcal{\widetilde P}_{\gamma_1,\gamma_2}^{\text{diag}}$}} \\ \midrule,
    after row=\midrule
    },
    every last row/.style={after row=\bottomrule}
    ]{data_modified_iterations_ball_diagonal.csv}
    }
    \caption{Immersed ball.}
    \label{tab:modified_ball_diag}
  \end{subtable}

  \caption{Outer iteration counts for FGMRES preconditioned by $\mathcal{\widetilde P}_{\gamma_1,\gamma_2}^{\text{diag}}$, varying the jump coefficient $\beta_2$ across several refinement levels for the square (\ref{tab:modified_square_diag}) and ball (\ref{tab:modified_ball_diag}) tests. The AL parameters are $\gamma_1=10$, $\gamma_2=10^{-2}$. Values between parentheses indicate the iteration counts for $\gamma_2=10^{-3}$. The last column reports the average number of inner CG iterations to invert the $(1,1)$-block when $\beta_2=10^{7}$.}
  \label{tab:modified_preconditioner}
\end{table}

\begin{figure}[ht]
  \centering
  \begin{minipage}{0.48\textwidth}
    \centering
    \includegraphics{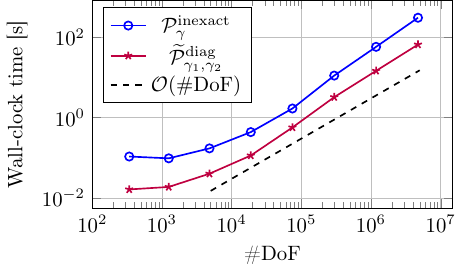}
    \subcaption{Immersed square.}
    \label{fig:wallclock_square}
  \end{minipage}\hfill
  \begin{minipage}{0.48\textwidth}
    \centering
    \includegraphics{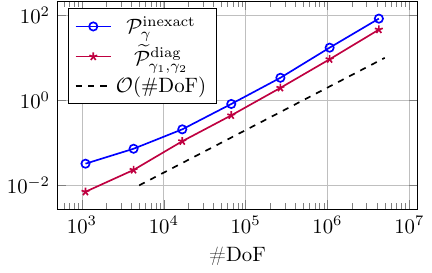}
    \subcaption{Immersed ball.}
    \label{fig:wallclock_ball}
  \end{minipage}
  \caption{Comparison between the inexact AL $\mathcal{P}_{\gamma}^{\text{inexact}}$ and the inexact MAL $\mathcal{\widetilde P}_{\gamma_1,\gamma_2}^{\text{diag}}$ preconditioners in terms of wall-clock time for $\beta_2=10^7$, using the test cases with the immersed square (\subref{fig:wallclock_square}) and the immersed ball (\subref{fig:wallclock_ball}).}
  \label{fig:wallclock_comparison}
\end{figure}

\subsubsection{Comparison with a preconditioner from the literature}
We end the validation of our modified preconditioner by comparing in Table~\ref{tab:comparison_literature} the number of iterations of the inexact MAL preconditioner with those obtained using the block upper-triangular preconditioner introduced in~\cite{boffi2023parallel} and extended in~\cite{alshehri2025multigridpreconditioningfddlmmethod} to the elliptic-interface case. The following implementation choices are used: the proposed preconditioner employs FGMRES(30) with a relative residual tolerance of $10^{-10}$ and inexact block applications via CG+AMG, while the literature preconditioner uses GMRES(50) with the same residual tolerance and a maximum of 500 iterations. The comparison is meant to indicate performance under these stated implementations. We mention that the
latter preconditioner requires the inversion of $\mathsf{A}$ and of the two-by-two block\footnote{The existence of this inverse follows from the ellipticity of $A_2$ on the kernel of $C_2$, together with the corresponding discrete inf-sup stability condition.}
$\begin{bmatrix}
    \mathsf{A_2}  & -\mathsf{C_2^T} \\
    -\mathsf{C_2} & 0
  \end{bmatrix}
$. We consider direct inversion of these blocks, but multigrid-based inexact variants for these block solves are also possible in practice~\cite{alshehri2025multigridpreconditioningfddlmmethod}.
The comparison is therefore an iteration count comparison under matched stopping criteria, not a detailed comparison between implementations; our interest is in the effectiveness of the preconditioners with respect to jumps in the coefficients. We consider again the
ball $\mathcal{B}_{0.3}(0,0)$ as immersed geometry, using a wider range of values for $\beta_2$ compared to the previous examples, in order to investigate the robustness with respect to increasing coefficients jumps. Each table entry reports, as first element, the number of iterations of FGMRES preconditioned by $\mathcal{\widetilde P}_{\gamma_1,\gamma_2}^{\text{diag}}$, and, as second element, the iterations of GMRES preconditioned by the block upper-triangular preconditioner in~\cite{boffi2023parallel}. For the latter, a deterioration in
the convergence is observed as the jump increases: for $\beta_2=10^5$, iteration counts displays a lack of robustness with respect to mesh refinement, while for $\beta_2=10^7$ GMRES
fails to converge within $500$ iterations. Conversely, the inexact MAL preconditioner maintains low iteration counts across all jumps and refinement levels. However, for moderate values of $\beta_2$, both approaches exhibit good iteration counts, which are robust in the mesh sizes.

\begin{table}[htbp]
  \centering
  \begin{minipage}{0.62\textwidth}
    \centering
    \resizebox{\textwidth}{!}{%
    \begin{filecontents*}{data_modified_iterations_ball_with_literature.csv}
{289 + 8 + 8},                5/3,      5/4,     6/4,  6/8
{1,089 + 25 + 25},            11/8,     10/8,    10/8,  10/13
{4,225 + 89 + 89},            14/10,    16/13,   16/18, 16/20
{16,641 + 337 + 337},         13/9,     16/14,   16/24, 16/29
{66,049 + 1,313 + 1,313},     14/8,     18/13,   18/31, 18/37
{263,169 + 5,185 + 5,185},    14/7,     19/11,   18/50, 18/†
{1,050,625 + 20,609 + 20,609},16/6,     21/10,   21/49, 21/†
\end{filecontents*}
\pgfplotstabletypeset[
    col sep=comma,
    string type,
    header=false,
    columns={0,1,2,3,4},
    columns/0/.style={column name={DoF $\bigl(|V_h|+|V_{2,h}|+|\Lambda_h|\bigr)$}, column type=l||},
    columns/1/.style={column name={$\beta_2=10$}, column type={c@{\hspace{8mm}}}},
    columns/2/.style={column name={$\beta_2=10^3$}, column type=c},
    columns/3/.style={column name={$\beta_2=10^5$}, column type=c},
    columns/4/.style={column name={$\beta_2=10^7$}, column type=c},
    every head row/.style={
    before row=\toprule
    \multicolumn{5}{c}{\textbf{(F)GMRES iteration counts: $\mathcal{\widetilde P}_{\gamma_1,\gamma_2}^{\text{diag}}$ / block-triangular~\cite{alshehri2025multigridpreconditioningfddlmmethod}}} \\ \midrule,
    after row=\midrule
    },
    every last row/.style={after row=\bottomrule}
    ]{data_modified_iterations_ball_with_literature.csv}
    }
  \end{minipage}
  \caption{Comparison between the inexact MAL preconditioner (with FGMRES) and the block upper-triangular preconditioner (with GMRES) from~\cite{alshehri2025multigridpreconditioningfddlmmethod}, applied to the immersed ball geometry. Symbol $\dagger$ indicates that GMRES did not converge within $500$ iterations. When $\beta_2=10$ we report the iteration counts for $\gamma_2=10^{-3}$, while other columns are obtained with $\gamma_2=10^{-2}$.}
  \label{tab:comparison_literature}
\end{table}
The superior robustness of our preconditioner with respect to coefficient jumps, as shown in Table~\ref{tab:comparison_literature}, can be attributed to the observations made in Remark~\ref{rmk:scalingLA2}.

\subsection{More complex forcing term}\label{subsec:complex_forcing_term}
We now consider the numerical test presented in~\cite[Sect. 6]{REGAZZONI2024117327}. The immersed domain is again $\Omega_2 = \mathcal{B}_{0.3}(0,0)$, with
forcing term $f(x,y)= \sin(\pi x) + \tanh(y)$. Neumann homogeneous boundary conditions are imposed on $\partial \Omega$. Numerical solutions for some selected values
of the jump $\beta_2 - \beta$ are shown in Figure~\ref{fig:contour_regazzoni}, where the role played
by the ratio $\frac{\beta_2}{\beta}$ is visible: the larger it is, the stronger the gradient discontinuities across the immersed domain $\Omega_2$.
We use the inexact modified AL-based preconditioner with $\gamma_1=10$ and $\gamma_2 = 10^{-2}$, as in previous tests. Iteration counts using this
configuration across mesh refinement are reported in Table~\ref{tab:modified_regazzoni}. We observe robust
iteration counts, essentially independent of the mesh sizes and of the jump determined by the coefficient $\beta_2$. Consistently with the previous examples, for the lowest value
of $\beta_2$ a mild dependence on the mesh sizes is observed, which is mitigated by decreasing the value of $\gamma_2$ to $10^{-3}$. The average number of inner iterations when $\beta_2=10^7$ is reported in the last column of Table~\ref{tab:modified_regazzoni}, where we observe them to be low and stable, confirming the previous trends.

\begin{figure}[h]
  \centering
  \begin{minipage}{0.25\textwidth}
    \centering
    \includegraphics[width=\linewidth]{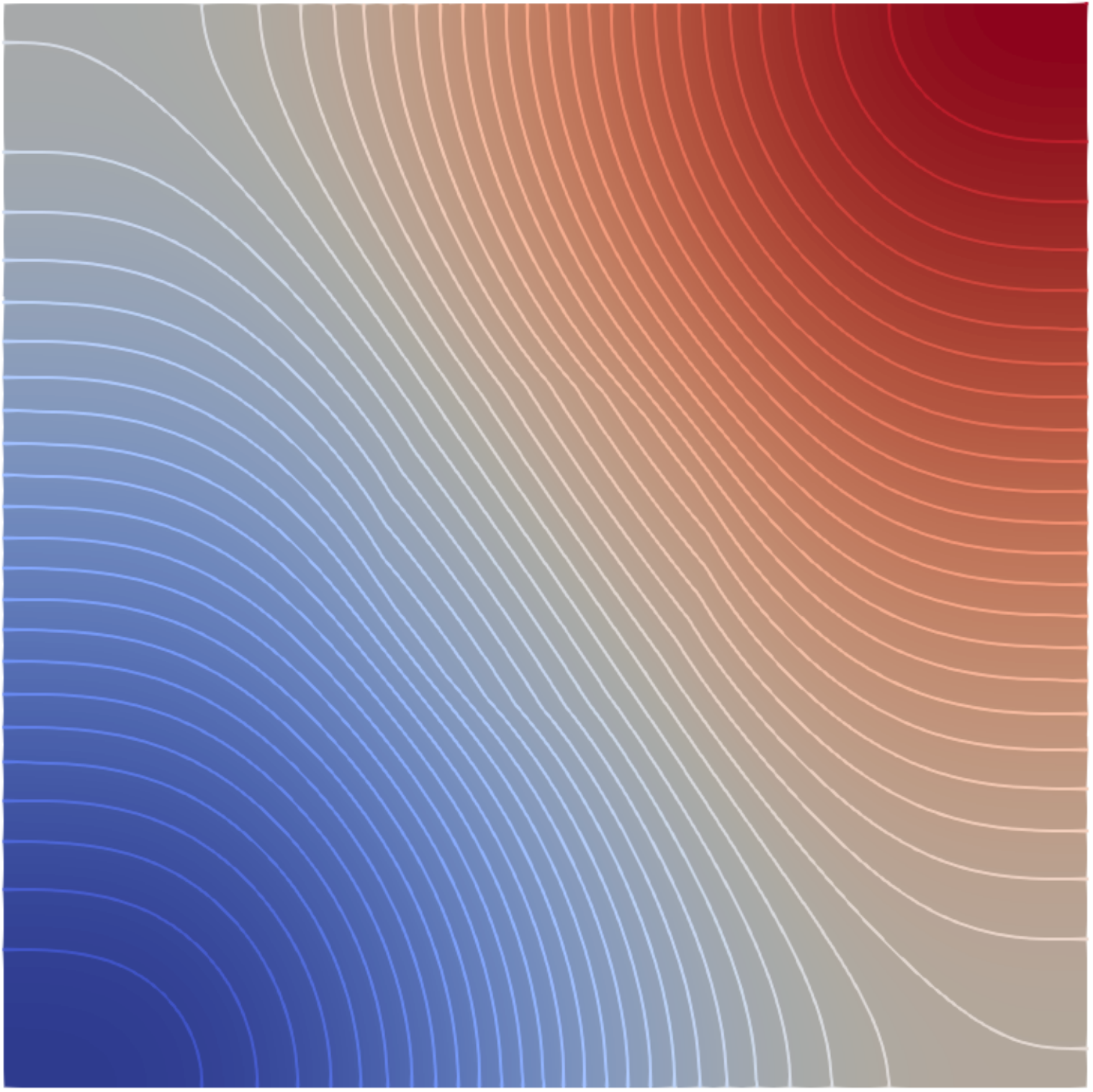}
    \caption*{$\beta_2 =1.4$}
    \label{fig:image1}
  \end{minipage}
  \begin{minipage}{0.25\textwidth}
    \centering
    \includegraphics[width=\linewidth]{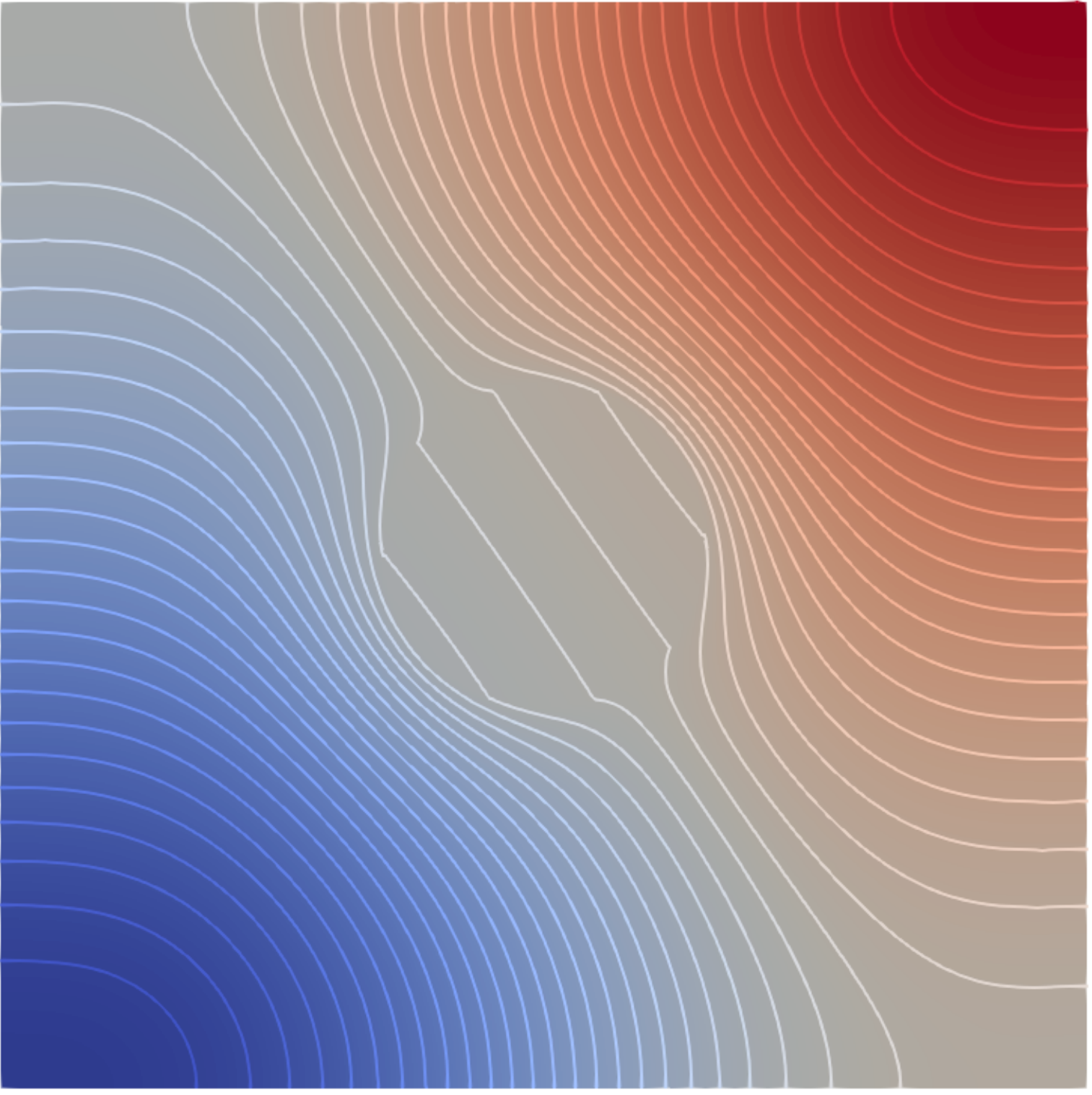}
    \caption*{$\beta_2 =10$}
    \label{fig:image2}
  \end{minipage}
  \begin{minipage}{0.25\textwidth}
    \centering
    \includegraphics[width=\linewidth]{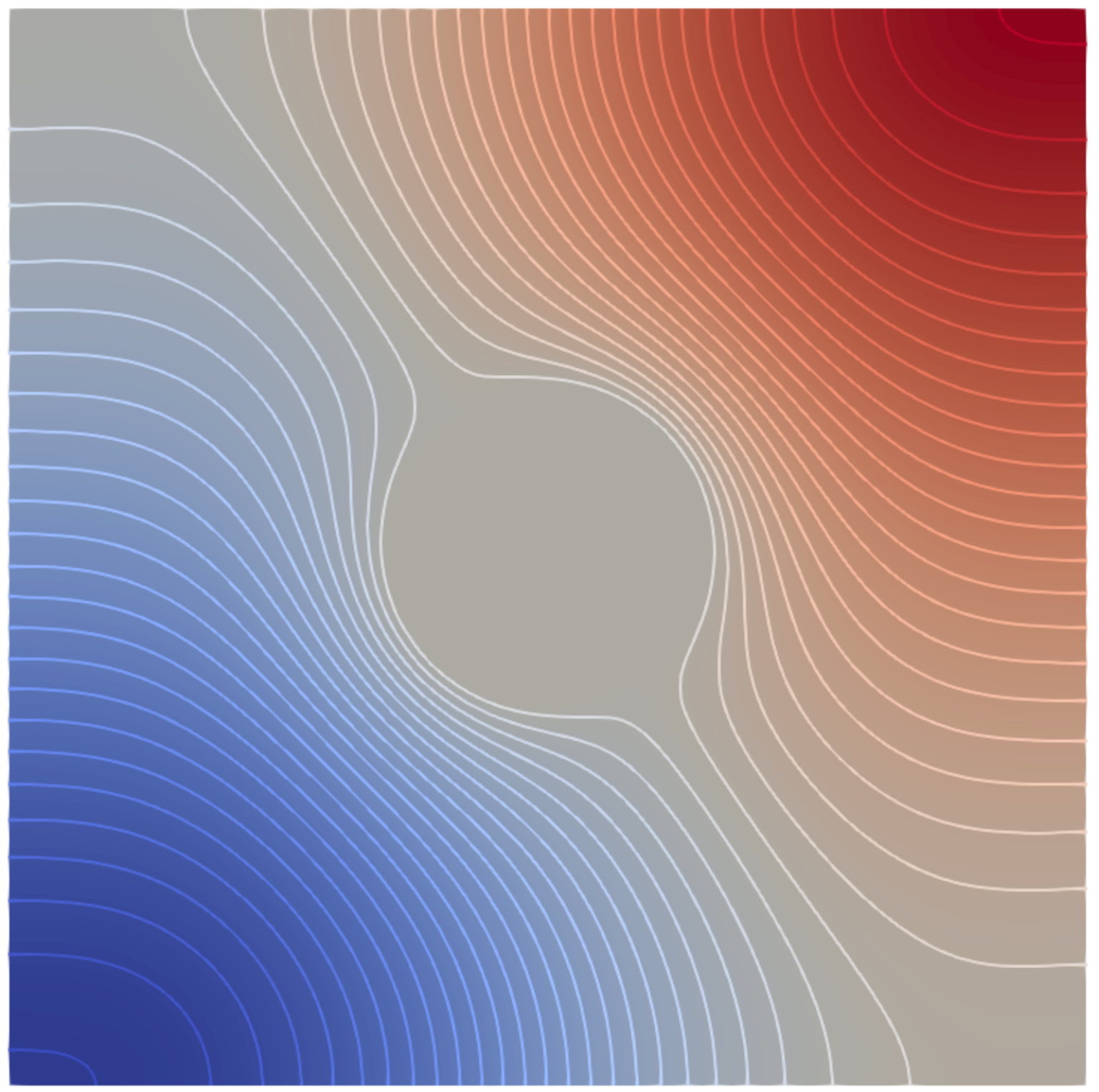}
    \caption*{$\beta_2 =10^7$}
    \label{fig:image3}
  \end{minipage}

  \vspace{0.3cm}

  \includegraphics{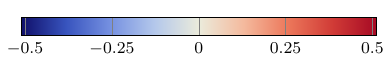}
  \caption{Numerical solutions $u_h$ for different values of $\beta_2$. White lines are contour lines.}
  \label{fig:contour_regazzoni}
\end{figure}

\begin{table}[ht]
  \centering
  \begin{subtable}{0.55\textwidth}
    \centering
    \resizebox{\textwidth}{!}{%
    \begin{filecontents*}{data_modified_iterations_regazzoni_diagonal.csv}
#DoF                     #10        #1e3  #1e7 #N    #time      #avg inner
289 + 8 + 8             ,  4 (5)  , 5 , 4   , 339   ,  0.,  1
1089 + 25 + 25             ,  9 (8)  , 9 , 8   , 1251  ,  0.,  1
{4,225 + 89 + 89}           ,  15 (14)  , 15 , 15  , 4803  ,  0.,  6
{16,641 + 337 + 337}         ,  15 (13)  , 14 , 15  , 18819 ,   0., 6
{66,049 + 1,313 + 1,313}       ,  16 (14)  , 16 , 16  , 74499 ,  0.,  6
{263,169 + 5,185 + 5,185}      ,  16 (13)  , 15 , 15  , 296451 ,  0., 7
{1,050,625 + 20,609 + 20,609}   ,  19 (13)  , 15 , 15  , 1182723 , 0., 8
{4,198,401 + 82,177 + 82,177}   ,  31 (14)  , 14 , 14  , 4720642 , 0., 9

\end{filecontents*}
\pgfplotstabletypeset[
    col sep=comma,
    string type,
    header=false,
    columns={0,1,2,3,6},
    columns/0/.style={column name={DoF $\bigl(|V_h|+|V_{2,h}|+|\Lambda_h|\bigr)$}, column type=l||},
    columns/1/.style={column name={$\beta_2=10$}, column type=c},
    columns/2/.style={column name={$\beta_2=10^3$}, column type=c},
    columns/3/.style={column name={$\beta_2=10^7$}, column type=c},
    columns/6/.style={column name=Inner, column type=c},
    every head row/.style={
    before row=\toprule
    \multicolumn{5}{c}{\textbf{FGMRES iteration counts with $\mathcal{\widetilde P}_{\gamma_1,\gamma_2}^{\text{diag}}$}} \\ \midrule,
    after row=\midrule
    },
    every last row/.style={after row=\bottomrule}
    ]{data_modified_iterations_regazzoni_diagonal.csv}
    }
    \label{tab:regazzoni_diag}
  \end{subtable}
  \caption{Outer iteration counts for FGMRES preconditioned by $\mathcal{\widetilde P}_{\gamma_1,\gamma_2}^{\text{diag}}$ when applied to Example~\ref{subsec:complex_forcing_term}, varying the jump coefficient $\beta_2$ across several refinement levels. The AL parameters are $\gamma_1=10$, $\gamma_2=10^{-2}$. Values between parentheses indicate the iteration counts for $\gamma_2=10^{-3}$. The last column reports the average number of inner iterations when $\beta_2=10^7$.}
  \label{tab:modified_regazzoni}
\end{table}

\subsection{Different finite elements}\label{subsec:diff_fem}
We test the performance of the modified preconditioner when using different spaces than conforming $\mathcal{Q}^1$ elements. As outlined in Remark~\ref{rmk:spaces}, piecewise constant elements
can be employed for the multiplier space $\Lambda_h$. Among them, we consider the
$\mathcal{Q}^1-(\mathcal{Q}^1 + \mathfrak{B})-\mathcal{Q}^0$ element, which was analyzed in~\cite{BoffiGastaldi2024}. With this choice for $\Lambda_h$,
the mass matrix on the multiplier space, and hence $\mathsf{W}$, are diagonal matrices by construction.
On the other hand, the matrix $\mathsf{C}_2$ is now a rectangular (mass) matrix. We repeat the same test cases as in Section~\ref{subsec:modified_numerical_experiments}, using the inexact MAL
preconditioner with $\gamma_1=10$ and $\gamma_2=10^{-2}$, for several refinement cycles and jumps in the coefficients. The
results of such tests are in Table~\ref{tab:iteration_counts_dg_space}. Iteration counts for FGMRES are robust with respect to
mesh refinements and jump sizes, confirming the effectiveness of the preconditioner also for other stable spaces. Inner iteration counts for the $(1,1)$-block when $\beta_2=10^7$ remain low and stable, and are reported in the last column of Table~\ref{tab:iteration_counts_dg_space}.

\begin{table}[h!]
  \centering
  \begin{subtable}{0.48\textwidth}
    \centering
    \resizebox{\textwidth}{!}{%
      \begin{filecontents*}{data_modified_iterations_square_dg.csv}
#DoF                             #10   #1e3  #1e7 #N      #time   #avg inner
289 + 41 + 16                    , 13 (14) , 14 , 10, 339          ,  0.022 ,  1
1089 + 145 + 64                  , 18 (18) , 20 , 20, 1251         ,  0.032 ,  1
{4,225 + 545 + 256}              , 19 (17) , 19 , 19, 4803         ,  0.058 ,  5
{16,641 + 2,113 + 1,024}         , 18 (16) , 19 , 19, 18819        ,  0.181 ,  6
{66,049 + 8,321 + 4,096}         , 19 (15)  ,19 , 19, 74499        ,  1.070 ,  7
{263,169 + 33,025 + 16,384}      , 21 (15) ,20 , 20, 296451       ,  6.441 ,   7
{1,050,625 + 131,585 + 65,536}   , 23 (16), 22 , 22, 1182723      ,  34.34 ,   8
{4,198,401 + 525,313 + 262,144}  , 38 (18), 23 , 23, 4720642      ,  132.27,   8

\end{filecontents*}
\pgfplotstabletypeset[
        col sep=comma,
        string type,
        header=false,
        columns={0,1,2,3,6},
        columns/0/.style={column name={DoF $\bigl(|V_h|+|V_{2,h}|+|\Lambda_h|\bigr)$}, column type={w{l}{4.6cm}||}},
        columns/1/.style={column name={$\beta_2=10$}, column type={w{c}{1cm}}},
        columns/2/.style={column name={$\beta_2=10^3$}, column type={w{c}{1cm}}},
        columns/3/.style={column name={$\beta_2=10^7$}, column type={w{c}{1cm}}},
        columns/6/.style={column name={Inner}, column type={w{c}{1cm}}},
        every head row/.style={
            before row=\toprule
            \multicolumn{5}{c}{\textbf{FGMRES iteration counts with $\mathcal{\widetilde P}_{\gamma_1,\gamma_2}$}} \\ \midrule,
            after row=\midrule
          },
        every last row/.style={after row=\bottomrule}
      ]{data_modified_iterations_square_dg.csv}
    }
    \caption{Immersed square.}
    \label{tab:dg_square}
  \end{subtable}\hfill
  \begin{subtable}{0.48\textwidth}
    \centering
    \resizebox{\textwidth}{!}{%
      \begin{filecontents*}{data_modified_iterations_ball_dg.csv}
#DoF                            #10      #1e3  #1e7 #N    #time      #avg inner
289 + 13 + 5                    , 6  (5) ,  5  ,   5   , 339   , 0.014,    1
1089 + 45 + 20                  , 6  (6) ,  6  ,   6   , 1251  , 0.024,    1
{4,225 + 169 + 80}              , 15 (13) ,  15 ,   15  , 4803  , 0.074,    8
{16,641 + 657 + 320}            , 17 (15) ,  19 ,   19 , 18819 ,  0.349,    8
{66,049 + 2,593 + 1,280}        , 19 (14) ,  19 ,   19 , 74499 ,  1.892,    9
{263,169 + 10,305 + 5,120}      , 20 (15) ,  21 ,   21,  296451 , 8.232,    10
{1,050,625 + 41,089 + 20,480}   , 23 (17) ,  23 ,   23, 1182723 , 44.01,    10
{4,198,401 + 164,097 + 81,920}  , 26 (18) ,  26 ,   25, 4720642 , 199.4,    11

\end{filecontents*}
\pgfplotstabletypeset[
        col sep=comma,
        string type,
        header=false,
        columns={0,1,2,3,6},
        columns/0/.style={column name={DoF $\bigl(|V_h|+|V_{2,h}|+|\Lambda_h|\bigr)$}, column type={w{l}{4.5cm}||}},
        columns/1/.style={column name={$\beta_2=10$}, column type={w{c}{1cm}}},
        columns/2/.style={column name={$\beta_2=10^3$}, column type={w{c}{1cm}}},
        columns/3/.style={column name={$\beta_2=10^7$}, column type={w{c}{1cm}}},
        columns/6/.style={column name={Inner}, column type={w{c}{1cm}}},
        every head row/.style={
            before row=\toprule
            \multicolumn{5}{c}{\textbf{FGMRES iteration counts with $\mathcal{\widetilde P}_{\gamma_1,\gamma_2}$}} \\ \midrule,
            after row=\midrule
          },
        every last row/.style={after row=\bottomrule}
      ]{data_modified_iterations_ball_dg.csv}
    }
    \caption{Immersed ball.}
    \label{tab:dg_ball}
  \end{subtable}
  \caption{FGMRES iteration counts for the modified preconditioner $\mathcal{\widetilde P}_{\gamma_1,\gamma_2}^{\text{diag}}$ applied to different geometries using the $\mathcal{Q}^1 - (\mathcal{Q}^1 + \mathfrak{B}) - \mathcal{Q}^0$ element. Values between parentheses indicate the iteration counts for $\gamma_2=10^{-3}$.}
  \label{tab:iteration_counts_dg_space}
\end{table}

\FloatBarrier
\subsection{Three-dimensional test: linear elasticity}
We end this section by testing the inexact MAL preconditioner on a three-dimensional example. To the best of our knowledge, preconditioners proposed for the considered formulation (e.g.,~\cite{alshehri2025multigridpreconditioningfddlmmethod,BOFFI2024406,WangDLMFD_prec}) have thus far been applied exclusively to two-dimensional problems involving the Laplace operator. Instead of solving again the scalar model problem defined in~\eqref{eqn:model_problem} in a higher dimension, we consider a three-dimensional linear elasticity example to test the strength of our approach in a more challenging setting. In our model, each subdomain is modeled as an isotropic and homogeneous material, allowing its mechanical behavior to be fully characterized by the corresponding Lamé constants~\cite{Gurtin_Fried_Anand_2010}. The interested reader is referred to~\cite{ALZETTA2020106334,belponer2025mixeddimensionalmodelingvasculartissues} for the study of fictitious domain formulations applied to fiber-reinforced and multiscale materials.

Let $\mu,\mu_2$ be the shear moduli in $\Omega$ and $\Omega_2$, respectively, and $\lambda, \lambda_2$ the first Lamé parameters. Consistently with the assumptions presented
in Section~\ref{sec:method}, we assume $\mu_2 > \mu>0$ and $\lambda_2 > \lambda>0$. We denote with $\delta_\mu = \mu_2 - \mu$ and $\delta_\lambda = \lambda_2 - \lambda$ the corresponding jumps. Vector-valued functions are indicated in boldface; thus, no ambiguity arises when $\lambda$ denotes either a multiplier or a Lamé parameter. Adapting the arguments from the previous sections
to the vector-valued case, the fictitious domain formulation for the elasticity problem reads as follows.

\begin{problem}[Elasticity problem with FD-DLM]\label{prob:elasticity}
Given $\boldsymbol{f} \in \bigl[L^2(\Omega)\bigr]^3$, find $(\boldsymbol{u},\boldsymbol{u}_2,\boldsymbol{\lambda})\in V\times V_2\times\Lambda$ such that
\begin{align*}
  2 \mu(\varepsilon(\boldsymbol{u}),\varepsilon(\boldsymbol{v}))_\Omega + \lambda(\operatorname{div} \boldsymbol{u}, \operatorname{div} \boldsymbol{v})_\Omega + c(\boldsymbol{\lambda},\boldsymbol{v}|_{\Omega_2})                        & =(\boldsymbol{f},\boldsymbol{v})_\Omega &  & \forall \boldsymbol{v} \in V, \nonumber       \\
  2 \delta_\mu (\varepsilon(\boldsymbol{u}_2),\varepsilon(\boldsymbol{v}_2))_{\Omega_2} + \delta_\lambda(\operatorname{div} \boldsymbol{u}_2, \operatorname{div} \boldsymbol{v}_2)_{\Omega_2} -c(\boldsymbol{\lambda}, \boldsymbol{v}_{2}) & = 0                                     &  & \forall \boldsymbol{v}_{2}\in V_{2},          \\
  c(\boldsymbol{\mu},\boldsymbol{u}|_{\Omega_2} - \boldsymbol{u_{2}})                                                                                                                                                                      & = 0                                     &  & \forall \boldsymbol{\mu}\in\Lambda, \nonumber
\end{align*}
where $\varepsilon(\boldsymbol{u}) = \frac{1}{2}(\nabla \boldsymbol{u} + \nabla \boldsymbol{u}^{\mathsf{T}})$ is the linearized strain tensor, and $V = [H_0^1(\Omega)]^3, V_2=[H^1(\Omega_2)]^3, \Lambda=[H^{-1}(\Omega_2)]^3$ are the vector-valued analogues of the spaces defined for the scalar problem.
\end{problem}It is straightforward to see that the algebraic formulation of Problem~\eqref{prob:elasticity} is formally identical to the scalar one in~\eqref{eqn:matrix}, except that the matrices $\mathsf{A}$ and $\mathsf{A_2}$
correspond to bilinear forms associated with standard linear elasticity problems. Apart from the strain tensor term, the main difference with respect to the Laplace case is the presence of the additional "div-div" terms, which
resembles the so-called "grad-div" stabilization used in standard AL preconditioners~\cite{Reusken,graddiv}. However, it is crucial to observe that, in contrast to such approaches, this term naturally arises from the underlying physics rather than being artificially introduced.
Hence, the setting is identical to the scalar case\footnote{Indeed, the additional term $(\operatorname{div} \boldsymbol{u}_2, \operatorname{div} \boldsymbol{v}_2)_{\Omega_2}$ does not remove the singularity from $\mathsf{A_2}$.} discussed in Section~\ref{sec:AL_prec}, and the preconditioner can be applied without any modification.
In spite of the presence of the "div-div" terms, augmenting only the second equation is insufficient and has a detrimental effect on the number of outer iterations, so both equations must be augmented.

The background domain is the cube $\Omega=[-1.25,1.25]^3$, and the immersed domain $\Omega_2$ is chosen as the rectangular box $[-a,a] \times [-b,b] \times [-c,c]$, with $(a,b,c)=(0.65, 0.3, 0.4)$. The initial configuration, displaying the immersed domain inside the background one, is shown in the left part of Figure~\ref{fig:3D_solution}. The displacement field $\boldsymbol{u}\colon \Omega \rightarrow \mathbb{R}^3$ is set to zero on $\partial \Omega$, while $\boldsymbol{f} = [2,1,1]^\mathsf{T}$ represents the body force.
We test the inexact modified preconditioner $\mathcal{\widetilde P}_{\gamma_1,\gamma_2}^{\text{diag}}$, fixing the augmentation parameters to $\gamma_1\!=\!10$ and $\gamma_2\!=\!10^{-2}$ as in previous two-dimensional tests, while iterations are stopped when the relative residual norm is reduced below $10^{-6}$. In our experiments, we vary
the contrast between the Lamé parameters in the subdomains $\Omega$ and $\Omega_2$, while monitoring the number of iterations under mesh refinement. We start with parameters $(\mu,\mu_2) = (1,10)$ and $(\lambda,\lambda_2)=(2,20)$, corresponding to a tenfold ratio between the Lamé
parameters. We then increase these ratios to $20$ and $100$, by considering $(\mu,\mu_2) = (1,20)$ and $(\lambda,\lambda_2)=(2,40)$, $(\mu,\mu_2) = (1,100)$ and
$(\lambda,\lambda_2)=(2,200)$, respectively. For each set of parameters, the materials exhibit %
a huge contrast in their Young's modulus $E=2\mu(1+\nu)$, which measures the material's stiffness under tensile or compressive force. This implies that the
immersed body becomes increasingly stiffer than the surrounding medium. Hence, the higher the ratio, the
less the immersed body deforms under the same applied force.

As shown in Table~\ref{tab:modified_3D}, the inexact MAL preconditioner delivers excellent outer iteration counts, largely independent of mesh refinement and ratios between the Lamé constants, while keeping moderate inner iteration counts. This indicates that the mesh-independent behavior observed in the scalar two-dimensional tests persists for this three-dimensional vector-valued example. Note that the explicit favorable dependence on $h_{\Omega_2}$ in the three-dimensional lower bound of Remark~\ref{remark:3d_lower_bound} holds for the ideal AL preconditioner $\mathcal{P}_\gamma$ and does not directly extend to the modified variant. Additionally, these numerical tests do not explore the nearly incompressible limit $\nu\to 0.5$. The reported elasticity experiment is intended to demonstrate the applicability of the preconditioner to a three-dimensional vector-valued interface problem with fixed Lamé parameters. Robustness with respect to $\nu\to 0.5$ is therefore an interesting direction for future investigation.

A detailed visualization of the numerical solution $\boldsymbol{u}_h$, along with the
reference and deformed configuration for the immersed body when $(\mu,\mu_2) = (1,20)$ and $(\lambda,\lambda_2)=(2,40)$, is displayed in Figure~\ref{fig:3D_solution}.

\begin{table}[!ht]
  \centering
  \begin{subtable}{0.85\textwidth}
    \centering
    \footnotesize
    \begin{filecontents*}{data_modified_iterations_elasticity_diagonal.csv}
#DoF                     #moderate      #lessmoderate  #stiff        #N      #time   #avg inner
375  +  24 + 24                , 25 ,            26     ,   26          , 141     ,      0. ,  1
{2187  +  81 + 81}          , 19 ,            17     ,   17          , 783     ,      0. ,  11
{14,739  + 375 + 375  }      , 17 ,            16     ,   14          , 5181    ,      0. ,  11
{107,811  + 2,187 + 2,187  }      , 15 ,            14     ,   13          , 37395   ,      0. ,  11
{823,875  + 14,739 + 14,739 }     , 21 ,            17     ,   14          , 284451  ,      0. ,  12
\end{filecontents*}
\pgfplotstabletypeset[
    col sep=comma,
    string type,
    header=false,
    columns={0,1,2,3,6},
    columns/0/.style={column name={DoF $\bigl(|V_h|+|V_{2,h}|+|\Lambda_h|\bigr)$}, column type=l||},
    columns/1/.style={column name={$\mu_2/\mu\!=\!\lambda_2/\lambda\!=\!10$}, column type=c},
    columns/2/.style={column name={$\mu_2/\mu\!=\!\lambda_2/\lambda\!=\!20$}, column type=c},
    columns/3/.style={column name={$\mu_2/\mu\!=\!\lambda_2/\lambda\!=\!100$}, column type=c},
    columns/6/.style={column name=Inner, column type=c},
    every head row/.style={
    before row=\toprule
    \multicolumn{5}{c}{\textbf{FGMRES iteration counts with $\mathcal{\widetilde P}_{\gamma_1,\gamma_2}^{\text{diag}}$ }} \\ \midrule,
    after row=\midrule
    },
    every last row/.style={after row=\bottomrule}
    ]{data_modified_iterations_elasticity_diagonal.csv}
    \label{tab:elasticity_diag}
  \end{subtable}
  \caption{FGMRES iteration counts for the modified preconditioner $\mathcal{\widetilde P}_{\gamma_1,\gamma_2}^{\text{diag}}$ applied to the three-dimensional linear elasticity example for different ratios of Lamé parameters between subdomains. Last column reports the average number of inner preconditioned CG iterations needed to approximately invert the $(1,1)$-block when the ratio between Lamé constants is $20$.}
  \label{tab:modified_3D}
\end{table}

\begin{figure}[h!]
  \centering
  \begin{minipage}{0.3\textwidth}
    \includegraphics[width=0.9\linewidth]{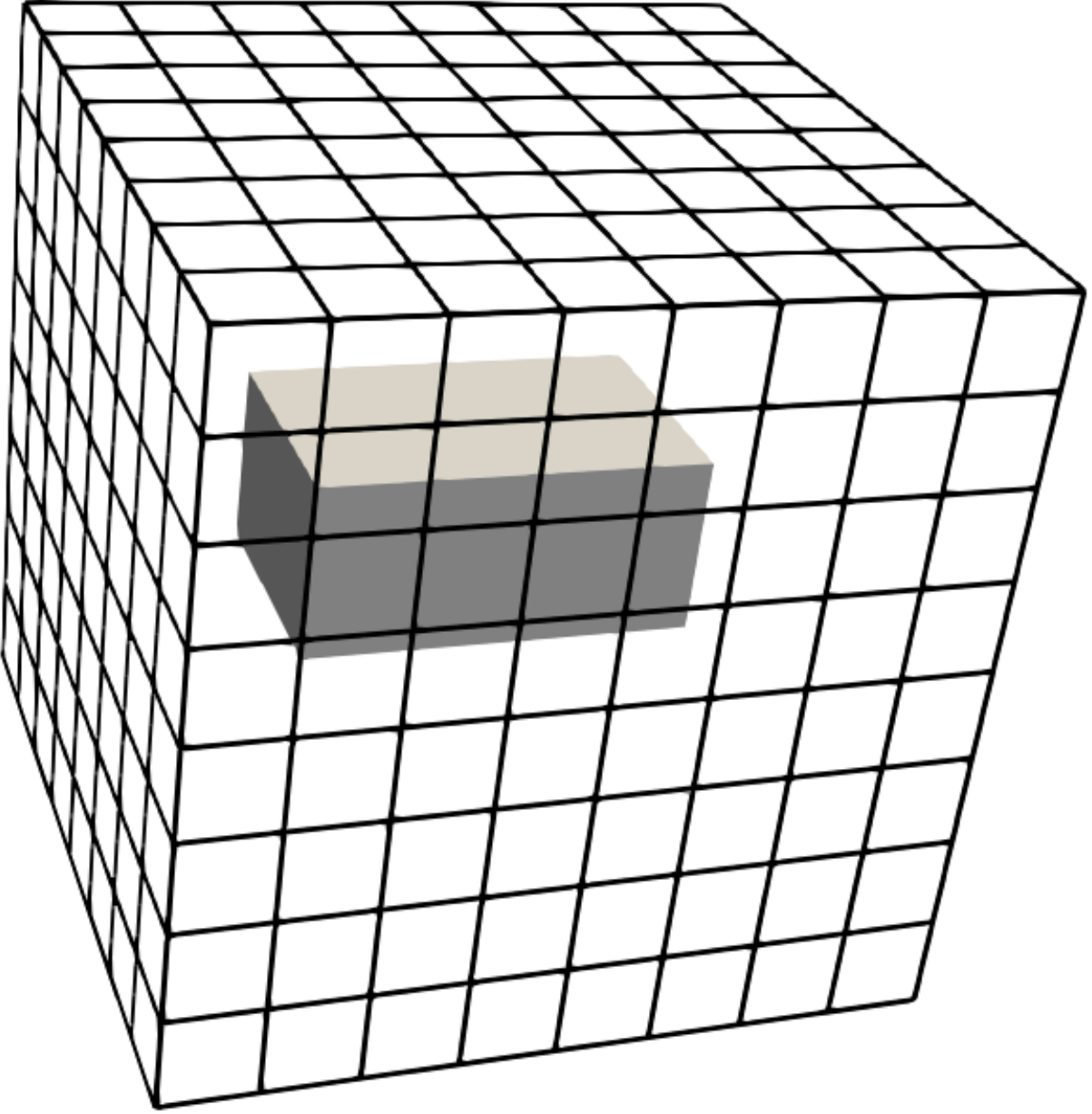}
  \end{minipage}\hspace{0.4cm}
  \begin{minipage}{0.3\textwidth}
    \centering
    \includegraphics[width=1\linewidth]{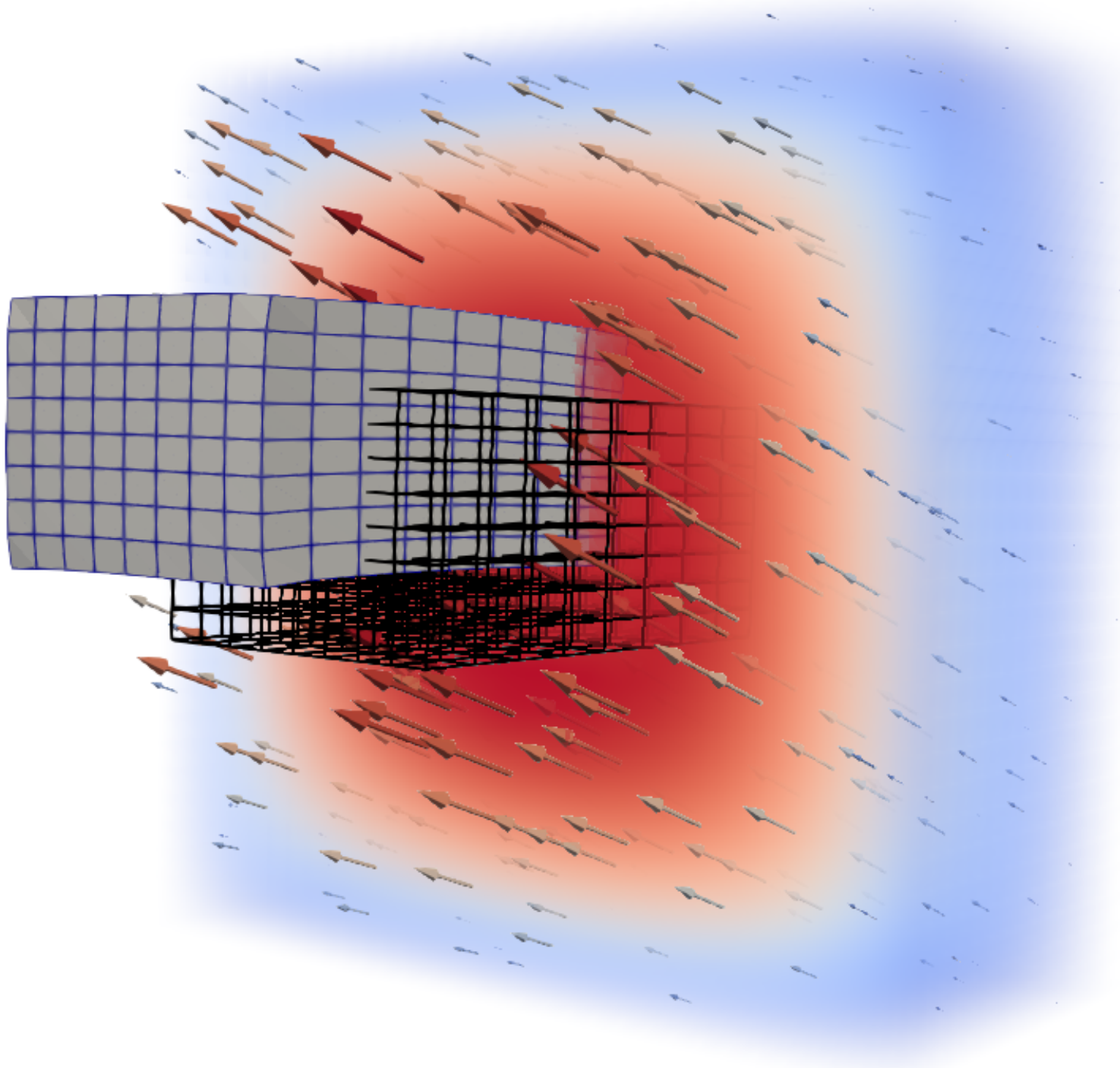}
  \end{minipage}\hspace{0.01cm}
  \begin{minipage}{0.15\textwidth}
    \centering
    \includegraphics{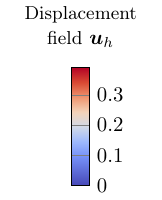}
  \end{minipage}

  \caption{Left: background domain (black lines) containing the immersed domain $\Omega_2$. Right: detail of the numerical solution
    obtained with $\mu_2/\mu = \lambda_2/\lambda = 20$. Arrows indicate the computed displacement field $\boldsymbol{u}_h$, while its magnitude is shown in the background domain using a color map with opacity. The black wireframe depicts the undeformed (reference) configuration of the immersed body, whereas the gray surface is the deformed one.}
  \label{fig:3D_solution}
\end{figure}

\FloatBarrier
\section{Conclusions}\label{sec:conclusions}
We have presented a novel augmented Lagrangian-based preconditioner to accelerate
the convergence of iterative solvers for saddle point systems arising from the finite element discretization
of elliptic interface problems using a fictitious domain approach. We have derived
ideal and modified variants of the preconditioner. The modified version, owing to its block-triangular structure, is particularly suitable for practical applications as it enables the use of off-the-shelf solvers from established high-performance libraries. A spectral analysis has been performed for both variants. We have shown, through
extensive numerical experiments with different configurations, that the proposed preconditioners are robust with respect to mesh sizes and jumps in
the coefficients. Ongoing work aims to extend the preconditioning strategy presented in this paper to fluid-structure interaction problems
and to develop tailored solvers for the efficient inversion of the augmented $(1,1)$-block.

\section*{Acknowledgments}
\noindent MB acknowledges partial support from grant MUR PRIN 2022 No. 20227PCCKZ (``Low Rank Structures and Numerical Methods in Matrix and Tensor Computations and their Applications``). LH and MF acknowledge
partial support from grant MUR PRIN 2022 No. 2022WKWZA8 (“Immersed methods for multiscale and multiphysics problems (IMMEDIATE)”), and the support of the European Research Council (ERC) under the European Union's Horizon 2020 research and innovation programme (call HORIZON-EUROHPC-JU-2023-COE-03, grant agreement No. 101172493 ``dealii-X''). LH acknowledges partial support by King Abdullah University of Science and Technology Research Funding (KRF) under Award No. ORFS-2025-CRG13-6911.3.
The authors are members of Gruppo Nazionale per il Calcolo Scientifico (GNCS) of Istituto Nazionale di Alta Matematica (INdAM).

\appendix

\section{}\label{sec:appendix_proofs}
In this appendix we give the proofs of the lemmas in Section~\ref{sec:spectral_modified}.

\subsection*{Proof of Lemma~\ref{lemma:prec_mat}}
Let
\[
  {\mathsf{A}}_{\gamma_1,\gamma_2} \coloneqq
  \begin{bmatrix}
    \mathsf{A} + \gamma_1 \mathsf{C^T W^{-1} C} & -\gamma_1 \mathsf{C^T W^{-1} C_2}                          \\[0.3em]
    \mathsf{-\gamma_2 C_2^T W^{-1}C}                     & \mathsf{A_2} + \gamma_2 \mathsf{C_2^T W^{-1} C_2}
  \end{bmatrix}
\]
and
\[
  \widetilde{\mathsf{A}}_{\gamma_1,\gamma_2} \coloneqq
  \begin{bmatrix}
    \mathsf{A} + \gamma_1 \mathsf{C^T W^{-1} C} & -\gamma_1 \mathsf{C^T W^{-1} C_2}                          \\[0.3em]
    0                                                    & \mathsf{A_2} + \gamma_2 \mathsf{C_2^T W^{-1} C_2}
  \end{bmatrix}.
\]
Using the factorization in~\eqref{eqn:fact}, we have

\begin{align*}
  \mathcal{A}_{\gamma_1, \gamma_2}\widetilde{\mathcal{P}}_{\gamma_1,\gamma_2}^{-1}
   & =
  \begin{bmatrix}
    \mathsf{A}_{\gamma_1,\gamma_2} & \mathsf{B^{T}} \\
    \mathsf{B}                     & 0
  \end{bmatrix}
  \begin{bmatrix}
    \mathsf{\widetilde{A}_{\gamma_1,\gamma_2}^{-1}} & 0               \\
    0                                               & \mathsf{I}_\ell
  \end{bmatrix}
  \begin{bmatrix}
    \mathsf{I}_{n+m} & \mathsf{B^{T}}   \\
    0                & -\mathsf{I}_\ell
  \end{bmatrix}
  \begin{bmatrix}
    \mathsf{I}_{n+m} & 0                          \\
    0                & -\mathsf{\widehat{S}^{-1}}
  \end{bmatrix} \label{eq:prec_mat} \\[6pt]
   & =
  \begin{bmatrix}
    \mathsf{A_{\gamma_1,\gamma_2}\widetilde{A}_{\gamma_1,\gamma_2}^{-1}}
     & -(\mathsf{A_{\gamma_1,\gamma_2}\widetilde{A}_{\gamma_1,\gamma_2}^{-1}} - \mathsf{I}_{n+m})\mathsf{B^{T}\widehat{S}^{-1}} \\
    \mathsf{B\widetilde{A}_{\gamma_1,\gamma_2}^{-1}}
     & -\mathsf{B\widetilde{A}_{\gamma_1,\gamma_2}^{-1}B^{T}\widehat{S}^{-1}}
  \end{bmatrix} \notag
\end{align*}

First, we observe that the inverse of $\mathsf{\widetilde{A}_{\gamma_1,\gamma_2}}$ is

\begin{equation*}
  \label{A.1}
  \mathsf{\widetilde{A}_{\gamma_1,\gamma_2}^{-1}} =
  \begin{bmatrix}
    \mathsf{A_{11}^{-1}} & -\mathsf{A_{11}^{-1}A_{12}A_{22}^{-1}} \\
    0                    & \mathsf{A_{22}^{-1}}
  \end{bmatrix}.
\end{equation*}

The $(1,1)$ block of $\mathcal{A}_{\gamma_1,\gamma_2}\widetilde{\mathcal{P}}_{\gamma_1, \gamma_2}^{-1}$ is

\[
  \mathsf{A_{\gamma_1,\gamma_2}\widetilde{A}_{\gamma_1,\gamma_2}^{-1}} =
  \begin{bmatrix}
    \mathsf{A_{11}} & \mathsf{A_{12}} \\
    \mathsf{A_{21}} & \mathsf{A_{22}}
  \end{bmatrix}
  \begin{bmatrix}
    \mathsf{A_{11}^{-1}} & -\mathsf{A_{11}^{-1}A_{12}A_{22}^{-1}} \\
    0                    & \mathsf{A_{22}^{-1}}
  \end{bmatrix}
  =
  \begin{bmatrix}
    \mathsf{I}_{n}             & 0                                                          \\
    \mathsf{A_{21}A_{11}^{-1}} & \mathsf{I}_m - \mathsf{A_{21}A_{11}^{-1}A_{12}A_{22}^{-1}}
  \end{bmatrix}.
\]

\vspace{0.2cm}
The $(1,2)$, $(2,1)$, and $(2,2)$ blocks are, respectively,

\begin{equation}\label{eq:block12}
  \begin{aligned}
    -(\mathsf{A_{\gamma_1,\gamma_2}\widetilde{A}_{\gamma_1,\gamma_2}^{-1}} - \mathsf{I}_{n+m}) \mathsf{B^{T}\widehat{S}^{-1}}
     & =-
    \begin{bmatrix}
      0                          & 0                                            \\
      \mathsf{A_{21}A_{11}^{-1}} & -\mathsf{A_{21}A_{11}^{-1}A_{12}A_{22}^{-1}}
    \end{bmatrix}
    \begin{bmatrix}
      \mathsf{C^{T}} \\ -\mathsf{C_{2}^{T}}
    \end{bmatrix}
    \mathsf{\widehat{S}^{-1}} \\[6pt]
     & =
    \begin{bmatrix}
      0 \\
      \mathsf{\bigl(-A_{21}A_{11}^{-1}C^{T} - A_{21}A_{11}^{-1}A_{12}A_{22}^{-1}C_{2}^{T}\bigr)\widehat{S}^{-1}}
    \end{bmatrix},
  \end{aligned}
\end{equation}

\vspace{0.6cm}

\begin{equation}\label{eq:block21}
  \begin{aligned}
    \mathsf{B \widetilde{A}_{\gamma_1,\gamma_2}^{-1}}
     & =
    \begin{bmatrix}
      \mathsf{C} & -\mathsf{C_2}
    \end{bmatrix}
    \begin{bmatrix}
      \mathsf{A_{11}^{-1}} & -\mathsf{A_{11}^{-1}A_{12}A_{22}^{-1}} \\
      0                    & \mathsf{A_{22}^{-1}}
    \end{bmatrix} \\[6pt]
     & =
    \begin{bmatrix}
      \mathsf{C A_{11}^{-1}} & -\mathsf{C A_{11}^{-1}A_{12}A_{22}^{-1}} - \mathsf{C_2 A_{22}^{-1}}
    \end{bmatrix},
  \end{aligned}
\end{equation}

\vspace{0.6cm}

\begin{equation}\label{eq:block22}
  \begin{aligned}
    -\mathsf{B \widetilde{A}_{\gamma_1,\gamma_2}^{-1} B^{T}\widehat{S}^{-1}}
     & =
    -
    \begin{bmatrix}
      \mathsf{C} & -\mathsf{C_2}
    \end{bmatrix}
    \begin{bmatrix}
      \mathsf{A_{11}^{-1}} & -\mathsf{A_{11}^{-1}A_{12}A_{22}^{-1}} \\
      0                    & \mathsf{A_{22}^{-1}}
    \end{bmatrix}
    \begin{bmatrix}
      \mathsf{C^{T}} \\
      -\mathsf{C_2^{T}}
    \end{bmatrix}
    \mathsf{\widehat{S}^{-1}} \\[6pt]
     & =
    -\Bigl(
    \mathsf{C A_{11}^{-1}C^{T}}
    + \mathsf{C A_{11}^{-1}A_{12}A_{22}^{-1}C_2^{T}}
    + \mathsf{C_2 A_{22}^{-1}C_2^{T}}
    \Bigr)\mathsf{\widehat{S}^{-1}}.
  \end{aligned}
\end{equation}

Plugging the above expressions into the expression for $\mathcal{A}_{\gamma_1,\gamma_2}\widetilde{\mathcal{P}}_{\gamma_1, \gamma_2}^{-1}$ gives~\eqref{eqn:apinv}. \qed
\vspace{0.4cm}
\subsection*{Proof of Lemma~\ref{lemma:SM}}
Applying the Sherman-Morrison-Woodbury matrix identity
\[
  (\mathsf{Y + U Z V})^{-1} = \mathsf{Y^{-1} - Y^{-1} U (Z^{-1} + V Y^{-1} U)^{-1} V Y^{-1}},
\]
to $\mathsf{A_{11}^{-1} = (A + \gamma_1 C^{T} W^{-1} C)^{-1}}$ gives
\[
  \mathsf{A_{11}^{-1}} = \mathsf{A^{-1} - A^{-1} C^{T} (\gamma_1^{-1} W + C A^{-1} C^{T})^{-1} C A^{-1}}.
\]
\noindent
Thus we have:

\begin{align*}
  \mathsf{C A_{11}^{-1} C^{T}}
   & = \mathsf{C}\Bigl(\mathsf{A^{-1}}
  - \mathsf{A^{-1} C^{T}}
    (\gamma_1^{-1}\mathsf{W} + \mathsf{C A^{-1}C^{T}})^{-1}
  \mathsf{C A^{-1}}\Bigr)\mathsf{C^{T}} \nonumber                                                               \\
   & = \mathsf{C A^{-1} C^{T}}
  \Bigl(\mathsf{I}_\ell - (\gamma_1^{-1}\mathsf{W} + \mathsf{C A^{-1}C^{T}})^{-1}
  \mathsf{C A^{-1} C^{T}}\Bigr) \nonumber                                                                       \\
   & = \mathsf{C A^{-1} C^{T}}
  \Bigl(\mathsf{I}_\ell - (\gamma_1^{-1}\mathsf{W} + \mathsf{C A^{-1}C^{T}})^{-1}
  (\mathsf{C A^{-1} C^{T}} + \gamma_1^{-1}\mathsf{W} - \gamma_1^{-1}\mathsf{W})\Bigr) \nonumber                 \\
   & = \mathsf{C A^{-1} C^{T}}
  \Bigl(\mathsf{I}_\ell-\mathsf{I}_\ell + (\gamma_1^{-1}\mathsf{W} + \mathsf{C A^{-1}C^{T}})^{-1}
  \gamma_1^{-1}\mathsf{W}\Bigr) \nonumber                                                                       \\
   & = \mathsf{C A^{-1} C^{T}}
  (\gamma_1^{-1}\mathsf{W} + \mathsf{C A^{-1}C^{T}})^{-1}\gamma_1^{-1}\mathsf{W} \nonumber                      \\
   & = (\mathsf{C A^{-1} C^{T}}+\gamma_1^{-1}\mathsf{W}-\gamma_1^{-1}\mathsf{W})
  (\gamma_1^{-1}\mathsf{W} + \mathsf{C A^{-1}C^{T}})^{-1}\gamma_1^{-1}\mathsf{W} \nonumber                      \\
   & = (\mathsf{I}_\ell - \gamma_1^{-1}\mathsf{W}
  (\gamma_1^{-1}\mathsf{W} + \mathsf{C A^{-1} C^{T}})^{-1})\gamma_1^{-1}\mathsf{W} \nonumber                    \\
   & = \gamma_1^{-1}\mathsf{W} - \gamma_1^{-1}\mathsf{W}(\gamma_1^{-1}\mathsf{W} + \mathsf{C A^{-1}C^{T}})^{-1}
  \gamma_1^{-1}\mathsf{W}. \nonumber                                                                            \\[1ex]
  \intertext{Multiplying by $\gamma_1 \mathsf{W^{-1}}$ gives}                                                   \\[0.5ex]
  \gamma_1 \mathsf{C A_{11}^{-1} C^{T} W^{-1}}
   & = \mathsf{I}_\ell - \gamma_1^{-1}\mathsf{W}
  (\gamma_1^{-1}\mathsf{W} + \mathsf{C A^{-1} C^{T}})^{-1} \nonumber                                            \\
   & = \mathsf{I}_\ell - \Bigl(
  (\gamma_1^{-1}\mathsf{W} + \mathsf{C A^{-1} C^{T}})
  (\gamma_1^{-1}\mathsf{W})^{-1}
  \Bigr)^{-1} \nonumber                                                                                         \\
   & = \mathsf{I}_\ell - \Bigl(\mathsf{I}_\ell +
  \gamma_1 \mathsf{C A^{-1} C^{T} W^{-1}}\Bigr)^{-1}.
\end{align*}

We complete the proof by observing that all necessary inverses exist.
In our system, $\mathsf{A}$ is symmetric positive definite. Therefore, since $\mathsf{C}$ has full row-rank, $\mathsf{C A^{-1} C^{T}}$ is also SPD, and $
  \mathsf{I}_\ell + \gamma_1 \mathsf{C A^{-1} C^{T} W^{-1}} $
is necessarily invertible. Note that in the Oseen problem analyzed in~\cite{spectral_analysis_modified}, $\mathsf{A_{2}}$ is invertible, therefore the same identity~\eqref{eqn:identity} is valid also for $\mathsf{A_{22}^{-1}}$. In our case, on the other hand, $\mathsf{A_{2}}$ is singular (cf. Remark~\ref{rmk:singularity}).

\vspace{0.1cm}
\subsection*{Proof of Lemma~\ref{lemma:dfge}} From Lemma~\ref{lemma:prec_mat} we know that
\[
  \mathcal{A}_{\gamma_1,\gamma_2} \mathcal{\widetilde P}_{\gamma_1,\gamma_2}^{-1} =
  \begin{bmatrix}
    \mathsf{I}_{n}                                             & 0 & 0                  \\
    \mathsf{A_{21}A_{11}^{-1}}                                 &
    \mathsf{I}_m - \mathsf{A_{21}A_{11}^{-1}A_{12}A_{22}^{-1}} &
    \bigl(-\mathsf{A_{21}A_{11}^{-1}C^{T}}
    - \mathsf{A_{21}A_{11}^{-1}A_{12}A_{22}^{-1}C_2^{T}}\bigr)\mathsf{\widehat{S}^{-1}} \\
    \mathsf{C A_{11}^{-1}}                                     &
    -\mathsf{C A_{11}^{-1}A_{12}A_{22}^{-1} - C_2 A_{22}^{-1}} &
    -\bigl(\mathsf{C A_{11}^{-1} C^{T}}
    + \mathsf{C A_{11}^{-1} A_{12} A_{22}^{-1} C_2^{T}}
    + \mathsf{C_2 A_{22}^{-1} C_2^{T}}\bigr)\mathsf{\widehat{S}^{-1}}
  \end{bmatrix}.
\]

\begin{itemize}
  \vspace{0.15cm}
  \item First, we consider the $(2,2)$-block.
        Using the definitions of $\mathsf{A_{21}}$ and $\mathsf{A_{12}}$, together with~\eqref{eqn:identity}, we obtain
        \begin{align*}
          \mathsf{I}_m - \mathsf{A_{21} A_{11}^{-1} A_{12} A_{22}^{-1}}
           & = \mathsf{I}_m - \mathsf{\bigl(-\gamma_{2} C_2^{T}W^{-1}C\bigr)A_{11}^{-1}\bigl(-\gamma_{1}C^{T}W^{-1}C_2\bigr)A_{22}^{-1}}      \\
           & = \mathsf{I}_m - \mathsf{\gamma_{2} C_2^{T}W^{-1}
          \bigl(\gamma_{1}C A_{11}^{-1}  C^{T} W^{-1}\bigr) C_2 A_{22}^{-1}}                                                                  \\
           & = \mathsf{I}_m - \mathsf{\gamma_{2} C_2^{T}W^{-1}
          \Bigl(}\mathsf{I}_\ell - \bigl(\mathsf{I}_\ell + \gamma_{1} \mathsf{C A^{-1} C^{T} W^{-1}}\bigr)^{-1}\Bigr)\mathsf{C_2 A_{22}^{-1}} \\
           & = \mathsf{I}_m - \mathsf{D E}
        \end{align*}
        with
        \[
          \mathsf{D = \gamma_{2} C_2^{T}W^{-1}\Bigl(}\mathsf{I}_\ell - (\mathsf{I}_\ell + \gamma_{1} \mathsf{C A^{-1} C^{T} W^{-1}})^{-1}\Bigr),
          \qquad
          \mathsf{E = C_2 A_{22}^{-1}}.
        \]

        \vspace{0.15cm}
  \item For the (2,3)-block, we have
        \begin{align*}
          \bigl(-\mathsf{A_{21} A_{11}^{-1} C^{T}} - \mathsf{A_{21} A_{11}^{-1} A_{12} A_{22}^{-1} C_2^{T}}\bigr)\mathsf{\widehat{S}^{-1}}
           & = \Bigl(\mathsf{\gamma_{2} C_2^{T} W^{-1} C A_{11}^{-1} C^{T}
          - \gamma_{2} C_2^{T}W^{-1} C A_{11}^{-1}\gamma_{1}C^{T}W^{-1} C_2 A_{22}^{-1} C_2^{T}}\Bigr)\mathsf{\widehat{S}^{-1}} \notag \\
           & = \mathsf{\gamma_{2} C_2^{T}W^{-1}C A_{11}^{-1}C^{T}}
          \Bigl(\mathsf{I}_\ell -\mathsf{ \gamma_{1} W^{-1} C_2 A_{22}^{-1} C_2^{T}\Bigr)}\mathsf{\widehat{S}^{-1}} \notag             \\
           & = \mathsf{\gamma_{2} C_2^{T}W^{-1}C A_{11}^{-1} C^{T} (\gamma_{1}W^{-1})
          \Bigl(\tfrac{1}{\gamma_{1}} W - C_2 A_{22}^{-1} C_2^{T}\Bigr)}\mathsf{\widehat{S}^{-1}} \notag                               \\
           & = \gamma_{2}\mathsf{C_2^{T} W^{-1} C A_{11}^{-1} C^{T} (\gamma_{1} W^{-1})}
          \Bigl(\mathsf{I}_\ell - \gamma_{1} \mathsf{C_2 A_{22}^{-1} C_2^{T} W^{-1}}\Bigr)
          \Bigl(\tfrac{1}{\gamma_{1}}\mathsf{W}\Bigr)\mathsf{\widehat{S}^{-1}} \notag                                                  \\
           & = \gamma_{2}\mathsf{C_2^{T} W^{-1}}
          \Bigl(\gamma_{1}\mathsf{C A_{11}^{-1} C^{T} W^{-1}}\Bigr)
          \Bigl(\mathsf{I}_\ell - \gamma_{1} \mathsf{C_2 A_{22}^{-1} C_2^{T} W^{-1}}\Bigr)
          \Bigl(\tfrac{1}{\gamma_{1}}\mathsf{W}\Bigr)\mathsf{\widehat{S}^{-1}} \notag                                                  \\
           & = \gamma_{2}\mathsf{C_2^{T} W^{-1}}
          \Bigl(\mathsf{I}_\ell-\bigl(\mathsf{I}_\ell+\gamma_{1}\mathsf{C A^{-1} C^{T} W^{-1}}\bigr)^{-1}\Bigr)
          \Bigl(\mathsf{I}_\ell - \gamma_{1} \mathsf{C_2 A_{22}^{-1} C_2^{T} W^{-1}}\Bigr)
          \Bigl(\tfrac{1}{\gamma_{1}}\mathsf{W}\Bigr)\mathsf{\widehat{S}^{-1}} \notag                                                  \\
           & = \mathsf{D G \Bigl(\tfrac{1}{\gamma_{1}}W\Bigr)\widehat{S}^{-1}}, \label{eq:final}
        \end{align*}
        where
        \[
          \mathsf{G = }\mathsf{I}_\ell -  \gamma_{1} \mathsf{C_2 A_{22}^{-1} C_2^{T} W^{-1}}.
        \]

        \vspace{0.15cm}
  \item Next, for the (3,2)-block we have
        \begin{align*}
          -\mathsf{C A_{11}^{-1} A_{12} A_{22}^{-1} - C_2 A_{22}^{-1}}
           & = \mathsf{C A_{11}^{-1} \gamma_{1} C^{T} W^{-1} C_2 A_{22}^{-1} - C_2 A_{22}^{-1}}                                                                     \\
           & = \Bigl(\mathsf{C A_{11}^{-1} \gamma_{1} C^{T} W^{-1}} - \mathsf{I}_\ell\Bigr)\mathsf{C_2 A_{22}^{-1}}                                                 \\
           & = \Bigl(\mathsf{I}_\ell - \Bigl(\mathsf{I}_\ell + \gamma_{1} \mathsf{C A^{-1} C^{T} W^{-1}}\Bigl)^{-1} - \mathsf{I}_\ell\Bigr)\mathsf{C_2 A_{22}^{-1}} \\
           & = -\Bigl(\mathsf{I}_\ell + \gamma_{1} \mathsf{C A^{-1} C^{T} W^{-1}}\Bigr)^{-1}\mathsf{C_2 A_{22}^{-1}}                                                \\
           & = -\mathsf{F E},
        \end{align*}
        with
        \[
          \mathsf{F = \Bigl(}\mathsf{I}_\ell + \gamma_{1} \mathsf{C A^{-1} C^{T} W^{-1}}\Bigr)^{-1}.
        \]

        \vspace{0.15cm}
  \item Finally, the (3,3)-block:
        \begin{align*}
          -\Bigl( & \mathsf{C A_{11}^{-1} C^{T} + C A_{11}^{-1} A_{12} A_{22}^{-1} C_2^{T} + C_2 A_{22}^{-1} C_2^{T}}\Bigr)\mathsf{\widehat{S}^{-1}} \nonumber               \\
                  & = -\Bigl(\mathsf{C A_{11}^{-1} C^{T} + C_2 A_{22}^{-1} C_2^{T} - C A_{11}^{-1}\gamma_{1} C^{T} W^{-1} C_2 A_{22}^{-1} C_2^{T}}\Bigr)
          (\mathsf{\gamma_{1}W^{-1}})(\mathsf{\gamma_{1}W^{-1}})^{-1}\mathsf{\widehat{S}^{-1}} \nonumber                                                                     \\
                  & = -\Bigl(\mathsf{\gamma_{1} C A_{11}^{-1} C^{T} W^{-1} + \gamma_{1} C_2 A_{22}^{-1} C_2^{T} W^{-1}
            - \gamma_{1} C A_{11}^{-1} \gamma_{1} C^{T} W^{-1} C_2 A_{22}^{-1} C_2^{T} W^{-1}}\Bigr)
          (\mathsf{\gamma_{1}W^{-1}})^{-1}\mathsf{\widehat{S}^{-1}} \nonumber                                                                                                \\
                  & = -\Bigl(\mathsf{I}_\ell - \bigl(\mathsf{I}_\ell - \gamma_{1} \mathsf{C A_{11}^{-1} C^{T} W^{-1}}\bigr)
          \bigl(\mathsf{I}_\ell - \gamma_{1} \mathsf{C_2 A_{22}^{-1} C_2^{T} W^{-1}}\bigr)\Bigr)
          (\gamma_{1}\mathsf{W^{-1}})^{-1}\mathsf{\widehat{S}^{-1}} \nonumber                                                                                                \\
                  & = -\Bigl(\mathsf{I}_\ell - \bigl(\mathsf{I}_\ell - \mathsf{I}_\ell + \bigl( \mathsf{I}_\ell + \gamma_{1} \mathsf{C A^{-1} C^{T} W^{-1}}\bigr)^{-1}\bigr)
          \bigl(\mathsf{I}_\ell - \gamma_{1} \mathsf{C_2 A_{22}^{-1} C_2^{T} W^{-1}}\bigr)\Bigr)
          (\gamma_{1}\mathsf{W^{-1}})^{-1}\mathsf{\widehat{S}^{-1}} \nonumber                                                                                                \\
                  & = -\Bigl(\mathsf{I}_\ell - \Bigl(\mathsf{I}_\ell + \gamma_{1} \mathsf{C A^{-1} C^{T} W^{-1}}\Bigr)^{-1}
          \bigl(\mathsf{I}_\ell - \gamma_{1} \mathsf{C_2 A_{22}^{-1} C_2^{T} W^{-1}}\bigr)\Bigr)
          (\gamma_{1}\mathsf{W^{-1}})^{-1}\mathsf{\widehat{S}^{-1}} \nonumber                                                                                                \\
                  & = (\mathsf{I}_\ell - \mathsf{F G})(-\gamma_{1}\mathsf{W^{-1}})^{-1}\mathsf{\widehat{S}^{-1}}.
        \end{align*}

\end{itemize}

\vspace{1cm}

\section{AMG Parameters}\label{sec:app:AMG_params}
The AMG parameter settings used in the numerical experiments are reported in Table~\ref{tab:amg_params}.
\renewcommand{\thetable}{B.\arabic{table}}
\setcounter{table}{0}
\begin{table}[h]
  \centering
  \renewcommand{\arraystretch}{1.2}
  \begin{tabular}{|l|l|}
    \hline
    \textbf{Parameter}    & \textbf{Value} \\
    \hline
    Smoother              & Chebyshev      \\
    Coarse solver         & Amesos-KLU     \\
    Smoother sweeps       & 2              \\
    V-cycle applications  & 1              \\
    Aggregation threshold & \(10^{-3}\)    \\
    Max size coarse level & 2000           \\
    \hline
  \end{tabular}
  \caption{Parameters for ML~\cite{Trilinos} (Trilinos 14.4.0).}
  \label{tab:amg_params}
\end{table}

\bibliographystyle{abbrv}
\bibliography{refs}

\end{document}